\documentclass[twoside]{article}
\usepackage[letterpaper, left=4cm, right=4cm, top=4cm, bottom=2.5cm]{geometry}
\usepackage{graphicx}
\usepackage{amsmath,amssymb,amsthm}
\usepackage{authblk}
\usepackage[utf8]{inputenc}
\usepackage{microtype}
\usepackage{mathrsfs}  
\usepackage{enumitem}
\usepackage{calc}
\usepackage{esint}
\usepackage{fancyhdr}
\usepackage{xcolor}
\usepackage[labelfont = bf]{caption}
\usepackage{doi}
\usepackage{subfigure}
\usepackage[numbers]{natbib}
\usepackage{float}
\usepackage{stmaryrd}
\usepackage{mathtools}

\newtheorem{theorem}{Theorem}[section]
\numberwithin{equation}{section}
\newtheorem{proposition}[theorem]{Proposition}

\newtheorem{corollary}[theorem]{Corollary}
\newtheorem{remark}[theorem]{Remark}

\newtheorem{algorithm}[theorem]{Algorithm}

\usepackage{titlesec}
\titleformat{\section}{\normalfont\scshape\centering}{\thesection.}{0.5em}{}
\titleformat*{\subsection}{\itshape}
\titleformat*{\subsubsection}{\itshape}

\providecommand{\keywords}[1]
{
	{\small\textit{Keywords:} #1}
}

\providecommand{\MSC}[1]
{
	{\small\textit{AMS MSC (2020):~~} #1}
}

\usepackage{hyperref}
\hypersetup{
	colorlinks=true,
	linkcolor=blue,
	citecolor = blue,
	filecolor=magenta,      
	urlcolor=cyan,
	pdfpagemode=FullScreen,
}

\providecommand{\jumptmp}[2]{#1\llbracket{#2}#1\rrbracket}
\providecommand{\jump}[1]{\jumptmp{}{#1}}

\AtBeginDocument{%
	\def\MR#1{}
}

\begin{document}
	\setlength{\abovedisplayskip}{5.5pt}
	\setlength{\belowdisplayskip}{5.5pt}
	\setlength{\abovedisplayshortskip}{5.5pt}
	\setlength{\belowdisplayshortskip}{5.5pt}

	\title{Explicit and efficient error estimation\\ for convex minimization problems}
	\author[1]{Sören Bartels\thanks{Email: \texttt{bartels@mathematik.uni-freiburg.de}}}
	\author[2]{Alex Kaltenbach\thanks{Email: \texttt{alex.kaltenbach@mathematik.uni-freiburg.de\vspace*{-4mm}}}}
	\date{\today}
	\affil[1]{\small{Institute of Applied Mathematics, Albert--Ludwigs--University Freiburg, Hermann--Herder--Stra\ss e~10, 79104 Freiburg}}
	\affil[2]{\small{Institute of Applied Mathematics, Albert--Ludwigs--University Freiburg, Ernst--Zermelo--Stra\ss e~1, 79104 Freiburg}}
	\maketitle

	\pagestyle{fancy}
	\fancyhf{}
	\fancyheadoffset{0cm}
	\addtolength{\headheight}{-0.25cm}
	\renewcommand{\headrulewidth}{0pt} 
	\renewcommand{\footrulewidth}{0pt}
	\fancyhead[CO]{\textsc{Explicit and efficient error estimation}}
	\fancyhead[CE]{\textsc{S. Bartels and A. Kaltenbach}}
	\fancyhead[R]{\thepage}
	\fancyfoot[R]{}
	
	\begin{abstract}
		We combine a  systematic approach
		for deriving general a posteriori error estimates for \hspace*{-0.15mm}convex \hspace*{-0.15mm}minimization \hspace*{-0.15mm}problems \hspace*{-0.15mm}based \hspace*{-0.15mm}on \hspace*{-0.15mm}convex \hspace*{-0.15mm}duality \hspace*{-0.15mm}relations \hspace*{-0.15mm}with~\hspace*{-0.15mm}a~\hspace*{-0.15mm}\mbox{recently}~\hspace*{-0.15mm}deri-ved \hspace*{-0.15mm}generalized \hspace*{-0.15mm}Marini \hspace*{-0.15mm}formula. \!The \hspace*{-0.15mm}a \hspace*{-0.15mm}posteriori \hspace*{-0.15mm}error \hspace*{-0.15mm}estimates~\hspace*{-0.15mm}are~\hspace*{-0.15mm}essentially~\hspace*{-0.15mm}constant-free \hspace*{-0.15mm}and \hspace*{-0.15mm}apply \hspace*{-0.15mm}to \hspace*{-0.15mm}a \hspace*{-0.15mm}large \hspace*{-0.15mm}class \hspace*{-0.15mm}of \hspace*{-0.15mm}variational \hspace*{-0.15mm}problems \hspace*{-0.15mm}including~\hspace*{-0.15mm}the~\hspace*{-0.15mm}\mbox{$p$--\hspace*{-0.15mm}Dirichlet}~\hspace*{-0.15mm}\mbox{problem},  as \hspace*{-0.15mm}well \hspace*{-0.15mm}as \hspace*{-0.15mm}degenerate \hspace*{-0.15mm}minimization, \hspace*{-0.15mm}obstacle \hspace*{-0.15mm}and \hspace*{-0.15mm}image \hspace*{-0.15mm}de-noising~\hspace*{-0.15mm}problems.~\hspace*{-0.15mm}In~\hspace*{-0.15mm}addition,  these a  posteriori error  estimates~are based on a comparison to a given non-conforming finite \hspace*{-0.15mm}element \hspace*{-0.15mm}solution. \hspace*{-0.15mm}For \hspace*{-0.15mm}the \hspace*{-0.15mm}$p$--\hspace*{-0.15mm}Dirichlet \hspace*{-0.15mm}problem,  \hspace*{-0.15mm}these  \hspace*{-0.15mm}a \hspace*{-0.15mm}posteriori~\hspace*{-0.15mm}error~\hspace*{-0.15mm}bounds~\hspace*{-0.15mm}are  equivalent to residual type a posteriori error bounds and, hence,~reliable and efficient.~~~~
	\end{abstract}

	
	\MSC{49M29, 65K15, 65N15, 65N50}

	\section{Introduction}\label{sec:intro}
	
	\subsection{Available estimates} \thispagestyle{empty}
	\qquad Various computable a posteriori error estimates have recently been derived for convex minimization \hspace*{-0.15mm}problems \hspace*{-0.15mm}such \hspace*{-0.15mm}as \hspace*{-0.15mm}the \hspace*{-0.15mm}$p$\hspace*{-0.15mm}--\hspace*{-0.3mm}Dirichlet \hspace*{-0.15mm}problem \hspace*{-0.15mm}or \hspace*{-0.15mm}degenerate~\hspace*{-0.15mm}\mbox{minimization}~\hspace*{-0.15mm}\mbox{problems}, cf. \cite{DK08,BDK12,LLC18,LC20,BM20}. These  a posteriori error estimates are typically defined~for~a particular finite \hspace*{-0.1mm}element \hspace*{-0.1mm}method~\hspace*{-0.1mm}and~\hspace*{-0.1mm}an~\hspace*{-0.1mm}appropriate~\hspace*{-0.1mm}discretization \hspace*{-0.1mm}of \hspace*{-0.1mm}the \hspace*{-0.1mm}problem~\hspace*{-0.1mm}\mbox{including}~\hspace*{-0.1mm}a~\hspace*{-0.1mm}suitable choice of quadrature. Further error sources, e.g., resulting from a stopping criterion of an iterative solution procedure are often not considered.~In~this~article,~we~\mbox{systematically}~\mbox{derive} a general a posteriori error estimate which applies to a large class of conforming~and~non-conforming numerical methods for non-linear and non-differentiable~problems, and provides a \hspace*{-0.1mm}computable \hspace*{-0.1mm}error \hspace*{-0.1mm}bound \hspace*{-0.1mm}that \hspace*{-0.1mm}is \hspace*{-0.1mm}independent \hspace*{-0.1mm}of \hspace*{-0.1mm}particular \hspace*{-0.1mm}discretizations~\hspace*{-0.1mm}or~\hspace*{-0.1mm}iteration~\hspace*{-0.1mm}\mbox{errors}. One important application of general estimates is the concept~of~model~hierarchies, e.g., determining the error of the numerical solution of a linearized model in solving a more complex non-linear problem. Our estimates avoid the use of discrete and continuous~Euler--Lagrange equations and only resort to first-order relations.~This~\mbox{approach}~combines~a~concept~used~in \cite{Rep08} with representations of a discrete dual solution obtained~via~post-processing~non-conforming approximations derived in, e.g., \cite{Mar85,Bar21}.
	
	\subsection{Error bounds via convex duality}
	
	\qquad Given a proper, convex and lower semi-continuous functional $\phi:\mathbb{R}^d\to \mathbb{R}\cup\{+\infty\}$~and~a (Lebesgue--)measurable functional $\psi : \Omega\times\mathbb{R}\to \mathbb{R}\cup\{+\infty\}$ such~that~for~almost~every~${x\in \Omega}$, where $\Omega\subseteq \mathbb{R}^d$, $d\in \mathbb{N}$, is a bounded polyhedral Lipschitz domain, the functional~$\psi(x,\cdot) : \mathbb{R}\to \mathbb{R}\cup\{+\infty\}$ is proper, convex
	and lower \mbox{semi-continuous},  we consider the minimization of the functional ${I:W^{1,p}_D(\Omega)\to \mathbb{R}\cup\{+\infty\}}$,~${p\in (1,\infty)}$,~for~every~$\smash{v\in W^{1,p}_D(\Omega)}$~defined~by 
	\begin{align}
		I(v)\vcentcolon=\int_{\Omega}{\phi(\nabla v)\,\textup{d}x}+\int_{\Omega}{\psi(\cdot, v)\,\textup{d}x},\label{intro:primal}
	\end{align}
	where $W^{1,p}_D(\Omega)$ denotes a suitable Sobolev space with a homogeneous Dirichlet boundary condition on a non-empty boundary part $\Gamma_D\subseteq \partial\Omega$ and may be replaced~by~$BV(\Omega)$,~i.e., the space \hspace*{-0.1mm}of \hspace*{-0.1mm}functions \hspace*{-0.1mm}of \hspace*{-0.1mm}bounded \hspace*{-0.1mm}variation. \hspace*{-0.1mm}Given \hspace*{-0.1mm}a
	 \hspace*{-0.1mm}possibly \hspace*{-0.1mm}non-unique~\hspace*{-0.1mm}minimizer~\hspace*{-0.1mm}${u\in W^{1,p}_D(\Omega)}$ and an arbitrary conforming approximation $
	\tilde{u}_h\! \in\!  \smash{W^{1,p}_D(\Omega)}$, which may~result~from~a~\mbox{post-pro-} cessing of a non-conforming, 
 discontinuous approximation, the convexity~\mbox{properties}~of~\eqref{intro:primal}, measured by an error functional $\rho_I^2:\smash{W^{1,p}_D(\Omega)}\times \smash{W^{1,p}_D(\Omega)}\to \mathbb{R}_{\ge 0}$,~lead~to the error estimate
	 \begin{align}
	 	\rho_I^2(\tilde{u}_h,u)\leq I(\tilde{u}_h)-I(u).\label{intro:strong_convexity}
	 \end{align}
 	Here, $\rho_I^2: \smash{W^{1,p}_D(\Omega)}\times \smash{W^{1,p}_D(\Omega)}\!\to\!  \mathbb{R}_{\ge 0}$ has the interpretation of a distance~measure~and~is,~e.g.,  (if existent) a lower bound for the second variation of the functional ${I\!:\!W^{1,p}_D(\Omega)\!\to\!  \mathbb{R}\hspace*{-0.1em}\cup\hspace*{-0.1em}\{+\infty\}}$. Also for degenerate problems a meaningful distance measure can be defined making use of a \hspace*{-0.1mm}co-coercivity \hspace*{-0.1mm}property. \hspace*{-0.1mm}To \hspace*{-0.1mm}get \hspace*{-0.1mm}a \hspace*{-0.1mm}computable \hspace*{-0.1mm}upper \hspace*{-0.1mm}bound \hspace*{-0.1mm}for \hspace*{-0.1mm}the \hspace*{-0.1mm}\mbox{approximation}~\hspace*{-0.1mm}\mbox{error}~\hspace*{-0.1mm}\eqref{intro:strong_convexity}, we \hspace*{-0.1mm}resort \hspace*{-0.1mm}to \hspace*{-0.1mm}the \hspace*{-0.1mm}(Fenchel) \hspace*{-0.1mm}dual \hspace*{-0.1mm}problem \hspace*{-0.1mm}to \hspace*{-0.1mm}the \hspace*{-0.1mm}minimization \hspace*{-0.1mm}of \hspace*{-0.1mm}\eqref{intro:primal}~\hspace*{-0.1mm}which,~\hspace*{-0.1mm}if,~\hspace*{-0.1mm}e.g.,~\hspace*{-0.1mm}${\phi\!\in\! C^0(\mathbb{R}^d)}$ and $\psi:\Omega\times \mathbb{R}\to \mathbb{R}$ is a Carath\'eodory mapping\footnote{A mapping $f:\Omega\times \mathbb{R}^k\to \mathbb{R}^l$, $k,l\in \mathbb{N}$, is said to be a \textit{Carath\'eodory mapping}, if $\psi(x,\cdot)\in C^0(\mathbb{R}^k;\mathbb{R}^l)$ for almost every $x\in \Omega$ and $\psi(\cdot,a):\Omega\subseteq\mathbb{R}^d\to \mathbb{R}^l$ is (Lebesgue--)measurable for all $a\in \mathbb{R}^k$.}, cf. \cite[p. 113 ff.]{ET99}, is given by~the~maximiza-tion of the functional $D:\smash{W^{p'}_N(\textup{div};\Omega)}\to \mathbb{R}\cup\{-\infty\}$, for every $y\in \smash{W^{p'}_N(\textup{div};\Omega)}$~defined~by 
 	\begin{align}
	 	D(y)\vcentcolon=-\int_{\Omega}{\phi^*(y)\,\textup{d}x}-\int_{\Omega}{\psi^*(\cdot,\textup{div}(y))\,\textup{d}x},\label{intro:dual}
 	\end{align}
 	where $W^{p'}_N(\textup{div};\Omega)$  consists of all vector fields in  $L^{p'}(\Omega;\mathbb{R}^d)$ whose distributional divergence exists \hspace*{-0.1mm}in \hspace*{-0.1mm}$\smash{L^{p'}\!(\Omega)}$ \hspace*{-0.1mm}and \hspace*{-0.1mm}whose \hspace*{-0.1mm}normal \hspace*{-0.1mm}component~\hspace*{-0.1mm}vanishes~\hspace*{-0.1mm}on~\hspace*{-0.1mm}${\Gamma_{\! N}\!\vcentcolon=\!\partial\Omega\!\setminus\!\Gamma_{\! D}}$.~\hspace*{-0.1mm}Apart~\hspace*{-0.1mm}from~\hspace*{-0.1mm}that,~\hspace*{-0.1mm}the functionals \hspace*{-0.1mm}$\phi^*:\mathbb{R}^d\!\to\! \mathbb{R}\cup\{+\infty\}$ \hspace*{-0.1mm}and \hspace*{-0.1mm}$\psi^*:\Omega\times \mathbb{R}\!\to\! \mathbb{R}\cup\{+\infty\}$~\hspace*{-0.1mm}\mbox{denote}~\hspace*{-0.1mm}the~\hspace*{-0.1mm}Fenchel~\hspace*{-0.1mm}\mbox{conjugates}~\hspace*{-0.1mm}to \hspace*{-0.1mm}$\phi:\mathbb{R}^d\!\to \!\mathbb{R}\cup\{+\infty\}$ and $\psi\!:\!\Omega\times \mathbb{R}\!\to\! \mathbb{R}\cup\{+\infty\}$ \hspace*{-0.1mm}(with \hspace*{-0.1mm}respect \hspace*{-0.1mm}to \hspace*{-0.1mm}the~\hspace*{-0.1mm}second~\hspace*{-0.1mm}argument), \hspace*{-0.1mm}resp. \hspace*{-0.1mm}A weak duality relation implies~${I(u)\ge D(\tilde{z}_h)}$ for all $\tilde{z}_h\in \smash{W^{p'}_N(\textup{div};\Omega)}$,~cf.~\mbox{\cite[Proposition~1.1]{ET99}}.
 	In~\hspace*{-0.15mm}particular, \hspace*{-0.15mm}for \hspace*{-0.15mm}every \hspace*{-0.15mm}conforming \hspace*{-0.15mm}approximations \hspace*{-0.15mm}$
 	\tilde{u}_h\! \in\!  W^{1,p}_D(\Omega)$~\hspace*{-0.15mm}and~\hspace*{-0.15mm}$\smash{\tilde{z}_h\!\in\! W^{p'}_N(\textup{div};\Omega)}$,~which both may result from~a post-processing of a non-conforming, 
 	discontinuous approximation, we obtain from an integration-by-parts, the general primal-dual a posteriori error estimate 
 	\begin{align}
 		\begin{aligned}
 		\rho_I^2(\tilde{u}_h,u)
 		&\leq \int_{\Omega}{\phi(\nabla \tilde{u}_h)-\nabla \tilde{u}_h\cdot \tilde{z}_h+\phi^*(\tilde{z}_h)\,\textup{d}x}\\&\quad +
 		\int_{\Omega}{\psi(\cdot, \tilde{u}_h)-\tilde{u}_h\,\textup{div}(\tilde{z}_h)+\psi^*(\cdot,\textup{div}(\tilde{z}_h))\,\textup{d}x}\\&=\vcentcolon\eta_h^2(\tilde{u}_h,\tilde{z}_h).
 	\end{aligned}\label{intro:general_estimate}
 	\end{align}
 	The bound \eqref{intro:general_estimate} and variants of it are well-known in the literature, cf. \cite{RX97,Rep99,Rep20C,Han05,Rep08,Bar15,BM20,Bar21}. The practical realizations of these bounds require the construction~of~an~appropriate -- ideally optimal -- $ \tilde{z}_h\in \smash{W^{p'}_N(\textup{div};\Omega)}$. Apparently, the estimate \eqref{intro:general_estimate} can only be efficient,~i.e., provide an optimal upper bound, if a strong duality relation applies, i.e., if we have that
 	\begin{align*}
 		\inf_{v\in W^{1,p}_D(\Omega)}{I(v)}=\sup_{y\in W^{p'}_N(\textup{div};\Omega)}{D(y)}.
 	\end{align*}
 
 	\subsection{Practical realization}\vspace*{-0.5mm}
 	
 	\qquad We \hspace*{-0.1mm}assume \hspace*{-0.1mm}here \hspace*{-0.1mm}that \hspace*{-0.1mm}for \hspace*{-0.1mm}all \hspace*{-0.1mm}$t\in \mathbb{R}$, $\psi(\cdot, t) = \psi_h(\cdot,t)$ \hspace*{-0.1mm}is  \hspace*{-0.1mm}element-wise \hspace*{-0.1mm}constant~\hspace*{-0.1mm}with~\hspace*{-0.1mm}\mbox{respect} to a triangulation $\mathcal{T}_h$, $h>0$, of $\Omega$ and choose an admissible vector~field~${\tilde{z}_h \in \smash{W^{p'}_N(\textup{div};\Omega)}}$ in the Raviart--Thomas finite element space $\mathcal{R}T^0_N(\mathcal{T}_h)\subseteq  \smash{W^{p'}_N(\textup{div};\Omega)}$,~cf.~\cite{RT75},~i.e., the space of element-wise affine vector fields that have continuous constant normal components~on element sides which vanish on $\Gamma_N$. Assume~that  $\tilde{u}_h\in W^{1,p}_D(\Omega)$~belongs~to~${\mathcal{S}^1_D(\mathcal{T}_h)\subseteq  W^{1,p}_D(\Omega)}$, i.e., the space of globally  continuous, element-wise affine functions that vanish on $\Gamma_D$.~Then, the primal-dual a posteriori error estimator defined~in~\eqref{intro:general_estimate}~can~be~re-written~as
 	\begin{align}
 		\begin{aligned}
 		\eta_h^2(\tilde{u}_h,\tilde{z}_h)&=
 		 \int_{\Omega}{\phi(\nabla \tilde{u}_h)-\nabla \tilde{u}_h\cdot \Pi_h\tilde{z}_h+\phi^*(\Pi_h\tilde{z}_h)\,\textup{d}x}\\[-0.5mm]&\quad +
 		\int_{\Omega}{\psi_h(\cdot,\Pi_h \tilde{u}_h)-\Pi_h\tilde{u}_h\,\textup{div}(\tilde{z}_h)+\psi^*_h(\cdot,\textup{div}(\tilde{z}_h))\,\textup{d}x}\\[-0.5mm]&\quad
 		+\int_{\Omega}{\psi_h(\cdot,\tilde{u}_h)-\psi_h(\cdot,\Pi_h \tilde{u}_h)\,\textup{d}x}+	\int_{\Omega}{\phi^*(\tilde{z}_h)-\phi^*(\Pi_h\tilde{z}_h)\,\textup{d}x},
 	\end{aligned}\label{intro:a_posteriori_2}
 	\end{align}
  	where $\Pi_h\!:\!L^1(\Omega;\mathbb{R}^l)\!\to\! \mathcal{L}^0(\mathcal{T}_h)^l$, $l\!\in\! \mathbb{N}$, denotes the $L^2$--projection~\mbox{operator}~onto~\mbox{element-wise} constant functions and vector fields, resp. Then, the integrands of the first two integrals on the right-hand side in \eqref{intro:a_posteriori_2} are element-wise constant and, by the \mbox{Fenchel--Young}~\mbox{inequality},
  	non-negative. In addition,
  	by Jensen's inequality, it holds ${\int_T{\psi_h(\cdot,\tilde{u}_h)-\psi_h(\cdot,\Pi_h \tilde{u}_h)\,\textup{d}x}\ge 0}$ and 	${\int_T{\phi^*(\tilde{z}_h)-\phi^*(\Pi_h\tilde{z}_h)\,\textup{d}x}\ge 0}$~for~all~${T\in \mathcal{T}_h}$. More generally, it can be bounded reliably using a trapezoidal quadrature rule leading to a fully practical contribution.\vspace*{-0.5mm}
  	
  	\subsection{Non-conforming representation}\vspace*{-0.5mm}
  	
  	\qquad A quasi-optimal discrete vector field $\tilde{z}_h\in \mathcal{R}T^0_N(\mathcal{T}_h)$ now is found via post-processing~a non-conforming, discontinuous Crouzeix--Raviart 
  	approximation of the primal~problem,~i.e., the minimization of $I_h^{\textit{cr}}:\mathcal{S}^{1,\textit{cr}}_D(\mathcal{T}_h)\to \mathbb{R}\cup\{+\infty\}$, for every ${v_h\in \mathcal{S}^{1,\textit{cr}}_D(\mathcal{T}_h)}$ defined by 
  	\begin{align}
  		I_h^{\textit{cr}}(v_h)\vcentcolon=\int_{\Omega}{\phi(\nabla_hv_h)\,\textup{d}x}+\int_{\Omega}{\psi_h(\cdot,\Pi_hv_h)\,\textup{d}x}.\label{intro:discrete_primal}
  	\end{align}
  	where $\mathcal{S}^{1,\textit{cr}}_D(\mathcal{T}_h)$ denotes the Crouzeix--Raviart finite element space, i.e,~the~space~of~element-wise affine functions that are continuous at the midpoints of element~sides~and~vanish~in~midpoints \hspace*{-0.1mm}(barycenters) \hspace*{-0.1mm}of \hspace*{-0.1mm}element \hspace*{-0.1mm}sides \hspace*{-0.1mm}belonging \hspace*{-0.1mm}to \hspace*{-0.1mm}$\Gamma_{\! D}$, \hspace*{-0.1mm}and \hspace*{-0.1mm}where \hspace*{-0.1mm}${\nabla_h\!:\!\mathcal{S}^{1,\textit{cr}}_D(\mathcal{T}_h)\!\to\! \mathcal{L}^0(\mathcal{T}_h)^d}$~\hspace*{-0.1mm}de-notes \hspace*{-0.1mm}the \hspace*{-0.1mm}element-wise \hspace*{-0.1mm}application \hspace*{-0.1mm}of \hspace*{-0.1mm}the \hspace*{-0.1mm}gradient \hspace*{-0.1mm}operator. \hspace*{-0.1mm}In~\hspace*{-0.1mm}\cite{CP20,Bar21}, \hspace*{-0.1mm}it \hspace*{-0.1mm}has \hspace*{-0.1mm}been~\hspace*{-0.1mm}shown~\hspace*{-0.1mm}that a \hspace*{-0.1mm}discrete \hspace*{-0.1mm}(Fenchel) \hspace*{-0.1mm}dual \hspace*{-0.1mm}problem \hspace*{-0.1mm}is \hspace*{-0.1mm}given \hspace*{-0.1mm}via \hspace*{-0.1mm}the \hspace*{-0.1mm}maximization \hspace*{-0.1mm}of 
  	\hspace*{-0.1mm}${D_h^{\textit{rt}}\!:\!\mathcal{R}T^0_N(\mathcal{T}_h)\!\to\!\mathbb{R}\hspace*{-0.1em}\cup\hspace*{-0.1em}\{\hspace*{-0.1em}-\infty\hspace*{-0.1em}\}}$, for~every~$\smash{y_h\in \mathcal{R}T^0_N(\mathcal{T}_h)}$ defined by
  	\begin{align}
  		D_h^{\textit{rt}}(y_h)\vcentcolon=-\int_{\Omega}{\phi^*(\Pi_hy_h)\,\textup{d}x}-\int_{\Omega}{\psi^*_h(\cdot,\textup{div}(y_h))\,\textup{d}x}.\label{intro:discrete_dual}
  	\end{align}
  	Then, \hspace*{-0.15mm}a \hspace*{-0.15mm}maximizer \hspace*{-0.15mm}$\smash{z_h^{\textit{rt}}\hspace*{-0.22em}\in\hspace*{-0.22em} \mathcal{R}T^0_N(\mathcal{T}_h)}$ \hspace*{-0.15mm}of \hspace*{-0.15mm}\eqref{intro:discrete_dual} \hspace*{-0.15mm}represents \hspace*{-0.15mm}a \hspace*{-0.15mm}quasi-optimal~\hspace*{-0.15mm}choice~\hspace*{-0.15mm}for~\hspace*{-0.15mm}$\smash{\tilde{z}_h\hspace*{-0.22em}\in\hspace*{-0.22em} \mathcal{R}T^0_N(\mathcal{T}_h)}$ in the a posteriori error estimate \eqref{intro:a_posteriori_2}.
  	However, owing to typical constraints such as, e.g., $-\textup{div}(z_h^{\textit{rt}}) = f_h$ in $\mathcal{L}^0(\mathcal{T}_h)$, 
  	this often requires solving a (potentially non-linear) saddle point problem. 
  	The latter
  	can, fortunately,  be avoided via post-processing~the~minimizer~of~\eqref{intro:discrete_primal}. More precisely,~if~${\phi\hspace*{-0.1em}\in\hspace*{-0.1em} C^1(\mathbb{R}^d)}$ and ${\psi_h(x,\cdot)\hspace*{-0.1em}\in\hspace*{-0.1em} C^1(\mathbb{R})}$ for almost every $x\hspace*{-0.1em}\in\hspace*{-0.1em} \Omega$, then~for~a~minimi-zer $u_h^{\textit{cr}}\hspace*{-0.1em}\in\hspace*{-0.1em}\mathcal{S}^{1,\textit{cr}}(\mathcal{T}_h)$   of \eqref{intro:discrete_primal}, we may represent 
  	$z_h^{\textit{rt}}\hspace*{-0.1em}\in\hspace*{-0.1em} \mathcal{R}T^0_N(\mathcal{T}_h)$~via~the~\mbox{reconstruction}~formula
  	\begin{align}
  		z_h^{\textit{rt}}=D\phi(\nabla_h u_h^{\textit{cr}})+D\psi_h(\cdot,\Pi_hu_h^{\textit{cr}})d^{-1}\big(\textup{id}_{\mathbb{R}^d}-\Pi_h\textup{id}_{\mathbb{R}^d}\big)\quad\textup{ in }\mathcal{R}T^0_N(\mathcal{T}_h),\label{intro:reconstruction_formula}
  	\end{align}
   derived,	e.g., in \cite[Proposition 3.1]{Bar21}. Even if $\phi:\smash{\mathbb{R}^d}\to \mathbb{R}\cup\{+\infty\}$ and $\psi_h:\Omega\times \mathbb{R}\!\to\! \mathbb{R}\cup\{+\infty\}$ are non-differentiable, it is sometimes  possible to derive  reconstruction formulas~similar~to \eqref{intro:reconstruction_formula} for quasi-optimal discrete vector fields $z_h^{\textit{rt}}\in \mathcal{R}T^0_N(\mathcal{T}_h)$, e.g., resorting to regularization arguments or given discrete Lagrange multipliers.\newpage
  
   \hspace*{-5.5mm}For the non-linear Dirichlet problem, i.e., if $\phi\in C^1(\mathbb{R}^d)$ and $\psi_h(x,t)\vcentcolon=-f_h(x)t$ for almost every $x\hspace*{-0.1em}\in\hspace*{-0.1em} \Omega$ and all $t\hspace*{-0.1em}\in\hspace*{-0.1em} \mathbb{R}$, where $f_h\hspace*{-0.1em}\vcentcolon=\hspace*{-0.1em}\Pi_h f\hspace*{-0.1em}\in\hspace*{-0.1em} \mathcal{L}^0(\mathcal{T}_h)$ and $f\hspace*{-0.1em}\in\hspace*{-0.1em} L^{p'}(\Omega)$, $p\hspace*{-0.1em}\in\hspace*{-0.1em} [1,\infty]$,~\eqref{intro:a_posteriori_2}~implies
   \begin{align}
   	\begin{aligned}
   			\rho_I^2(\tilde{u}_h,u)&\leq \eta_h^2(\tilde{u}_h,z_h^{\textit{rt}})\\&\leq \int_{\Omega}{\big(D\phi(\nabla \tilde{u}_h)-D\phi(\nabla_hu_h^{\textit{cr}})\big)\cdot(\nabla \tilde{u}_h-\nabla_hu_h^{\textit{cr}})\,\textup{d}x}\\&\quad+ \int_{\Omega}{
   				\big(D\phi^*(z_h^{\textit{rt}})-\phi^*(\Pi_hz_h^{\textit{rt}})\big)\cdot(z_h^{\textit{rt}}-\Pi_hz_h^{\textit{rt}})\,\textup{d}x},
   		\end{aligned}\label{intro:a_posteriori_3}
   \end{align}
	where \hspace*{-0.15mm}the \hspace*{-0.15mm}first \hspace*{-0.15mm}integral \hspace*{-0.15mm}on \hspace*{-0.15mm}the \hspace*{-0.15mm}right-hand \hspace*{-0.15mm}side \hspace*{-0.15mm}has \hspace*{-0.15mm}the \hspace*{-0.15mm}interpretation \hspace*{-0.15mm}of \hspace*{-0.15mm}a \hspace*{-0.15mm}residual,~\hspace*{-0.15mm}while~\hspace*{-0.15mm}the~\hspace*{-0.15mm}second \hspace*{-0.1mm}integral \hspace*{-0.1mm}contains  \hspace*{-0.1mm}data \hspace*{-0.1mm}approximation \hspace*{-0.1mm}errors.
	\hspace*{-0.5mm}The \hspace*{-0.1mm}estimate \hspace*{-0.1mm}\eqref{intro:a_posteriori_3} \hspace*{-0.1mm}is \hspace*{-0.1mm}less~\hspace*{-0.1mm}\mbox{accurate}~\hspace*{-0.1mm}than~\hspace*{-0.1mm}the estimate \eqref{intro:general_estimate} but turns out to be particularly useful for establishing~efficiency~properties. 
	In fact, for the $p$--Dirichlet problem, i.e., if $\phi\!\vcentcolon=\!\smash{\frac{1}{p}\vert \cdot\vert^p}\!\in\! C^1(\mathbb{R}^d)$,~${p\!\in\! (1,\infty)}$, resorting~to~\eqref{intro:a_posteriori_3}, we will find that the primal-dual a posteriori error~estimator~$\eta_h^2(u_h^{\textit{c}},z_h^{\textit{rt}})$, where $u_h^{\textit{c}}\in \mathcal{S}^1_D(\mathcal{T}_h)$ denotes the unique minimizer~of~${I_h^{\textit{c}}\vcentcolon= I|_{\smash{\mathcal{S}^1_D(\mathcal{T}_h)}}\hspace*{-0.1em}:\hspace*{-0.1em} \mathcal{S}^1_D(\mathcal{T}_h)\to \mathbb{R}}$,~is~\mbox{globally} equivalent to the classical residual type a posteriori error estimator $\smash{\eta_{\textit{res},h}^2(u_h^{\textit{c}})}$, cf.  \cite{DK08},~and,~therefore,~reliable, efficient and equivalent to the error quantity $\smash{\rho_I^2(u_h^{\textit{c}},u)}$ (if suitably chosen).~More~generally, using the triangle inequality in \eqref{intro:a_posteriori_3}, we see that the primal-dual a posteriori~error~estimator is estimated by approximation errors of conforming and non-conforming approximations.

    \subsection{New contributions}
   
   	\qquad The \hspace*{-0.15mm}bound \hspace*{-0.15mm}\eqref{intro:general_estimate}  \hspace*{-0.15mm}and \hspace*{-0.15mm}variants \hspace*{-0.15mm}thereof \hspace*{-0.15mm}are \hspace*{-0.15mm}well-known \hspace*{-0.15mm}in \hspace*{-0.15mm}literature,~\hspace*{-0.15mm}cf.~\hspace*{-0.15mm}\cite{RX97,Rep99,Rep20C,Rep08,BM20,Bar21}. 
   	Foremost, in \cite{Rep08}, S. I. Repin  proposed general a posteriori error estimates based on estimating the approximation error by the primal-dual gap. In particular, he pointed out that these error bounds require a quasi-optimal dual vector field $\tilde{z}_h\in \smash{W^{p'}_N(\textup{div};\Omega)}$~to~be~\mbox{practicable}. In the continuous case and if, e.g.,  $\phi\!\in\! C^1(\mathbb{R}^d)$ with
   	$\vert D\phi(t)\vert\!\leq \!c_0\vert t\vert^{p-1}+c_1$~for~all~${t\!\in \!\mathbb{R}^d}$, which is already well-understood, the reconstruction of such~a~\mbox{quasi-optimal} vector field is challenging as, e.g., a maximizer~${z\!\in\! \smash{W^{p'}_N(\textup{div};\Omega)}}$~of~\eqref{intro:dual}~needs~to~\mbox{satisfy}~the
   	optimality relation  $z=D\phi(\nabla u)$ \hspace*{-0.1mm}in \hspace*{-0.1mm}$L^{p'}(\Omega;\mathbb{R}^d)$~\hspace*{-0.1mm}and,~\hspace*{-0.1mm}thus,~\hspace*{-0.1mm}depends~\hspace*{-0.1mm}on~\hspace*{-0.1mm}a~\hspace*{-0.1mm}minimizer~\hspace*{-0.1mm}${u\in W^{1,p}_D(\Omega)}$
   	\hspace*{-0.1mm}of~\hspace*{-0.1mm}\eqref{intro:primal}.
   	For a discrete reconstruction, however, a discrete analogue of such an  optimality relation has yet been unavailable. In \cite{Bar21}, if, e.g., $\phi\in C^1(\mathbb{R}^d)$  and $\psi_h(x,\cdot)\in C^1(\mathbb{R})$~for~a.e.~${x\in\Omega}$~-- or already L. D. Marini in \cite{Mar85} in the~linear~case~-- via \eqref{intro:reconstruction_formula}, the~desired~discrete \mbox{analogue}, commonly referred to as generalized Marini formula, has recently~been~provided.~We~merge these results and extend them to non-differentiable~convex~minimization~problems, cf. \cite{BK22}.
   	The reconstruction formula \eqref{intro:reconstruction_formula} was also  referred to in \cite{LLC18,LC20} in the~case~of~the~non-linear Dirichlet problem. However, the different a posteriori error estimates derived~therein~rely~less on convex duality arguments such  as, e.g., \eqref{intro:general_estimate} or \eqref{intro:a_posteriori_3} do, but more on the reconstruction~of~an $\mathcal{S}^3_D(\mathcal{T}_h)$--conformal\footnote{Here, $\mathcal{S}^3_D(\mathcal{T}_h)$ denotes space of globally continuous and element-wise cubic functions that vanish~on~$\Gamma_D$.\vspace*{-12mm}} companion, which is computationally~cheap~but  rather indirect. Our approach is direct, general and incurs  an effort for the~computation~of~the~\mbox{right-hand}~side in \eqref{intro:general_estimate} comparable to the effort of the computation of the left-hand side in \eqref{intro:general_estimate}, i.e., to the computation~of~the primal approximation $\tilde{u}_h\in \mathcal{S}^1_D(\mathcal{T}_h)$. 
   	In the~case~of~the~\mbox{Poisson}~\mbox{problem}, in \cite{Brae07,BS08}, D.~Braess and J. Schöberl  equally resorted to convex duality arguments and the explicit reconstruction of a quasi-optimal vector field $\tilde{z}_h\in \smash{W^{p'}_N(\textup{div};\Omega)}$.~However,~their reconstruction technique, also called equilibration, cf. \cite{PraSyn47,AinOde00,LucWoh04,ErnVor15,AnjPau19,SmeVor20}, is based~on~local corrections on each patch, which equally is computationally cheap. 
   	 As~a~whole, we propose the combination of primal-dual a posteriori error estimates together with the reconstruction of quasi-optimal vector fields $\tilde{z}_h\in \smash{W^{p'}_N(\textup{div};\Omega)}$ based on reconstruction formulas like \eqref{intro:reconstruction_formula}  as a broadly applicable and usually computationally cheap alternative~to  residual type a posteriori error estimators. 
   Apart from that, numerical~experiments,~cf.~Section~\ref{sec:numerical_experiments}, justify the choice $\tilde{u}_h=I_hu_h^{\textit{cr}}\in\smash{ \mathcal{S}^1_D(\mathcal{T}_h)}$, where $I_h:\smash{ \mathcal{S}^{1,\textit{cr}}_D(\mathcal{T}_h)} \to \smash{\mathcal{S}^1_D(\mathcal{T}_h)}$ is a suitable~quasi-interpolation~operator, as conformal approximation of~$u\in \smash{W^{1,p}_D(\Omega)}$.~This,~in~turn, reduces the computational effort to the same level as,~e.g.,~for~residual~type~error~estimators.

   	\subsection{Outline}
   	
   	\qquad \textit{This article is organized as follows:} In Section~\ref{sec:preliminaries}, we introduce the employed~\mbox{notation}, define the relevant finite element spaces and give a brief review of continuous~and~discrete convex minimization problems. In particular, we prove discrete~convex~optimality~relations \hspace*{-0.15mm}under \hspace*{-0.15mm}minimal \hspace*{-0.15mm}regularity \hspace*{-0.15mm}assumptions \hspace*{-0.15mm}(cf. \hspace*{-0.15mm}\!Proposition \hspace*{-0.15mm}\ref{prop:optimality}). \hspace*{-0.15mm}\!In~\hspace*{-0.15mm}\mbox{Section}~\hspace*{-0.15mm}\ref{sec:general},~\hspace*{-0.15mm}we~\hspace*{-0.15mm}\mbox{discuss}~\hspace*{-0.15mm}a~\hspace*{-0.15mm}\mbox{general} \hspace*{-0.1mm}a \hspace*{-0.1mm}posteriori \hspace*{-0.1mm} error \hspace*{-0.1mm}estimate \hspace*{-0.1mm}and
   	\hspace*{-0.1mm}potential \hspace*{-0.1mm}error \hspace*{-0.1mm}sources \hspace*{-0.1mm}that~\hspace*{-0.1mm}may~\hspace*{-0.1mm}need~\hspace*{-0.1mm}to~\hspace*{-0.1mm}be~\hspace*{-0.1mm}taken~\hspace*{-0.1mm}into~\hspace*{-0.1mm}\mbox{account}.  In~\hspace*{-0.1mm}Section~\hspace*{-0.1mm}\ref{sec:post-processing}, \hspace*{-0.1mm}the \hspace*{-0.1mm}general \hspace*{-0.1mm}a
   	 \hspace*{-0.1mm}posteriori \hspace*{-0.1mm}error \hspace*{-0.1mm}estimate \hspace*{-0.1mm}is \hspace*{-0.1mm}refined \hspace*{-0.1mm}using \hspace*{-0.1mm}particular~\hspace*{-0.1mm}convex~\hspace*{-0.1mm}\mbox{duality} relations and a post-processing of a discrete
   	 convex minimization problem.~In~Section~\ref{sec:applications},~we apply these results to
well-known \hspace*{-0.1mm}convex \hspace*{-0.1mm}minimization \hspace*{-0.1mm}problems \hspace*{-0.1mm}including \hspace*{-0.1mm}the \hspace*{-0.1mm}\mbox{$p$\hspace*{-0.1mm}--\hspace*{-0.1mm}Dirichlet} \hspace*{-0.1mm}problem \hspace*{-0.1mm}and \hspace*{-0.1mm}an \hspace*{-0.1mm}optimal
   	 \hspace*{-0.1mm}design \hspace*{-0.1mm}problem, \hspace*{-0.1mm}a \hspace*{-0.1mm}prototypical \hspace*{-0.1mm}example \hspace*{-0.1mm}from \hspace*{-0.1mm}topology \hspace*{-0.1mm}optimization. For the $p$--Dirichlet problem, we establish a reliability and efficiency~result~(cf.~Theorem~\ref{thm:equivalences}).
   	 \hspace*{-0.1mm}In~\hspace*{-0.15mm}Section~\hspace*{-0.15mm}\ref{sec:numerical_experiments},~\hspace*{-0.15mm}we~\hspace*{-0.15mm}\mbox{confirm} \hspace*{-0.15mm}our \hspace*{-0.15mm}theoretical~\hspace*{-0.15mm}findings~\hspace*{-0.15mm}via~\hspace*{-0.15mm}numerical~\hspace*{-0.15mm}experiments.~\hspace*{-0.15mm}In~\hspace*{-0.1mm}Appendix~\hspace*{-0.1mm}\ref{subsec:convex_analysis}, we \hspace*{-0.15mm}collect  \hspace*{-0.15mm}definitions \hspace*{-0.15mm}and \hspace*{-0.15mm}results \hspace*{-0.15mm}from \hspace*{-0.15mm}convex~\hspace*{-0.15mm}analysis \hspace*{-0.15mm}needed \hspace*{-0.15mm}in \hspace*{-0.15mm}the \hspace*{-0.15mm}paper.~\hspace*{-0.15mm}Error~\hspace*{-0.15mm}estimates for the node-averaging operator in terms of shifted $N$--functions~are~proved~in~Appendix~\ref{subsec:auxiliary}.~~~~~
   	 
	\section{Preliminaries}\label{sec:preliminaries}

		\qquad Throughout the entire article, if not otherwise specified, we denote by ${\Omega\subseteq \mathbb{R}^d}$,~${d\in\mathbb{N}}$, a bounded polyhedral Lipschitz domain, whose topological boundary is disjointly divided~into a closed Dirichlet part $\Gamma_D$ and a Neumann part $\Gamma_N$, i.e.,  ${\partial\Omega=\Gamma_D\cup\Gamma_N}$~and~${\emptyset=\Gamma_D\cap\Gamma_N}$.\vspace*{-0.5mm}

	\subsection{Standard function spaces}\label{subsec:function_spaces}

	\qquad For $p\in \left[1,\infty\right]$ and $l\in \mathbb{N}$, we employ the standard notations\vspace*{-0.5mm}\footnote{Here, $W^{\smash{-\frac{1}{p},p}}(\Gamma_N)\vcentcolon=(W^{\smash{1-\frac{1}{p'},p'}}(\Gamma_N))^*$ and $W^{\smash{-\frac{1}{p},p}}(\partial\Omega)\vcentcolon=(W^{\smash{1-\frac{1}{p'},p'}}(\partial\Omega))^*$ .\vspace*{-10mm}}
	\begin{align*}
		\begin{aligned}
		W^{1,p}_D(\Omega;\mathbb{R}^l)&\vcentcolon=\big\{v\in L^p(\Omega;\mathbb{R}^l)&&\hspace*{-3.25mm}\mid \nabla v\in L^p(\Omega;\mathbb{R}^{l\times d}),\, \textup{tr}(v)=0\text{ in }L^p(\Gamma_D;\mathbb{R}^l)\big\},\\
		W^{p}_N(\textup{div};\Omega)&\vcentcolon=\big\{y\in L^p(\Omega;\mathbb{R}^d)&&\hspace*{-3.25mm}\mid \textup{div}(y)\in L^p(\Omega),\,\textup{tr}(y)\cdot n=0\text{ in }W^{-\frac{1}{p},p}(\Gamma_N)\big\},
	\end{aligned}
	\end{align*}
	$\smash{W^{1,p}(\Omega;\mathbb{R}^l)\vcentcolon=W^{1,p}_D(\Omega;\mathbb{R}^l)}$ if $\smash{\Gamma_D=\emptyset}$, and $\smash{W^{p}(\textup{div};\Omega)\vcentcolon=W^{p}_N(\textup{div};\Omega)}$ if $\smash{\Gamma_N=\emptyset}$,
	where~we denote \hspace*{-0.1mm}by \hspace*{-0.1mm}$\textup{tr}\hspace*{-0.1em}:\hspace*{-0.1em}\smash{W^{1,p}(\Omega;\mathbb{R}^l)}\hspace*{-0.12em}\to\hspace*{-0.12em}\smash{L^p(\partial\Omega;\mathbb{R}^l)}$ \hspace*{-0.1mm}and \hspace*{-0.1mm}by \hspace*{-0.1mm}$
	\textup{tr}(\cdot)\cdot n\hspace*{-0.1em}:\hspace*{-0.1em}\smash{W^p(\textup{div};\Omega)}\hspace*{-0.12em}\to\hspace*{-0.12em} \smash{W^{-\frac{1}{p},p}(\partial\Omega)}$,~\hspace*{-0.1mm}the~\hspace*{-0.1mm}trace and normal trace operator, resp. In particular, we  \mbox{predominantly}~\mbox{omit} $\textup{tr}(\cdot)$~in~this~context. In addition, we resort to the abbreviations~$L^p(\Omega) \hspace*{-0.1em}\vcentcolon=\hspace*{-0.1em} L^p(\Omega;\mathbb{R}^1)$,~${W^{1,p}(\Omega)\hspace*{-0.1em}\vcentcolon=\hspace*{-0.1em}W^{1,p}(\Omega;\mathbb{R}^1)}$~and $W^{1,p}_D(\Omega)\vcentcolon=W^{1,p}_D(\Omega;\mathbb{R}^1)$.\vspace*{-0.5mm}

	\subsection{Triangulations and standard finite element spaces}\label{subsec:finite_elements}
	
	\qquad In what follows, we will always denote by  $\mathcal{T}_h$, $h\hspace*{-0.1em}>\hspace*{-0.1em}0$,  a sequence of~regular,~i.e.,~\mbox{uniformly} shape regular and conforming, triangulations of $\Omega\subseteq \mathbb{R}^d$, $d\in\mathbb{N}$, cf. \cite{EG21}. The~sets~$\mathcal{S}_h$~and~$\mathcal{N}_h$  contain the sides and vertices, resp., of the elements of $\mathcal{T}_h$.	In the context of locally refined meshes, we employ the average mesh-size $\smash{h\vcentcolon=(\vert \Omega\vert/\textup{card}(\mathcal{N}_h))^{\frac{1}{d}}}>0$.~In~addition,~we~define $h_T\vcentcolon=\textup{diam}(T)$  for all $T\in \mathcal{T}_h$ and $h_S\vcentcolon=\textup{diam}(S)$  for all $S\in \mathcal{S}_h$.
	For \hspace*{-0.1mm}$k\hspace*{-0.17em}\in\hspace*{-0.17em} \mathbb{N}\cup\{0\}$~\hspace*{-0.1mm}and~\hspace*{-0.1mm}$T\hspace*{-0.17em}\in\hspace*{-0.17em} \mathcal{T}_h$, let $\mathcal{P}_k(T)$ \hspace*{-0.1mm}denote \hspace*{-0.1mm}the \hspace*{-0.1mm}set \hspace*{-0.1mm}of \hspace*{-0.1mm}polynomials \hspace*{-0.1mm}of  \hspace*{-0.1mm}maximal \hspace*{-0.1mm}degree~\hspace*{-0.1mm}$k$~\hspace*{-0.1mm}on~\hspace*{-0.1mm}$T$. Then, for~$k\hspace*{-0.17em}\in\hspace*{-0.17em} \mathbb{N}\cup\{0\}$~and $l\in \mathbb{N}$,  the sets of continuous and element-wise~polynomial~functions or vector~fields,~resp., are defined by\vspace*{-0.5mm}
	\begin{align*}
	\begin{aligned}
	\mathcal{S}^k(\mathcal{T}_h)^l&\vcentcolon=	\big\{v_h\in C^0(\overline{\Omega};\mathbb{R}^l)\hspace*{-3mm}&&\mid v_h|_T\in\mathcal{P}_k(T)^l\text{ for all }T\in \mathcal{T}_h\big\},\\
	\mathcal{L}^k(\mathcal{T}_h)^l&\vcentcolon=   \big\{v_h\in L^\infty(\Omega;\mathbb{R}^l)\hspace*{-3mm}&&\mid v_h|_T\in\mathcal{P}_k(T)^l\text{ for all }T\in \mathcal{T}_h\big\}.
	\end{aligned}
	\end{align*}
	The element-wise mesh-size function $h_\mathcal{T}\in \mathcal{L}^0(\mathcal{T}_h)$ is defined~by~${h_\mathcal{T}|_T\vcentcolon=h_T}$~for~all~${T\in \mathcal{T}_h}$.
 For every $T\in \mathcal{T}_h$ and $S\in \mathcal{S}_h$, we denote  by $\smash{x_T\vcentcolon=\frac{1}{d+1}\sum_{z\in \mathcal{N}_h\cap T}{z}}$ and $\smash{x_S\vcentcolon=\frac{1}{d}\sum_{z\in \mathcal{N}_h\cap S}{z}}$,  the midpoints (barycenters) of $T$ and $S$, resp. The $L^2$--projection~operator onto \mbox{element-wise} constant functions or vector fields, resp., is denoted by\vspace*{-0.5mm}
	\begin{align*}
		\smash{\Pi_h:L^1(\Omega;\mathbb{R}^l)\to \mathcal{L}^0(\mathcal{T}_h)^l.}
	\end{align*}
	For every $v_h\in \mathcal{L}^1(\mathcal{T}_h)^l$, it holds $\Pi_hv_h|_T=v_h(x_T)$ for every $T\in \mathcal{T}_h$. For~${p\in \left[1,\infty\right]}$,~there~exists a constant $c_{\Pi}\!>\!0$ such that for all $v\!\in\! L^p(\Omega;\mathbb{R}^l)$ and $T\!\in\! \mathcal{T}_h$,~cf.~\cite[Thm.~18.16]{EG21},~it~holds
	\begin{description}[noitemsep,topsep=1.5pt,font=\normalfont\itshape]
		\item[(L0.1)]\hypertarget{(L0.1)}{} $\|\Pi_h v\|_{L^p(T;\mathbb{R}^l)}\leq \| v\|_{L^p(T;\mathbb{R}^l)}$,
		\item[(L0.2)]\hypertarget{(L0.2)}{} $\|v-\Pi_h v\|_{L^p(T;\mathbb{R}^l)}\leq c_{\Pi}h_T \|\nabla v\|_{L^p(T;\mathbb{R}^{l\times d})}$ if $v\in W^{1,p}(\Omega;\mathbb{R}^l)$.
	\end{description}
	The node-averaging operator $\mathcal{J}_h^{\textit{av}}\hspace*{-0.15em}:\hspace*{-0.15em}\mathcal{L}^k(\mathcal{T}_h)^l\hspace*{-0.2em}\to\hspace*{-0.15em} \mathcal{S}^k_D(\mathcal{T}_h)^l$, where $ {\mathcal{S}^k_D(\mathcal{T}_h)^l\hspace*{-0.2em}\vcentcolon=\hspace*{-0.15em}\mathcal{S}^k(\mathcal{T}_h)^l\hspace*{-0.15em}\cap\hspace*{-0.15em} W^{1,1}_D(\Omega)^l}$, denoting~for~$z\in \mathcal{D}_h^k$, where $ \mathcal{D}_h^k$ denotes the set of degrees of freedom associated~with~$\mathcal{S}^k(\mathcal{T}_h)^l$, by $\mathcal{T}_h(z)\vcentcolon=\{T\in \mathcal{T}_h\mid z\in T\}$ the set of elements sharing~$z$, for all ${v_h\!\in\! \mathcal{L}^k(\mathcal{T}_h)^l}$,~is~defined~by
	\begin{align*}
		\mathcal{J}_h^{\textit{av}}v_h\vcentcolon=\sum_{z\in \smash{\mathcal{D}_h^k}}{\langle v_h\rangle_z\varphi_z},\qquad \langle v_h\rangle_z\vcentcolon=\begin{cases}
			\frac{1}{\textup{card}(\mathcal{T}_h(z))}\sum_{T\in \mathcal{T}_h(z)}{(v_h|_T)(z)}&\;\text{ if }z\in \Omega\cup \Gamma_N,\\
			0&\;\text{ if }z\in \Gamma_D,
		\end{cases}
	\end{align*}
	where we denote by $(\varphi_z)_{\smash{z\in \mathcal{D}_h^k}}\!\subseteq\! \mathcal{S}^k(\mathcal{T}_h)$,  the nodal basis of $\mathcal{S}^k(\mathcal{T}_h)$.~For~${p\!\in\! [1,\infty]}$,~there~\mbox{exists} a \hspace*{-0.1mm}constant \hspace*{-0.1mm}$c_{\textit{av}}\!>\!0$ \hspace*{-0.1mm}such \hspace*{-0.1mm}that \hspace*{-0.1mm}for \hspace*{-0.1mm}all \hspace*{-0.1mm}$v_h\!\in \!\mathcal{L}^k(\mathcal{T}_h)^l$, \hspace*{-0.1mm}$T\!\in\! \mathcal{T}_h$,   \hspace*{-0.1mm}$m\!\in\!\{0,\dots,k+1\}$,~\hspace*{-0.1mm}cf.~\hspace*{-0.1mm}\mbox{\cite[\!Lem.~\!22.12]{EG21}},\footnote{Here, for every $S\in\mathcal{S}_h\setminus\partial \Omega$, $\jump{v_h}_S\vcentcolon=v_h|_{T_+}-v_h|_{T_-}$ on $S$, where $T_+, T_-\in \mathcal{T}_h$ satisfy $\partial T_+\cap\partial  T_-=S$, and for every $S\in\mathcal{S}_h\cap\partial \Omega$, $\jump{v_h}_S\vcentcolon=v_h|_T$ on $S$, where $T\in \mathcal{T}_h$ satisfies $S\subseteq \partial T$.}
	\begin{description}[noitemsep,topsep=1.5pt,font=\normalfont\itshape]
		\item[(AV.1)]\hypertarget{(AV.1)}{} $\|\nabla_h^m(v_h-\mathcal{J}_h^{\textit{av}}v_h)\|_{L^p(T;\mathbb{R}^{l\times d^m})}\leq \smash{c_{\textit{av}}\sum_{S\in \mathcal{S}_h(T)}{\big\| h_S^{1/p-m}\jump{v_h}_S\big\|_{L^p(S;\mathbb{R}^l)}}}$,
		\item[(AV.2)]\hypertarget{(AV.2)}{} $\|\mathcal{J}_h^{\textit{av}}v_h\|_{L^p(T;\mathbb{R}^l)}\leq c_{\textit{av}}\|v_h\|_{L^p(\omega_T;\mathbb{R}^l)}$,
	\end{description}
	where $\mathcal{S}_h(T)\hspace*{-0.05em}\vcentcolon=\hspace*{-0.05em}\{S\hspace*{-0.05em}\in\hspace*{-0.05em} \mathcal{S}_h\mid S\cap \textup{int}(\omega_T)\hspace*{-0.05em}\neq\hspace*{-0.05em} \emptyset\}$ and ${\omega_T\hspace*{-0.05em}\vcentcolon=\hspace*{-0.05em}\bigcup\{T'\hspace*{-0.05em}\in\hspace*{-0.05em} \mathcal{T}_h\mid T'\cap T\hspace*{-0.05em}\neq\hspace*{-0.05em} \emptyset\}}$~for~all~${T\hspace*{-0.05em}\in\hspace*{-0.05em} \mathcal{T}_h}$ and 
	$\nabla_h^m\hspace*{-0.12em}:\hspace*{-0.12em}\mathcal{L}^k(\mathcal{T}_h)^l\hspace*{-0.12em}\to\hspace*{-0.12em} \mathcal{L}^{k-1}(\mathcal{T}_h)^{l\times d^m}\!$, for every $v_h\hspace*{-0.12em}\in\hspace*{-0.12em}\mathcal{L}^k(\mathcal{T}_h)^l$ defined~by~${(\nabla_h^mv_h)|_T\hspace*{-0.12em}\vcentcolon=\hspace*{-0.12em}\nabla^m(v_h|_T)}$ for all $\smash{T\in \mathcal{T}_h}$, denotes the element-wise $m$--th~gradient~operator.
	
	\subsection{Crouzeix--Raviart finite elements}\label{subsec:crouzeix_raviart}
	
	\qquad A particular instance of a larger class of non-conforming finite element~spaces,~introdu-ced \hspace*{-0.1mm}in \hspace*{-0.1mm}\cite{CR73},
	\hspace*{-0.1mm}is \hspace*{-0.1mm}the \hspace*{-0.1mm}Crouzeix--Raviart \hspace*{-0.1mm}finite \hspace*{-0.1mm}element \hspace*{-0.1mm}space, \hspace*{-0.1mm}which \hspace*{-0.1mm}consists~\hspace*{-0.1mm}of~\hspace*{-0.1mm}\mbox{element-wise}~\hspace*{-0.1mm}affine functions that are continuous at the midpoints (barycenters) of inner element sides, i.e.,~~~~~
	\begin{align*}
		\mathcal{S}^{1,\textit{cr}}(\mathcal{T}_h)\vcentcolon=\bigg\{v_h\in \mathcal{L}^1(\mathcal{T}_h)\,\Big|\, \int_{S}{\jump{v_h}_S\,\textup{d}s}=0\text{ for all }S\in \mathcal{S}_h\setminus\partial\Omega\bigg\}.
	\end{align*}
	Crouzeix--Raviart finite element functions that vanish at the midpoints of boundary~element sides that correspond to the  Dirichlet boundary $\Gamma_D$ are contained~in~the~space
	\begin{align*}
		\smash{	\mathcal{S}^{1,\textit{cr}}_D(\mathcal{T}_h)\vcentcolon=\big\{v_h\in\mathcal{S}^{1,\textit{cr}}(\mathcal{T}_h)\mid v_h(x_S)=0\text{ for all }S\in \mathcal{S}_h\cap \Gamma_D\big\}.}
	\end{align*}
	In particular, we have that $	\mathcal{S}^{1,\textit{cr}}_D(\mathcal{T}_h)=	\mathcal{S}^{1,\textit{cr}}(\mathcal{T}_h)$ if $\Gamma_D=\emptyset$.
	A basis of  $\mathcal{S}^{1,\textit{cr}}(\mathcal{T}_h)$~is~given~by functions $\varphi_S\!\in\! \mathcal{S}^{1,\textit{cr}}(\mathcal{T}_h)$, $S\in \mathcal{S}_h$, satisfying the Kronecker~property $\varphi_S(x_{S'})=\delta_{S,S'}$ for all $S,S'\in \mathcal{S}_h$. A basis of  $\mathcal{S}^{1,\textit{cr}}_D(\mathcal{T}_h)$~is~given~by~$(\varphi_S)_{S\in \mathcal{S}_h;S\not\subseteq\Gamma_D}$. 
	Since for every $v_h\in \mathcal{S}^{1,\textit{cr}}_D(\mathcal{T}_h)$, it holds $v_h=v_h(x_S)+\nabla_h v_h(\textup{id}_{\mathbb{R}^d}-x_S)$ in $T_+\cup T_-$ for all ${T_+,T_-\in \mathcal{T}_h}$~with~${T_+\cap T_-\!=\!S\!\in\!\mathcal{S}_h}$, we have that $\jump{v_h}_S=\jump{\nabla_h v_h}_S(\textup{id}_{\mathbb{R}^d}-x_S)$ in $S$ for all $S\in\mathcal{S}_h$. As an immediate consequence, also resorting to the discrete trace inequality\footnote{Appealing to \cite[Lemma 12.8]{EG21}, for $p\in [1,\infty]$ and $k\in \mathbb{N}\cup\{0\}$, there exists a constant~${c_{\textup{tr}}>0}$ such that for every $v_h\in \mathcal{L}^0(\mathcal{T}_h)$, it holds $h_T^{\smash{1/p}}\|v_h\|_{L^p(S)}\leq c_{\textup{tr}} \|v_h\|_{L^p(T)}$ for all ${T\in \mathcal{T}_h}$~and~${S\in \mathcal{S}_h}$~with~${S\subseteq \partial T}$.}, for fixed $p\in [1,\infty]$ and every $S\in \mathcal{S}_h$,~it~holds
	\begin{align}
		\begin{aligned}
\big\| h_S^{1/p}\jump{v_h}_S\big\|_{L^p(S)}&\leq \big\| h_S^{1/p+1}\jump{\nabla_h v_h}_S\big\|_{L^p(S;\mathbb{R}^d)}\\&\leq c_{\textup{tr}}\sum_{T\in \mathcal{T}_h;S\subseteq \partial T}{\big\|h_T\nabla_h v_h\big\|_{L^p(T;\mathbb{R}^d)}}
.
	\end{aligned}\label{eq:crouzeix_raviart_1}
	\end{align}
	A combination of \eqref{eq:crouzeix_raviart_1} and  \textit{(\hyperlink{(AV.1)}{AV.1})}  implies that for $p\!\in\! [1,\infty]$, there~exists~a~constant~${c_{\textit{av}}\!>\!0}$ such that for all $v_h\in \smash{\mathcal{S}^{1,\textit{cr}}_D(\mathcal{T}_h)}$, $T\in \mathcal{T}_h$ and $m\in \{0,1,2\}$, we have that
	\begin{description}[noitemsep,topsep=0.1pt,font=\normalfont\itshape]
		\item[(AV.3)]\hypertarget{(AV.3)}{} $\|\nabla_h^m (v_h-\mathcal{J}_h^{\textit{av}}v_h)\|_{L^p(T;\mathbb{R}^{d^m})}\leq c_{\textit{av}}\sum_{S\in \mathcal{S}_h(T)}{\big\| h_S^{1/p+1-m}\jump{\nabla_h v_h}_S\big\|_{L^p(S;\mathbb{R}^d)}}$,
		\item[(AV.4)]\hypertarget{(AV.4)}{} $\|\nabla_h^m (v_h- \mathcal{J}_h^{\textit{av}}v_h)\|_{L^p(T;\mathbb{R}^{d^m})}\leq c_{\textit{av}}\big\|h_{\mathcal{T}}^{1-m}\nabla_hv_h\big\|_{L^p(\omega_T;\mathbb{R}^d)}$.
	\end{description}
	
	\subsection{Raviart--Thomas finite elements}\label{subsec:raviart_thomas}

	\qquad The lowest order Raviart--Thomas finite element space, introduced~in~\cite{RT75},~\mbox{consists}~of~el-ement-wise affine vector fields that have continuous constant normal components on inner elements sides, i.e.,\footnote{Here, for every $S\in\mathcal{S}_h\setminus\partial \Omega$, $\jump{y_h\cdot n}_S\vcentcolon=\smash{y_h|_{T_+}\cdot n_{T_+}+y_h|_{T_-}\cdot n_{T_-}}$ on $S$, where $T_+, T_-\in \mathcal{T}_h$ satisfy $\smash{\partial T_+\cap\partial  T_-=S}$,  and for every $T\in \mathcal{T}_h$, $\smash{n_T:\partial T\to \mathbb{S}^{d-1}}$ denotes the outward unit normal vector field~to~$ T$, 
		and  for every $\smash{S\in\mathcal{S}_h\cap\partial \Omega}$, $\smash{\jump{y_h\cdot n}_S\vcentcolon=\smash{y_h|_T\cdot n}}$ on $S$, where $T\in \mathcal{T}_h$ satisfies $S\subseteq \partial T$ and $\smash{n:\partial\Omega\to \mathbb{S}^{d-1}}$ denotes the outward unit normal vector field to $\Omega$.\vspace*{-12.5mm}}
	\begin{align*}
\smash{\mathcal{R}T^0(\mathcal{T}_h)\vcentcolon=\big\{y_h\in \mathcal{L}^1(\mathcal{T}_h)^d\mid }&\,\smash{y_h|_T\cdot n_T=\textup{const}\text{ in }\partial T\text{ for  all }T\in \mathcal{T}_h,}\\ 
&\smash{	\jump{y_h\cdot n}_S=0\text{ on }S\text{ for all }S\in \mathcal{S}_h\setminus \partial\Omega\big\}.}
	\end{align*}
	Raviart--Thomas finite element functions that possess vanishing normal components~on~the Neumann boundary $\Gamma_N$ are contained in the space 
	\begin{align*}
		\smash{\mathcal{R}T^0_N(\mathcal{T}_h)\vcentcolon=\big\{y_h\in	\mathcal{R}T^0(\mathcal{T}_h)\mid y_h\cdot n=0\text{ on }\Gamma_N\big\}.}
	\end{align*}
	In particular, we have that $\mathcal{R}T^0_N(\mathcal{T}_h)=\mathcal{R}T^0(\mathcal{T}_h)$ if $\Gamma_N=\emptyset$. A basis of  $\mathcal{R}T^0(\mathcal{T}_h)$ is given by  vector fields $\psi_S\in \mathcal{R}T^0(\mathcal{T}_h)$, $S\in \mathcal{S}_h$, satisfying the Kronecker property $\psi_S|_{S'}\cdot n_{S'}=\delta_{S,S'}$ on $S'$ for all $S'\in \mathcal{S}_h$, where $n_S$ for all $S\in \mathcal{S}_h$ is the unit normal~vector~on~$S$~that~points from $T_-$ to $T_+$ if $T_+\cap T_-=S\in \mathcal{S}_h$. A~basis~of  $\mathcal{R}T^0_N(\mathcal{T}_h)$~is~given~by~$\psi_S\!\in\! \mathcal{R}T^0_N(\mathcal{T}_h)$,~${S\!\in\! \mathcal{S}_h\!\setminus\!\Gamma_N}$.

	\subsection{Integration-by-parts formula with respect to $\mathcal{S}^{1,\textit{cr}}(\mathcal{T}_h)$ and $ \mathcal{R}T^0(\mathcal{T}_h)$}\label{subsec:integration_by_parts}
		\qquad An \hspace*{-0.15mm}element-wise \hspace*{-0.15mm}integration-\hspace*{-0.15mm}by-\hspace*{-0.15mm}parts \hspace*{-0.15mm}implies \hspace*{-0.15mm}that \hspace*{-0.15mm}for \hspace*{-0.15mm}all \hspace*{-0.15mm}$v_h\hspace*{-0.2em}\in\hspace*{-0.2em} \mathcal{S}^{1,\textit{cr}}(\mathcal{T}_h)$ \hspace*{-0.15mm}and \hspace*{-0.15mm}${y_h\hspace*{-0.2em}\in\hspace*{-0.2em} \mathcal{R}T^0(\mathcal{T}_h)}$, we have the integration-by-parts
	formula
	\begin{align}
		\int_{\Omega}{ \nabla_hv_h\cdot\Pi_hy_h\,\textup{d} x}+\int_{\Omega}{ \Pi_hv_h\,\textup{div}(y_h)\,\textup{d} x}=\int_{\partial\Omega}{v_h\,y_h\cdot n\,\textup{d}s}.\label{eq:pi}
	\end{align}
	Here, we have  exploited that $y_h\in \mathcal{R}T^0(\mathcal{T}_h)$ has continuous constant  normal components~on inner element sides, i.e., 	$\jump{y_h\cdot n}_S=0$ on $S$ for every $S\in \mathcal{S}_h\setminus \partial\Omega$, and that the jumps~of ${v_h\in \mathcal{S}^{1,\textit{cr}}(\mathcal{T}_h)}$  across inner element sides have vanishing integral~mean,~i.e.,~${\int_{S}{\jump{v_h}_S\,\textup{d}s}=0}$ for every ${S\in \mathcal{S}_h\setminus\partial\Omega}$.
	In particular, for all $v_h\in \smash{\mathcal{S}^{1,\textit{cr}}_D(\mathcal{T}_h)}$~and~${y_h\in \mathcal{R}T^0_N(\mathcal{T}_h)}$,~\eqref{eq:pi}~reads~as 
	\begin{align}
		\int_{\Omega}{ \nabla_hv_h\cdot\Pi_h y_h\,\textup{d} x}=-\int_{\Omega}{ \Pi_hv_h\,\textup{div}(y_h)\,\textup{d} x}.\label{eq:pi0}
	\end{align}
	In \hspace*{-0.15mm}\cite{CP20,Bar21}, \hspace*{-0.15mm}the \hspace*{-0.15mm}integration-by-parts \hspace*{-0.15mm}formula \hspace*{-0.15mm}\eqref{eq:pi0} \hspace*{-0.15mm}formed \hspace*{-0.15mm}a \hspace*{-0.15mm}cornerstone \hspace*{-0.15mm}in \hspace*{-0.15mm}the~\hspace*{-0.15mm}derivation~\hspace*{-0.15mm}of~\hspace*{-0.15mm}dis-crete \hspace*{-0.15mm}convex \hspace*{-0.15mm}duality \hspace*{-0.15mm}relations \hspace*{-0.15mm}and, \hspace*{-0.15mm}as \hspace*{-0.15mm}such, \hspace*{-0.15mm}also \hspace*{-0.15mm}plays \hspace*{-0.15mm}a \hspace*{-0.15mm}central 
	 \hspace*{-0.15mm}role 
	 \hspace*{-0.15mm}in \hspace*{-0.15mm}the~\hspace*{-0.15mm}hereinafter~\hspace*{-0.15mm}analysis.
	
	\subsection{Convex minimization problems}\label{subsec:convex_min}
	
\qquad	Let $\phi:\mathbb{R}^d\to \mathbb{R}\cup\{+\infty\}$ be a proper, convex and lower semi-continuous~functional and~let $\psi:\Omega\times\mathbb{R}\to \mathbb{R}\cup\{+\infty\}$ be (Lebesgue--)measurable such that for almost every $x\in \Omega$, the function  $\psi(x,\!\cdot):\Omega\times\mathbb{R}\to \mathbb{R}\cup\{+\infty\}$ is proper,  convex and~lower~\mbox{semi-continuous}.~Then, for given $p\in \left(1,\infty\right)$, we examine  the convex minimization problem  that seeks for a function ${u\in W^{1,p}_D(\Omega)}$ that is minimal for $I:W^{1,p}_D(\Omega)\to \mathbb{R}\cup\{+\infty\}$, for every~${v\!\in\! \smash{W^{1,p}_D(\Omega)}}$~defined~by
	\begin{align}
		I(v)\vcentcolon=\int_{\Omega}{\phi(\nabla v)\,\textup{d}x}+\int_{\Omega}{\psi(\cdot,v)\,\textup{d}x}.\label{primal}
	\end{align}
	We will always assume that $\phi:\mathbb{R}^d\to \mathbb{R}\cup\{+\infty\}$ and $\psi:\Omega\times\mathbb{R}\to \mathbb{R}\cup\{+\infty\}$  are such~that \eqref{primal}  is proper, convex, weakly coercive,
	and lower semi-continuous, so that the direct me-thod in the calculus of variations implies the existence~of~a~\mbox{minimizer}~${u\in W^{1,p}_D(\Omega)}$~of~\eqref{primal}.
	A (Fenchel) dual  problem to \eqref{primal} is given by  the maximization of ${D\!:\!\smash{L^{p'}(\Omega;\mathbb{R}^d)}\!\to\! \mathbb{R}\hspace*{-0.1em}\cup\hspace*{-0.1em}\{\hspace*{-0.1em}-\infty\hspace*{-0.1em}\}}$, for every $y\in L^{p'}(\Omega;\mathbb{R}^d)$ defined by
	\begin{align}
		D(y)\vcentcolon=-\int_{\Omega}{\phi^*( y)\,\textup{d}x}-F^*(\textup{div}(y)),\label{dual}
	\end{align}
	where $\textup{div}: L^{p'}(\Omega;\mathbb{R}^d)\to (W^{1,p}_D(\Omega))^*$~for~every~$y\in  \smash{L^{p'}(\Omega;\mathbb{R}^d)}$ and $v\in \smash{W^{1,p}_D(\Omega)}$ is defined by ${\langle \textup{div}(y),v\rangle_{\smash{W^{1,p}_D(\Omega)}}\vcentcolon=-\int_{\Omega}{y\cdot\nabla v\,\textup{d}x}}$, and
	$\smash{F^*:L^{p'}(\Omega)\to \mathbb{R}\cup\{\pm\infty \}}$ denotes the Fenchel conjugate  to ${F:L^p(\Omega)\to \mathbb{R}\cup\{+\infty\}}$, defined by $F(v)\vcentcolon=\int_{\Omega}{\psi(\cdot,v)\,\textup{d}x}$~for~every~${v\in L^p(\Omega)}$.  Note that for every ${y\in \smash{W^{p'}_N(\textup{div};\Omega)}}$, we have the explicit representation
	\begin{align*}
	D(y)=-\int_{\Omega}{\phi^*( y)\,\textup{d}x}-\int_{\Omega}{\psi^*(\cdot,\textup{div}(y))\,\textup{d}x}.
	\end{align*}
	In general, cf. \cite[Proposition 1.1]{ET99}, we have  the weak duality relation
	\begin{align}
		I(u)=\inf_{v\in W^{1,p}_D(\Omega)}{I(v)}\ge \sup_{y\in L^{p'}(\Omega;\mathbb{R}^d)}{D(y)}.\label{weak_duality}
	\end{align}
	If, for instance, $\phi\in C^0(\mathbb{R}^d)$ and $\psi:\Omega\times \mathbb{R}\to \mathbb{R}$ is a Carath\'eodory mapping,  then, in \cite[p. 113 ff.]{ET99}, 
	it is shown that \eqref{dual} admits at least one maximizer $z\in \smash{W^{p'}_N(\textup{div};\Omega)}$, i.e., \eqref{dual} can be \hspace*{-0.1mm}restricted \hspace*{-0.1mm}to \hspace*{-0.1mm}\hspace*{-0.1mm}the \hspace*{-0.1mm}maximization \hspace*{-0.1mm}in \hspace*{-0.1mm}$\smash{W^{p'}_N(\textup{div};\Omega)}$ \hspace*{-0.1mm}and \hspace*{-0.1mm}strong \hspace*{-0.1mm}duality~\hspace*{-0.1mm}applies,~\hspace*{-0.1mm}i.e.,~\hspace*{-0.1mm}we~\hspace*{-0.1mm}have~\hspace*{-0.1mm}that
	\begin{align}
		\smash{I(u)= D(z)}.\label{strong_duality}
	\end{align}
	In addition, cf. \cite[Proposition 5.1]{ET99}, we then have the optimality relations
	\begin{align}
	\begin{aligned}
z\cdot\nabla u&=\phi^*(z)+\phi(\nabla u)&&\quad\textup{ a.e. in  }\Omega,\\
\textup{div}(z)\, u&=\psi^*(\cdot,\textup{div}(z))+\psi(\cdot, u)&&\quad\textup{ a.e. in  }\Omega.
		\end{aligned}	\label{optimality_relations}
	\end{align}
	If $\phi\in C^1(\mathbb{R}^d)$ and there exist $c_0,c_1\ge 0$ such that $\vert D\phi(t)\vert\leq c_0\vert t\vert^{p-1}+c_1$ for all $t\in\mathbb{R}^d$,~then, by the Fenchel--Young identity~(cf.~\eqref{eq:fenchel_young_id}), \eqref{optimality_relations}$_1$ is equivalent to 
	\begin{align}
	\smash{z= D\phi(\nabla u)\quad\textup{ in }L^{p'}(\Omega;\mathbb{R}^d).}
	\end{align}
	Similarly, if 
	 $\psi(x,\cdot)\in C^1(\mathbb{R})$ for almost every $x\in \Omega$ and there exist $c_0\ge 0$ and $c_1\in L^{p'}(\Omega)$ such that $\vert D\psi(x,t)\vert\leq c_0\vert t\vert^{p-1}+c_1(x)$ for almost every $x\in \Omega$ and all~$t\in \mathbb{R}$,~then~\eqref{optimality_relations}$_2$~is equivalent to 
	\begin{align}
	\smash{\textup{div}(z)=D\psi(\cdot, u)\quad\textup{ in }L^{p'}(\Omega).}
	\end{align}
	In addition, for the remainder of this article, we further assume that \eqref{primal} is co-coercive~at~a \hspace*{-0.1mm}minimizer \hspace*{-0.1mm}$\smash{u\!\in\! \smash{W^{1,p}_D(\Omega)}}$, \hspace*{-0.1mm}i.e., 
	\hspace*{-0.1mm}there~\hspace*{-0.1mm}\mbox{exists}~\hspace*{-0.1mm}a~\hspace*{-0.1mm}\mbox{functional} \hspace*{-0.1mm}${\rho_I^2\!:\!W^{1,p}_D(\Omega)\!\times\! W^{1,p}_D(\Omega)\!\to\! \mathbb{R}_{\ge 0}}$~\hspace*{-0.1mm}such~\hspace*{-0.1mm}that for every ${v\in \smash{W^{1,p}_D(\Omega)}}$,~it~holds
	\begin{align}
		\smash{\rho_I^2(v,u)\leq I(v)-I(u).}\label{strong_convexity}
	\end{align}
	Note that, in general, $\rho_I^2:\smash{W^{1,p}_D(\Omega)}\times \smash{W^{1,p}_D(\Omega)}\to \mathbb{R}_{\ge 0}$ does not need be definite~and,~\mbox{hence}, $u\hspace*{-0.1em}\in\hspace*{-0.1em} \smash{W^{1,p}_D(\Omega)}$ does   not need be unique. If, however, $\rho_I^2\hspace*{-0.1em}:\hspace*{-0.1em}\smash{W^{1,p}_D(\Omega)}\times \smash{W^{1,p}_D(\Omega)}\hspace*{-0.1em}\to\hspace*{-0.1em} \mathbb{R}_{\ge 0}$~is~\mbox{definite}, then $u\in \smash{W^{1,p}_D(\Omega)}$  is unique.
	
	\subsection{Discrete convex minimization problem}\label{subsec:discrete_convex_min}
	
	\qquad Let $\psi_h:\Omega\times\mathbb{R}\to \mathbb{R}\cup\{+\infty\}$ denote a suitable approximation of $\psi:\Omega\times\mathbb{R}\to \mathbb{R}\cup\{+\infty\}$ such \hspace*{-0.1mm}that \hspace*{-0.1mm}$\psi_h(\cdot,t)\!\in\! \mathcal{L}^0(\mathcal{T}_h)$ \hspace*{-0.1mm}for \hspace*{-0.1mm}all \hspace*{-0.1mm}$t\!\in\! \mathbb{R}$ \hspace*{-0.1mm}and \hspace*{-0.1mm}for \hspace*{-0.1mm}almost \hspace*{-0.1mm}every \hspace*{-0.1mm}$x\!\in \!\Omega$, \hspace*{-0.1mm}${\psi_h(x,\cdot)\!:\!\Omega\!\times\!\mathbb{R}\!\to\! \mathbb{R}\!\cup\!\{+\infty\}}$ is a proper, convex  and lower semi-continuous functional. Then, for given $p\in \left(1,\infty\right)$, we examine the (discrete) convex minimization problem  that seeks for a function  $u_h^{\textit{cr}}\in \mathcal{S}^{1,\textit{cr}}_D(\mathcal{T}_h)$ that is minimal for $I_h^{\textit{cr}}:\smash{\mathcal{S}^{1,\textit{cr}}_D(\mathcal{T}_h)}\to \mathbb{R}\cup\{+\infty\}$, for every $v_h\in \smash{\mathcal{S}^{1,\textit{cr}}_D(\mathcal{T}_h)}$ defined by
	\begin{align}
		I_h^{\textit{cr}}(v_h)\vcentcolon=\int_{\Omega}{\phi(\nabla_h v_h)\,\textup{d}x}+\int_{\Omega}{\psi_h(\cdot,\Pi_h v_h)\,\textup{d}x}.\label{discrete_primal}
	\end{align}
	Once again, we always assume that $\phi\hspace*{-0.1em}:\hspace*{-0.1em}\mathbb{R}^d\hspace*{-0.1em}\to\hspace*{-0.1em} \mathbb{R}\cup\{+\infty\}$ and ${\psi_h\hspace*{-0.1em}:\hspace*{-0.1em}\Omega\hspace*{-0.1em}\times\hspace*{-0.1em}\mathbb{R}\hspace*{-0.1em}\to\hspace*{-0.1em} \mathbb{R}\cup\{+\infty\}}$~are~such that \eqref{discrete_primal} admits a minimizer $u_h^{\textit{cr}}\in \mathcal{S}^{1,\textit{cr}}_D(\mathcal{T}_h)$.
	The replacement of $\psi\hspace*{-0.1em}:\hspace*{-0.1em}\Omega\times\mathbb{R}\hspace*{-0.1em}\to\hspace*{-0.1em} \mathbb{R}\cup\{+\infty\}$ by the approximation ${\psi_h:\Omega\times\mathbb{R}\to \mathbb{R}\cup\{+\infty\}}$  and inserting  the $L^2$--projection operator $\Pi_h\!:\!L^1(\Omega)\!\to\! \mathcal{L}^0(\mathcal{T}_h)$ in \eqref{discrete_primal} is~\hspace*{-0.1mm}crucial~\hspace*{-0.1mm}for~\hspace*{-0.1mm}the~\hspace*{-0.1mm}derivation \hspace*{-0.1mm}of \hspace*{-0.1mm}discrete \hspace*{-0.1mm}convex~\hspace*{-0.1mm}duality~\hspace*{-0.1mm}\mbox{relations}, as it leads to
	\begin{align*}
	\inf_{v_h\in \mathcal{S}^{1,\textit{cr}}_D(\mathcal{T}_h)}{I_h^{\textit{cr}}(v_h)}=\inf_{\overline{v}_h\in \Pi_h(\mathcal{S}^{1,\textit{cr}}_D(\mathcal{T}_h))}{\overline{I}_h^{\textit{cr}}(\overline{v}_h)},
	\end{align*}
	cf. \cite{CP20,Bar21}, where $\smash{\overline{I}_h^{\textit{cr}}\!:\!\Pi_h(\mathcal{S}^{1,\textit{cr}}_D(\mathcal{T}_h))\!\to\! \mathbb{R}\cup\{+\infty\}}$ for every $\smash{\overline{v}_h\!\in\! \Pi_h(\mathcal{S}^{1,\textit{cr}}_D(\mathcal{T}_h))}$~is~defined~by
	\begin{align}
		\overline{I}_h^{\textit{cr}}(\overline{v}_h)\vcentcolon=\inf_{\substack{v_h\in \mathcal{S}^{1,\textit{cr}}_D(\mathcal{T}_h)\\ \Pi_hv_h=\overline{v}_h\textup{ in }\mathcal{L}^0(\mathcal{T}_h)}}{\int_{\Omega}{\phi(\nabla_hv_h)\,\textup{d}x}}+\int_{\Omega}{\psi_h(\cdot,\overline{v}_h)\,\textup{d}x},\label{equivalent_discrete_primal}
	\end{align}
	i.e., the minimization of \eqref{discrete_primal} is equivalently expressible through the minimization~of~\eqref{equivalent_discrete_primal}. 
	This  motivates to examine \eqref{equivalent_discrete_primal} for its (Fenchel) dual problem via the Lagrange functional ${\overline{L}_h\hspace*{-0.1em}:\hspace*{-0.1em}\Pi_h(\mathcal{S}^{1,\textit{cr}}_D(\mathcal{T}_h))\hspace*{-0.1em}\times\hspace*{-0.1em} \mathcal{R}T^0_N(\mathcal{T}_h)\hspace*{-0.1em}\to\hspace*{-0.1em} \mathbb{R}\hspace*{-0.1em}\cup\hspace*{-0.1em}\{\pm\infty\}}$, for every $\smash{(\overline{v}_h,y_h)^\top\!\in\! \Pi_h(\mathcal{S}^{1,\textit{cr}}_D(\mathcal{T}_h))\!\times\! \mathcal{R}T^0_N(\mathcal{T}_h)}$ defined by
	\begin{align}
		\overline{L}_h(\overline{v}_h,y_h)\vcentcolon=-\int_{\Omega}{\textup{div}(y_h)\,\overline{v}_h\,\textup{d}x\,}-\int_{\Omega}{\phi^*(\Pi_hy_h)\,\textup{d}x}+\int_{\Omega}{\psi_h(\cdot,\overline{v}_h)\,\textup{d}x}.\label{discrete_lagrange_functional}
	\end{align}
	On \hspace*{-0.15mm}the \hspace*{-0.15mm}basis \hspace*{-0.15mm}of \hspace*{-0.15mm}the \hspace*{-0.15mm}Lagrange \hspace*{-0.15mm}functional \hspace*{-0.15mm}\eqref{discrete_lagrange_functional}, \hspace*{-0.15mm}in \hspace*{-0.15mm}\cite{Bar21,BW21}, \hspace*{-0.15mm}it \hspace*{-0.15mm}has \hspace*{-0.15mm}been \hspace*{-0.15mm}established~\hspace*{-0.15mm}that~\hspace*{-0.15mm}a~\hspace*{-0.15mm}(Fenchel) dual problem to the minimization of \eqref{discrete_primal} and \eqref{equivalent_discrete_primal}, resp., 
	is given by the~\mbox{maximization}~of $D_h^{\textit{rt}}:\mathcal{R}T^0_N(\mathcal{T}_h)\to \mathbb{R}\cup\{-\infty\}$, for every $y_h\in\mathcal{R}T^0_N(\mathcal{T}_h)$~defined~by
	\begin{align}
		D_h^{\textit{rt}}(y_h)\vcentcolon=-\int_{\Omega}{\phi^*(\Pi_h y_h)\,\textup{d}x}-\int_{\Omega}{\psi_h^*(\cdot,\textup{div}(y_h))\,\textup{d}x}.\label{discrete_dual}
	\end{align} 
	Appealing \hspace*{-0.1mm}to \hspace*{-0.1mm}\cite[\hspace*{-0.1mm}Proposition \hspace*{-0.1mm}3.1]{Bar21}  \hspace*{-0.1mm}or \hspace*{-0.1mm}\cite[\hspace*{-0.1mm}Corollary \hspace*{-0.1mm}3.6]{BW21}, \hspace*{-0.1mm}the~\hspace*{-0.1mm}discrete~\hspace*{-0.1mm}weak~\hspace*{-0.1mm}duality~\hspace*{-0.1mm}relation
	\begin{align*}
		\begin{aligned}
		\inf_{v_h\in \mathcal{S}^{1,\textit{cr}}_D(\mathcal{T}_h)}{I_h^{\textit{cr}}(v_h)}&=\inf_{\overline{v}_h\in \Pi_h(\mathcal{S}^{1,\textit{cr}}_D(\mathcal{T}_h))}{\overline{I}_h^{\textit{cr}}(\overline{v}_h)}
		\\[-0.5mm]&\ge \inf_{\overline{v}_h\in \Pi_h(\mathcal{S}^{1,\textit{cr}}_D(\mathcal{T}_h))}{\sup_{y_h\in \mathcal{R}T^0_N(\mathcal{T}_h)}{\overline{L}_h(\overline{v}_h,y_h) }}
		\\[-0.5mm]&\ge \sup_{y_h\in \mathcal{R}T^0_N(\mathcal{T}_h)}{\inf_{\overline{v}_h\in \Pi_h(\mathcal{S}^{1,\textit{cr}}_D(\mathcal{T}_h))}{\overline{L}_h(\overline{v}_h,y_h) }}
		\\[-0.5mm]&
		\ge \sup_{y_h\in \mathcal{R}T^0_N(\mathcal{T}_h)}{D_h^{\textit{rt}}(y_h)}.
	\end{aligned}
	\end{align*}
	holds. \hspace*{-0.5mm}If, \hspace*{-0.15mm}in \hspace*{-0.15mm}addition, \hspace*{-0.15mm}$\phi\hspace*{-0.2em}\in \hspace*{-0.2em} C^1(\mathbb{R}^d)$ \hspace*{-0.15mm}and \hspace*{-0.15mm}$\psi_h(x,\hspace*{-0.15mm}\cdot\hspace*{-0.15mm})\hspace*{-0.2em}\in\hspace*{-0.2em} C^1(\mathbb{R})$ \hspace*{-0.15mm}for \hspace*{-0.15mm}almost \hspace*{-0.15mm}every \hspace*{-0.15mm}$x\hspace*{-0.2em}\in\hspace*{-0.2em} \Omega$,~\hspace*{-0.15mm}then~\hspace*{-0.15mm}a~\hspace*{-0.15mm}\mbox{minimizer} $u_h^{\textit{cr}}\in \mathcal{S}^{1,\textit{cr}}_D(\mathcal{T}_h)$ of \eqref{discrete_primal} and a maximizer $z_h^{\textit{rt}}\in \mathcal{R}T^0_N(\mathcal{T}_h)$ of \eqref{discrete_dual}, cf. \cite[Corollary~3.7]{BW21}, are related by
	\begin{align}
		\begin{aligned}
		\Pi_h z_h^{\textit{rt}}&=D\phi(\nabla_h u_h^{\textit{cr}})&&\quad\text{ in }\mathcal{L}^0(\mathcal{T}_h)^d,\\
		\textup{div}(z_h^{\textit{rt}})&=D\psi_h(\cdot,\Pi_hu_h^{\textit{cr}})&&\quad\text{ in }\mathcal{L}^0(\mathcal{T}_h).\end{aligned}\label{eq:discrete_optimality_relations}
	\end{align}
	Apart from that, note that by the Fenchel--Young identity (cf.~\eqref{eq:fenchel_young_id}),~\eqref{eq:discrete_optimality_relations}~is~equivalent~to
	\begin{align}
			\begin{aligned}
			\Pi_h z_h^{\textit{rt}}\cdot \nabla_h u_h^{\textit{cr}}&=\phi^*(\Pi_hz_h^{\textit{rt}})+\phi(\nabla_h u_h^{\textit{cr}})&&\quad\text{ a.e. in }\Omega,\\
				\textup{div}(z_h^{\textit{rt}})\,\Pi_hu_h^{\textit{cr}}& =\psi_h^*(\cdot,	\textup{div}(z_h^{\textit{rt}}))+\psi_h(\cdot,\Pi_hu_h^{\textit{cr}})&&\quad\text{ a.e. in }\Omega.\end{aligned}\label{eq:discrete_optimality_relations2}
	\end{align}
	Eventually, in this case, we have the discrete reconstruction formula 
	\begin{align}
		z_h^{\textit{rt}}= D\phi(\nabla_h u_h^{\textit{cr}})+D\psi_h(\cdot, \Pi_hu_h^{\textit{cr}})d^{-1}\big(\textup{id}_{\mathbb{R}^d}-\Pi_h\textup{id}_{\mathbb{R}^d}\big)\quad\text{ in }\mathcal{R}T^0_N(\mathcal{T}_h),\label{eq:reconstruction_formula}
	\end{align}
	and discrete strong duality relation applies, i.e.,
	\begin{align}
			I_h^{\textit{cr}}(u_h^{\textit{cr}})=D_h^{\textit{rt}}(z_h^{\textit{rt}}).\label{eq:discrete_strong_duality}
	\end{align}

	\qquad More generally, without additional regularity assumptions on $\phi:\mathbb{R}^d\to \mathbb{R}\cup\{+\infty\}$ and $\psi_h:\Omega\times \mathbb{R}\to \mathbb{R}\cup\{+\infty\}$,
	the following discrete convex optimality relations apply:~~~~~~~~~~~
	
	\begin{proposition}\label{prop:optimality}
		Assume that $u_h^{\textit{cr}}\in  \mathcal{S}^{1,\textit{cr}}_D(\mathcal{T}_h)$ and $z_h^{\textit{rt}}\in \mathcal{R}T^0_N(\mathcal{T}_h)$ satisfy
		\begin{align}\int_{\Omega}{\phi(\nabla_hu_h^{\textit{cr}})\,\textup{d}x}&= \sup_{y_h\in \mathcal{R}T^0_N(\mathcal{T}_h)}{\int_{\Omega}{\nabla_hu_h^{\textit{cr}}\cdot\Pi_hy_h\,\textup{d}x}-\int_{\Omega}{\phi^*(\Pi_hy_h)\,\textup{d}x}},\label{prop:optimality1}\\
				\int_{\Omega}{\psi_h^*(\cdot,\textup{div}(z_h^{\textit{rt}}))\,\textup{d}x}&= \sup_{v_h\in  \mathcal{S}^{1,\textit{cr}}_D(\mathcal{T}_h)}{\int_{\Omega}{\textup{div}(z_h^{\textit{rt}})\,\Pi_hv_h\,\textup{d}x}-\int_{\Omega}{\psi_h(\cdot,\Pi_hv_h)\,\textup{d}x}},\label{prop:optimality2}
		\end{align}
		hold. Then, the following statements are equivalent:
		\begin{description}[noitemsep,topsep=1pt,labelwidth=\widthof{\textit{(ii)}},leftmargin=!,font=\normalfont\itshape]
			\item[(i)] \eqref{eq:discrete_strong_duality} holds.
			\item[(ii)] $(\Pi_hu_h^{\textit{cr}},z_h^{\textit{rt}})^\top\in  \Pi_h(\mathcal{S}^{1,\textit{cr}}_D(\mathcal{T}_h))\times\mathcal{R}T^0_N(\mathcal{T}_h) $ is a saddle point of \eqref{discrete_lagrange_functional}.
		\end{description}
		Moreover, if either of the cases \textit{(i)} or \textit{(ii)} applies, then the  optimality relations  \eqref{eq:discrete_optimality_relations2} hold.
	\end{proposition}

	\begin{proof}
		Using the assumptions \eqref{prop:optimality1} and \eqref{prop:optimality2}, and the integration-by-parts formula~\eqref{eq:pi0}, a direct calculation shows that
		\begin{align}
			\begin{aligned}
				I_h^{\textit{cr}}(u_h^{\textit{cr}})=\sup_{y_h\in \mathcal{R}T^0_N(\mathcal{T}_h)}{\overline{L}_h(\Pi_hu_h^{\textit{cr}},y_h) },\qquad
					D_h^{\textit{rt}}(z_h^{\textit{rt}})=\inf_{v_h\in \mathcal{S}^{1,\textit{cr}}_D(\mathcal{T}_h)}{\overline{L}_h(\Pi_hv_h,z_h^{\textit{rt}}) }.
			\end{aligned}\label{eq:optimality1}
		\end{align}
	
		\textit{ad (i) $\Rightarrow$ (ii).} Combining \eqref{eq:optimality1} and  \eqref{eq:discrete_strong_duality},~we~find~that
		\begin{align}\begin{aligned}
				I_h^{\textit{cr}}(u_h^{\textit{cr}})&=\sup_{y_h\in \mathcal{R}T^0_N(\mathcal{T}_h)}{\overline{L}_h(\Pi_hu_h^{\textit{cr}},y_h) }\ge \overline{L}_h(\Pi_hu_h^{\textit{cr}},z_h^{\textit{rt}})\\&\ge \inf_{v_h\in \mathcal{S}^{1,\textit{cr}}_D(\mathcal{T}_h)}{\overline{L}_h(\Pi_hv_h,z_h^{\textit{rt}}) }=D_h^{\textit{rt}}(z_h^{\textit{rt}})=	I_h^{\textit{cr}}(u_h^{\textit{cr}}).
			\end{aligned}\label{eq:optimality2}
		\end{align}
		As \hspace*{-0.1mm}a \hspace*{-0.1mm}result \hspace*{-0.1mm}of \hspace*{-0.1mm}\eqref{eq:optimality2}, \hspace*{-0.1mm}$(\Pi_hu_h^{\textit{cr}},z_h^{\textit{rt}})^\top\!\in\!  \Pi_h(\mathcal{S}^{1,\textit{cr}}_D(\mathcal{T}_h))\times\mathcal{R}T^0_N(\mathcal{T}_h) $ \hspace*{-0.1mm}is \hspace*{-0.1mm}a \hspace*{-0.1mm}saddle \hspace*{-0.1mm}point~\hspace*{-0.1mm}of~\hspace*{-0.1mm}\eqref{discrete_lagrange_functional},~\hspace*{-0.1mm}i.e., 
		\begin{align}
			\max_{y_h\in \mathcal{R}T^0_N(\mathcal{T}_h)}{\overline{L}_h(\Pi_hu_h^{\textit{cr}},y_h) }=\overline{L}_h(\Pi_hu_h^{\textit{cr}},z_h^{\textit{rt}})= \min_{v_h\in \mathcal{S}^{1,\textit{cr}}_D(\mathcal{T}_h)}{\overline{L}_h(\Pi_hv_h,z_h^{\textit{rt}}) }.\label{eq:optimality3}
		\end{align}
		
		\textit{ad (ii) $\Rightarrow$ (i).} If $(\Pi_hu_h^{\textit{cr}},z_h^{\textit{rt}})^\top\!\in  \Pi_h(\mathcal{S}^{1,\textit{cr}}_D(\mathcal{T}_h))\times\mathcal{R}T^0_N(\mathcal{T}_h) $ is a saddle point~of~\eqref{discrete_lagrange_functional},~i.e., \eqref{eq:optimality3} applies, then the infimum and supremum in \eqref{eq:optimality1} become a minimum and maximum, resp., so that from  \eqref{eq:optimality3}, it immediately follows that \eqref{eq:discrete_strong_duality} applies.
		
		\textit{Optimality relations.}
		From  \eqref{eq:optimality3} we deduce that $0\in (\partial_1 \overline{L}_h)(\Pi_hu_h^{\textit{cr}},z_h^{\textit{rt}})$, where the~sub-differential $\partial_1$ is taken in  $\Pi_h(\mathcal{S}^{1,\textit{cr}}_D(\mathcal{T}_h))$ equipped with $(\cdot,\cdot)_{L^2(\Omega)}$, and ${0\hspace*{-0.15em}\in\hspace*{-0.15em} (\partial_2 \overline{L}_h)(\Pi_hu_h^{\textit{cr}},z_h^{\textit{rt}})}$, 
		where the sub-differential $\partial_2$ is taken in $\mathcal{R}T^0_N(\mathcal{T}_h)$ equipped with $(\cdot,\cdot)_{L^2(\Omega;\mathbb{R}^d)}$.~Then,~by~the Fenchel--Young identity (cf.~\eqref{eq:fenchel_young_id}), $0\in \smash{(\partial_1 \overline{L}_h)(\Pi_hu_h^{\textit{cr}},z_h^{\textit{rt}})}$ is equivalent to
		\begin{align}
				\int_{\Omega}{\Pi_hz_h^{\textit{rt}}\cdot\nabla_hu_h^{\textit{cr}} \,\textup{d}x}=\int_{\Omega}{\phi^*(\Pi_hz_h^{\textit{rt}})+\phi(\nabla_hu_h^{\textit{cr}})  \,\textup{d}x },\label{eq:optimality4}
		\end{align}
		while $0\in (\partial_2 \overline{L}_h)(\Pi_hu_h^{\textit{cr}},z_h^{\textit{rt}})$ is equivalent to 
		 	\begin{align}
		 	\int_{\Omega}{\textup{div}(z_h^{\textit{rt}})\,\Pi_hu_h^{\textit{cr}}\,\textup{d}x}=	\int_{\Omega}{\psi^*_h(\cdot,\textup{div}(z_h^{\textit{rt}})) +\psi_h(\cdot,\Pi_hu_h^{\textit{cr}}) \,\textup{d}x }.\label{eq:optimality5}
		 \end{align}
	 	Eventually, since, by the Fenchel--Young inequality (cf.~\eqref{eq:fenchel_young_ineq}), we have that $\Pi_hz_h^{\textit{rt}}\cdot\nabla_hu_h^{\textit{cr}}\leq\phi^*(\Pi_hz_h^{\textit{rt}})+\phi(\nabla_hu_h^{\textit{cr}})$ and $\textup{div}(z_h^{\textit{rt}})\Pi_hu_h^{\textit{cr}}\leq \psi^*_h(\cdot,\textup{div}(z_h^{\textit{rt}}))+\psi_h(\cdot,\Pi_hu_h^{\textit{cr}})$ almost everywhere in $\Omega$, from \eqref{eq:optimality4} and \eqref{eq:optimality5}, we conclude that \eqref{eq:discrete_optimality_relations2} hold.
	\end{proof}

	\begin{remark}[Sufficient conditions for  \eqref{prop:optimality1} and \eqref{prop:optimality2}]\label{rem:sufficient_conditions}
		\begin{itemize}[noitemsep,topsep=2pt,leftmargin=!,labelwidth=\widthof{\textit{(ii)}},font=\itshape]
		\item[(i)] For \eqref{prop:optimality1} it is, e.g., sufficient that $\phi\in C^1(\mathbb{R}^d)$ (cf. \cite[Proposition 2.2]{Bar21}) or that there exist regularizations  $(\phi_\varepsilon)_{\varepsilon>0}\subseteq C^1(\mathbb{R}^d)$ such that $\phi_\varepsilon\to \phi$ and $\phi_\varepsilon^*\to \phi^*$ locally uniformly on their domains (cf. \cite[Remark 2.3]{Bar21}), e.g.,
		for $\phi\vcentcolon=\vert \cdot\vert\in C^0(\mathbb{R}^d)$, one can employ $(\phi_\varepsilon)_{\varepsilon>0}\subseteq C^1(\mathbb{R}^d)$, defined by $\phi_\varepsilon(t)\vcentcolon=\min\{\vert t\vert+\frac{\varepsilon}{2},\smash{\frac{\vert t\vert^2}{2\varepsilon}}\}$ for every $t\in \mathbb{R}^d$ and $\varepsilon>0$, for which~we~have~that $\phi_\varepsilon^*(t)=\varepsilon\smash{\frac{t^2}{2}}$ if $\vert t\vert \leq 1$ and $\phi_\varepsilon^*(t)=\infty$ else for every $\varepsilon>0$.
		
		\item[(ii)] For \eqref{prop:optimality2} it is, e.g., sufficient that $\psi_h(x,\cdot)\in C^1(\mathbb{R})$ for almost every $x\in \Omega$
		 or that $\mathcal{L}^0(\mathcal{T}_h)=\Pi_h(\mathcal{S}^{1,\textit{cr}}_D(\mathcal{T}_h))$, e.g., if $\Gamma_D\neq \partial\Omega$ (cf. \cite[Corollary 3.2]{BW21}).
		\end{itemize}
	\end{remark}

	\section{General a posteriori error estimation}\label{sec:general}\vspace*{-0.5mm}
	
	\qquad In this section, we derive general a posteriori error~estimates~for~\mbox{convex},~possibly~non-differentiable, minimization problems such as in Subsection \ref{subsec:convex_min}.~These~\mbox{error} estimates are well-known in the literature, cf. \cite{RX97,Rep99,Rep20C,Han05,Rep08,Bar15,BM20,Bar21}, and form the fundament for any a posteriori error analysis on the basis of convex duality relations. As~its~proof~is~propor-tionate simple, 
	for the benefit of the reader, we want to briefly reproduce~it~here.~Moreover, similar to \cite{HPSV21}, we want to point out potential error sources that may need~to~be~taken~into account and discuss various practical aspects of the concept.
	
	\begin{proposition}\label{prop:general_pd_estimator0}
		Let $p\hspace*{-0.1em}\in\hspace*{-0.1em} [1,\infty)$ and $\psi(x,t)\hspace*{-0.1em}=\hspace*{-0.1em}\psi_h(x,t)$ for almost every $x\hspace*{-0.1em}\in\hspace*{-0.1em} \Omega$~and~all~${t\hspace*{-0.1em}\in\hspace*{-0.1em} \mathbb{R}}$. 
		 Then, for every $\tilde{u}_h\in \smash{W^{1,p}_D(\Omega)}$ and $\tilde{z}_h\in \smash{W^{p'}_N(\textup{div};\Omega)}$, we have that\vspace*{-0.5mm}
		\begin{align}
			\begin{aligned}
				\rho^2_I(\tilde{u}_h,u)&\leq\int_{\Omega}{\phi(\nabla \tilde{u}_h)- \tilde{z}_h\cdot\nabla \tilde{u}_h+\phi^*(\tilde{z}_h)\,\textup{d}x}\\&\quad+
				\int_{\Omega}{\psi_h(\cdot,\tilde{u}_h)- \textup{div}(\tilde{z}_h)\, \tilde{u}_h+\psi_h^*(\cdot,\textup{div}(\tilde{z}_h))\,\textup{d}x}=\vcentcolon\eta^2_h(\tilde{u}_h,\tilde{z}_h).\end{aligned}\label{eq:general_pd_estimator0}
		\end{align}
	\end{proposition}

		\begin{proof}
	\hspace*{-1mm}Let  \hspace*{-0.15mm}$\tilde{u}_h\hspace*{-0.2em}\in\hspace*{-0.2em} \smash{W^{1,p}_D\hspace*{-0.05em}(\Omega)}$  \hspace*{-0.15mm}and  \hspace*{-0.15mm}$\tilde{z}_h\hspace*{-0.2em}\in\hspace*{-0.2em} \smash{W^{p'}_N\hspace*{-0.05em}(\textup{div};\hspace*{-0.05em}\Omega)}$  \hspace*{-0.15mm}be  \hspace*{-0.15mm}fixed,  \hspace*{-0.15mm}but  \hspace*{-0.15mm}arbitrary.~\hspace*{-0.15mm}Then,~\hspace*{-0.15mm}\mbox{referring}~\hspace*{-0.15mm}to~\hspace*{-0.15mm}the~\hspace*{-0.15mm}co-
	coercivity of $I:W^{1,p}_D(\Omega)\to \mathbb{R}\cup \{+\infty\}$ (cf. \eqref{strong_convexity}) and~the~weak~\mbox{duality}~\mbox{principle}~(cf.~\eqref{weak_duality}),  we find that\vspace*{-0.5mm} 
		\begin{align*}
				\rho^2_I(\tilde{u}_h,u)&\leq I(\tilde{u}_h)-I(u)\leq I(\tilde{u}_h)-D(\tilde{z}_h)\\
				&=\int_{\Omega}{\phi(\nabla \tilde{u}_h)\,\textup{d}x}+\int_{\Omega}{\psi_h(\cdot,\tilde{u}_h)\,\textup{d}x}+
				\int_{\Omega}{\phi^*(\tilde{z}_h)\,\textup{d}x}+\int_{\Omega}{\psi_h^*(\cdot,\textup{div}(\tilde{z}_h))\,\textup{d}x}.
		\end{align*}
		Eventually, using the integration-by-parts formula \eqref{eq:pi0}, we conclude the assertion.
	\end{proof}

	\begin{remark}\label{rem:general_pd_estimator0}
		\begin{itemize}[noitemsep,topsep=2pt,leftmargin=!,labelwidth=\widthof{\textit{(iii)}},font=\itshape]
			\item[(i)]  \hspace*{-0.1mm}The \hspace*{-0.1mm}primal-dual \hspace*{-0.1mm}a \hspace*{-0.1mm}posteriori \hspace*{-0.1mm}error \hspace*{-0.1mm}estimator \hspace*{-0.1mm}$ \eta^2_h(\tilde{u}_h,\tilde{z}_h)$~\hspace*{-0.1mm}for~\hspace*{-0.1mm}${\tilde{u}_h\!\in\! \smash{W^{1,p}_D(\Omega)}}$ and $\tilde{z}_h\!\in\! \smash{W^{p'}_N(\textup{div};\Omega)}$ yields a reliable upper bound for the approximation~error~$\rho^2_I(\tilde{u}_h,u)$, i.e., \hspace*{-0.1mm}we \hspace*{-0.1mm}do \hspace*{-0.1mm}not \hspace*{-0.1mm}have \hspace*{-0.1mm}to \hspace*{-0.1mm}compute \hspace*{-0.1mm}any \hspace*{-0.1mm}exact/discrete \hspace*{-0.1mm}solution \hspace*{-0.1mm}of \hspace*{-0.1mm}the \hspace*{-0.1mm}primal \hspace*{-0.1mm}or~\hspace*{-0.1mm}dual~\hspace*{-0.1mm}\mbox{problem}, resp. However, note that for non-admissible $\tilde{u}_h\in \smash{W^{1,p}_D(\Omega)}$~and/or~${\tilde{z}_h\in \smash{W^{p'}_N(\textup{div};\Omega)}}$, the critical case $ \eta^2_h(\tilde{u}_h,\tilde{z}_h)=+\infty$ might occur.
			
			\item[(ii)] The \hspace*{-0.15mm}a \hspace*{-0.15mm}posteriori \hspace*{-0.15mm}error \hspace*{-0.15mm}estimate \hspace*{-0.1mm}\eqref{eq:general_pd_estimator0} \hspace*{-0.15mm}is \hspace*{-0.15mm}entirely \hspace*{-0.15mm}constant-free, \hspace*{-0.5mm}making~\hspace*{-0.15mm}the~\hspace*{-0.15mm}error~\hspace*{-0.15mm}\mbox{estimator}  $ \eta^2_h\!:\!\smash{W^{1,p}_D(\Omega)}\!\times\!\smash{W^{p'}_N(\textup{div};\Omega)}\!\to\!\mathbb{R}\!\cup\!\{+\infty\}$ \hspace*{-0.5mm}attractive \hspace*{-0.15mm}compared \hspace*{-0.15mm}to \hspace*{-0.15mm}classical~\hspace*{-0.15mm}residual~\hspace*{-0.15mm}type~\hspace*{-0.15mm}error estimators, \hspace*{-0.1mm}which \hspace*{-0.1mm}usually \hspace*{-0.1mm}depend \hspace*{-0.1mm}on \hspace*{-0.1mm}constants \hspace*{-0.1mm}that \hspace*{-0.1mm}are \hspace*{-0.1mm}difficult~\hspace*{-0.1mm}to~\hspace*{-0.1mm}bound~\hspace*{-0.1mm}\mbox{accurately}.
			
			\item[(iii)] The (local) refinement indicators $(\eta_{h,T}^2(\tilde{u}_h,\tilde{z}_h))_{T\in \mathcal{T}_h}$, for every ${T\in \mathcal{T}_h}$~defined~by\vspace*{-0.5mm}
			\begin{align}
				\begin{aligned}
					\eta_{h,T}^2(\tilde{u}_h,\tilde{z}_h)&\vcentcolon=\int_T{\phi(\nabla \tilde{u}_h)- \tilde{z}_h\cdot\nabla \tilde{u}_h+\phi^*(\tilde{z}_h)\,\textup{d}x}
					\\&\quad+
					\int_T{\psi_h(\cdot,\tilde{u}_h)- \textup{div}(\tilde{z}_h)\, \tilde{u}_h+\psi_h^*(\cdot,\textup{div}(\tilde{z}_h))\,\textup{d}x},
				\end{aligned}\label{eq:def_eta}
			\end{align}
			are non-negative by the Fenchel--Young inequality (cf. \eqref{eq:fenchel_young_ineq}).
				\item[(iv)] If \hspace*{-0.15mm}we \hspace*{-0.15mm}choose \hspace*{-0.15mm}$\tilde{u}_h\hspace*{-0.2em}\in\hspace*{-0.2em} \mathcal{S}^1_D(\mathcal{T}_h)\hspace*{-0.2em}\subseteq\hspace*{-0.2em}\smash{W^{1,p}_D(\Omega)} $ \hspace*{-0.15mm}and \hspace*{-0.15mm}$\tilde{z}_h\hspace*{-0.2em}\in\hspace*{-0.2em}\mathcal{R}T^0_N(\mathcal{T}_h)\hspace*{-0.2em}\subseteq \hspace*{-0.2em}\smash{W^{p'}_N(\textup{div};\Omega)}$~\hspace*{-0.15mm}in~\hspace*{-0.15mm}\mbox{Proposition}~\hspace*{-0.15mm}\ref{eq:general_pd_estimator0}, then, exploiting that $\nabla \tilde{u}_h\in \mathcal{L}^0(\mathcal{T}_h)^d$ and $\textup{div}(\tilde{z}_h)\in \mathcal{L}^0(\mathcal{T}_h)$, for~every~${T\in \mathcal{T}_h}$,~it~holds\vspace*{-0.5mm} 
				\begin{align}
					\begin{aligned}
					\eta^2_{h,T}(\tilde{u}_h,\tilde{z}_h)&=\int_T{\phi(\nabla \tilde{u}_h)- \Pi_h\tilde{z}_h\cdot\nabla \tilde{u}_h+\phi^*(\Pi_h\tilde{z}_h)\,\textup{d}x}
					\\&\quad+
					\int_T{\psi_h(\cdot,\Pi_h\tilde{u}_h)- \textup{div}(\tilde{z}_h)\,\Pi_h \tilde{u}_h+\psi_h^*(\cdot,\textup{div}(\tilde{z}_h))\,\textup{d}x}\\&\quad+
					\int_T{\psi_h(\cdot,\tilde{u}_h)-\psi_h(\cdot,\Pi_h\tilde{u}_h) \,\textup{d}x}+	\int_T{\phi^*(\tilde{z}_h)-\phi^*(\Pi_h\tilde{z}_h)\,\textup{d}x}\\
					&=\vcentcolon\eta^2_{A,T}(\tilde{u}_h,\tilde{z}_h)+\eta^2_{B,T}(\tilde{u}_h,\tilde{z}_h)+\eta^2_{C,T}(\tilde{u}_h)+\eta^2_{D,T}(\tilde{z}_h).
					\end{aligned}\label{eq:general_pd_estimator1}
				\end{align}
				Then, $\smash{\eta^2_{A,h}(\tilde{u}_h,\tilde{z}_h)}$, $\smash{\eta^2_{B,h}(\tilde{u}_h,\tilde{z}_h)}$, $\smash{\eta^2_{C,h}(\tilde{u}_h)}$ and $\smash{\eta^2_{D,h}(\tilde{z}_h)}$ are defined by summation of the corresponding element-wise quantities.\vspace*{-5mm}\newpage
			\begin{itemize}[noitemsep,topsep=2pt,leftmargin=!,labelwidth=\widthof{\textit{(iii.a)}},font=\itshape]
				\item[(iv.a)] The representation \eqref{eq:general_pd_estimator1} for  $\tilde{u}_h\in \mathcal{S}^1_D(\mathcal{T}_h)$ and $\tilde{z}_h\in \mathcal{R}T^0_N(\mathcal{T}_h)$ has~the~\mbox{particular} advantage \hspace*{-0.1mm}that \hspace*{-0.1mm}the \hspace*{-0.1mm}integrands \hspace*{-0.1mm}of \hspace*{-0.1mm}the \hspace*{-0.1mm}first \hspace*{-0.1mm}two \hspace*{-0.1mm}integrals \hspace*{-0.1mm}on \hspace*{-0.1mm}the \hspace*{-0.1mm}right-hand~\hspace*{-0.1mm}side,~\hspace*{-0.1mm}i.e., \hspace*{-0.1mm}of \hspace*{-0.1mm}$\eta^2_{A,h}(\tilde{u}_h,\tilde{z}_h)$ \hspace*{-0.1mm}and \hspace*{-0.1mm}$\eta^2_{B,h}(\tilde{u}_h,\tilde{z}_h)$, \hspace*{-0.1mm}are \hspace*{-0.1mm}element-wise~\hspace*{-0.1mm}constant,~\hspace*{-0.1mm}i.e.,~\hspace*{-0.1mm}we~\hspace*{-0.1mm}have~\hspace*{-0.1mm}that\vspace*{-0.5mm}
				\begin{align*}
					\phi(\nabla \tilde{u}_h)- \Pi_h\tilde{z}_h\cdot\nabla \tilde{u}_h+\phi^*(\Pi_h\tilde{z}_h)&\in\mathcal{L}^0(\mathcal{T}_h),\\
					\psi_h(\cdot,\Pi_h\tilde{u}_h)- \textup{div}(\tilde{z}_h)\,\Pi_h \tilde{u}_h+\psi_h^*(\cdot,\textup{div}(\tilde{z}_h))&\in\mathcal{L}^0(\mathcal{T}_h),
				\end{align*}
				which  settles the question of a suitable choice of  a quadrature  for these~\mbox{integrals}, i.e., they do not produce any further quadrature errors.
				
				\item[(iv.b)] Noting that both $ {(x \mapsto   \psi_h(x,\tilde{u}_h(x))), (x \mapsto  \phi^*(\tilde{z}_h(x))) : T \to  \mathbb{R} \cup \{+\infty\}}$~\mbox{define} convex functions for all $T\hspace*{-0.05em}\in\hspace*{-0.05em}  \mathcal{T}_h$,  $\tilde{u}_h\hspace*{-0.05em}\in\hspace*{-0.05em} \mathcal{S}^1_D(\mathcal{T}_h)$ and ${\tilde{z}_h\hspace*{-0.05em}\in\hspace*{-0.05em} \mathcal{R}T^0_N(\mathcal{T}_h)}$,~where~we~used that $\psi_h(\cdot,t)\!\in\! \mathcal{L}^0(\mathcal{T}_h) $ for all $t\in \mathbb{R}$,  similar to \cite[Remark 4.8]{BM20}, we propose a trapezoidal quadrature that leads to a reliable upper bound,~i.e.,~for~every~${T\!\in\! \mathcal{T}_h}$, we have that\vspace*{-0.5mm}
				\begin{align}
					\begin{aligned}
					\eta^2_{C,T}(\tilde{u}_h)&\leq \widehat{\eta}^2_{C,T}(\tilde{u}_h)\vcentcolon=\int_T{\widehat{I}_h\big[\psi_h(\cdot,\tilde{u}_h)\big]-\psi_h(\cdot,\Pi_h\tilde{u}_h) \,\textup{d}x},\\
					\eta^2_{D,T}(\tilde{z}_h)&\leq \widehat{\eta}^2_{D,T}(\tilde{z}_h)\vcentcolon=\int_T{\widehat{I}_h\big[\phi^*(\tilde{z}_h)\big]-\phi^*(\Pi_h\tilde{z}_h)\,\textup{d}x},
				\end{aligned}\label{trapez}
				\end{align}
				where we denote by $\smash{\widehat{I}_h}\hspace*{-0.1em}:\hspace*{-0.1em}C^0(\mathcal{T}_h)\hspace*{-0.1em}\to\hspace*{-0.1em} \mathcal{L}^1(\mathcal{T}_h)$\footnote{Here, $C^0(\mathcal{T}_h)\vcentcolon=\{v_h\in L^\infty(\Omega)\mid v_h|_T\in C^0(T)\text{ for all }T\in \mathcal{T}_h\}$.} the element-wise~nodal~\mbox{interpolation} operator, for every $v_h\in C^0(\mathcal{T}_h)$ defined by $\smash{\widehat{I}_h}v_h|_T\vcentcolon=\sum_{z\in \mathcal{N}_h\cap T}{(v_h|_T)(z)\varphi_z}$~for all \hspace*{-0.1mm}$T\!\in\! \mathcal{T}_h$. \hspace*{-0.5mm}More \hspace*{-0.1mm}precisely, \hspace*{-0.1mm}we \hspace*{-0.1mm}propose \hspace*{-0.1mm}the \hspace*{-0.1mm}trapezoidal \hspace*{-0.1mm}a \hspace*{-0.1mm}posteriori~\hspace*{-0.1mm}error~\hspace*{-0.1mm}\mbox{estimator} $ \widehat{\eta}^2_h:\smash{\mathcal{S}^1_D(\mathcal{T}_h)}\times\smash{\mathcal{R}T^0_N(\mathcal{T}_h)}\to\mathbb{R}\cup\{+\infty\}$, for every  $\tilde{u}_h\hspace*{-0.05em}\in\hspace*{-0.05em} \mathcal{S}^1_D(\mathcal{T}_h)$~and~${\tilde{z}_h\hspace*{-0.05em}\in\hspace*{-0.05em} \mathcal{R}T^0_N(\mathcal{T}_h)}$ defined by\vspace*{-0.5mm}
				\begin{align}
					\widehat{\eta}^2_h(\tilde{u}_h,\tilde{z}_h)&\vcentcolon=\eta^2_{A,h}(\tilde{u}_h,\tilde{z}_h)+\eta^2_{B,h}(\tilde{u}_h,\tilde{z}_h)+\widehat{\eta}^2_{C,h}(\tilde{u}_h)+\widehat{\eta}^2_{D,h}(\tilde{z}_h),\label{eq:hat1}
				\end{align}
				i.e., $\widehat{\eta}^2_h(\tilde{u}_h,\tilde{z}_h)=\sum_{T\in \mathcal{T}_h}{	\widehat{\eta}^2_{h,T}(\tilde{u}_h,\tilde{z}_h)}$, where for every $T\in \mathcal{T}_h$
				\begin{align}
					\widehat{\eta}^2_{h,T}(\tilde{u}_h,\tilde{z}_h)\vcentcolon=\eta^2_{A,T}(\tilde{u}_h,\tilde{z}_h)+\eta^2_{B,T}(\tilde{u}_h,\tilde{z}_h)+\widehat{\eta}^2_{C,T}(\tilde{u}_h)+\widehat{\eta}^2_{D,T}(\tilde{z}_h).\label{eq:hat2}
				\end{align}
				Then, for every $\tilde{u}_h\in\smash{\mathcal{S}^1_D(\mathcal{T}_h)}$ and  ${\tilde{z}_h\in \smash{\mathcal{R}T^0_N(\mathcal{T}_h)}}$, appealing~to~\eqref{trapez}~and~(iii), we have that
				\begin{align}
					\widehat{\eta}^2_{h,T}(\tilde{u}_h,\tilde{z}_h)\ge	\eta^2_{h,T}(\tilde{u}_h,\tilde{z}_h)\ge 0\quad\text{ for all }T\in \mathcal{T}_h.\label{eq:hat3}
				\end{align}
			\end{itemize}
			\item[(v)] The assumption that $\psi(x,t) =\psi_h(x,t)$ for almost every $x\in \Omega$ and all $t\in \mathbb{R}$ can be avoided by considering $I_h:W^{1,p}_D(\Omega)\to \mathbb{R}\cup\{+\infty\}$, for every $v\in W^{1,p}_D(\Omega)$ defined by\vspace*{-0.5mm}
			\begin{align*}
				I_h(v)\vcentcolon=\int_{\Omega}{\phi(\nabla v)\,\textup{d}x}+\int_{\Omega}{\psi_h(\cdot,v)\,\textup{d}x},
			\end{align*}
			and noting that for every $v\in \smash{W^{1,p}_D(\Omega)}$, it holds $	I(v)-I_h(v)=\int_{\Omega}{\psi(\cdot,v)-\psi_h(\cdot,v) \,\textup{d}x}$.
			In combination with a priori bounds for a minimizing function $u\in  \smash{W^{1,p}_D(\Omega)}$~of~the original functional $I\!:\!W^{1,p}_D(\Omega)\!\to\! \mathbb{R}\cup\{+\infty\}$, this approximation error leads~to~a~\mbox{computable} bound that can be included in the error analysis:
			\begin{itemize}[noitemsep,topsep=1pt,leftmargin=!,labelwidth=\widthof{\textit{(v.a)}},font=\itshape]
					\item[(v.a)] If $\psi(x,t)\hspace*{-0.1em} \vcentcolon=\hspace*{-0.1em} -f(x)t$ and $\psi_h(x,t)\hspace*{-0.1em} \vcentcolon=\hspace*{-0.1em} -f_h(x)t$ for almost every $x\hspace*{-0.1em}\in\hspace*{-0.1em} \Omega$~and~all~${t\hspace*{-0.1em}\in\hspace*{-0.1em}\mathbb{R}}$, where $f\in L^{p'}(\Omega)$, $p\in [1,\infty]$, and $f_h\vcentcolon=\Pi_hf\in \mathcal{L}^0(\mathcal{T}_h)$, with~an~element-wise application of  Poincar\'e's inequality, we find that\vspace*{-0.5mm}
					\begin{align*}
						\smash{I(u)-I_h(u)\leq c_{\textup{P}}\|h_{\mathcal{T}}(f-f_h)\|_{L^{p'}(\Omega)}\|\nabla u\|_{L^p(\Omega;\mathbb{R}^d)}.}
					\end{align*}
					\item[(v.b)]  If $\psi(x,t)\vcentcolon= \smash{\frac{\alpha}{2}(t-g(x))^2}$ and $\psi_h(x,t) \vcentcolon= \smash{\frac{\alpha}{2}(t-g_h(x))^2}$ for almost every $x\in \Omega$ and all $t\in\mathbb{R}$, where $g\in L^\infty(\Omega)$ and $g_h\vcentcolon=\Pi_hg\in \mathcal{L}^0(\mathcal{T}_h)$, then we find  that\vspace*{-0.5mm}
					\begin{align*}
						\smash{I(u)-I_h(u)\leq   \alpha\big( \|g-g_h\|_{L^2(\Omega)}^2+ \|u-\Pi_hu\|_{L^2(\Omega)}^2\big) .}
					\end{align*}
			\end{itemize}
		\end{itemize}
	\end{remark}

	\section{A posteriori error estimation based on post-processing non-conforming approximations}\label{sec:post-processing}
		
		\qquad In  this section, we refine the  a posteriori error estimate \eqref{eq:general_pd_estimator0}  resorting~to~a~\mbox{post-process}-ing of the non-conforming, discontinuous approximation  \eqref{discrete_primal}
		of the primal problem \eqref{primal} and  discrete convex optimality relations derived in \cite{Bar21,BW21} or Proposition \ref{prop:optimality}.
		
		\begin{proposition}\label{prop:post-pocess0}
			Let $\psi(x,t)\vcentcolon=\psi_h(x,t)$ for almost every $x\in \Omega$ and all $t\in \mathbb{R}$.
			Moreover, let $u_h^{\textit{cr}}\in \mathcal{S}^{1,\textit{cr}}_D(\mathcal{T}_h)$ and $z_h^{\textit{rt}}\in \mathcal{R}T^0_N(\mathcal{T}_h)$ be such that \eqref{eq:discrete_strong_duality},  \eqref{prop:optimality1}  and \eqref{prop:optimality2}  are satisfied.
			Then, for every ${\tilde{u}_h\in \mathcal{S}^1_D(\mathcal{T}_h)}$, we have that
			\begin{align}\begin{aligned}
					\eta^2_{A,h}(\tilde{u}_h, z_h^{\textit{rt}})&= \int_{\Omega}{\phi(\nabla \tilde{u}_h)-\Pi_h z_h^{\textit{rt}}\cdot(\nabla \tilde{u}_h-\nabla_h u_h^{\textit{cr}})-\phi(\nabla_h u_h^{\textit{cr}})\,\textup{d}x}\\
					\eta^2_{B,h}(\tilde{u}_h, z_h^{\textit{rt}})&= \int_{\Omega}{\psi_h(\cdot,\Pi_h \tilde{u}_h)-\textup{div}(z_h^{\textit{rt}})\,(\Pi_h \tilde{u}_h-\Pi_h u_h^{\textit{cr}})-\psi_h(\cdot,\Pi_h u_h^{\textit{cr}})\,\textup{d}x}.
				\end{aligned}\label{eq:post-pocess0.1}
			\end{align}
		\end{proposition}
	
		\begin{proof}
			Appealing to Proposition \ref{prop:optimality},  the optimality relations \eqref{eq:discrete_optimality_relations2}
			apply. 
			Thus, choosing $\tilde{z}_h=z_h^{\textit{rt}}\in \mathcal{R}T^0_N(\mathcal{T}_h)$, we immediately conclude the assertion.
		\end{proof}
	
		\qquad If $\phi\in  C^1(\mathbb{R}^d)$ or $\psi_h(x,\cdot)\in C^1(\mathbb{R})$ for almost every $x\in \Omega$, then~\eqref{eq:general_pd_estimator0}~can~be~refined.
		
		\begin{corollary}\label{cor:post-pocess}
			Let the assumptions of Proposition \ref{prop:post-pocess0} be satisfied. Then, the following statements apply:
			\begin{itemize}[noitemsep,topsep=2pt,leftmargin=!,labelwidth=\widthof{\textit{(iii)}},font=\itshape]
			\item[(i)] If $\phi\in C^1(\mathbb{R}^d)$,  then for every $\tilde{u}_h\in \mathcal{S}^1_D(\mathcal{T}_h)$, we have that
			\begin{align}
			\begin{aligned}
			\eta^2_{A,h}(\tilde{u}_h, z_h^{\textit{rt}}) \leq  \int_{\Omega}{\big(D\phi(\nabla \tilde{u}_h)-D\phi(\nabla_h u_h^{\textit{cr}})\big)\cdot(\nabla \tilde{u}_h-\nabla_h u_h^{\textit{cr}})\,\textup{d}x}.
			\end{aligned}\label{eq:nondata_error.1}
			\end{align}
		 	\item[(ii)]  If $\psi_h(x,\cdot)\in C^1(\mathbb{R})$ for almost every $x\in \Omega$, then for every $\tilde{u}_h\in \mathcal{S}^1_D(\mathcal{T}_h)$,~we~have~that
		 	\begin{align}
		 		\begin{aligned}
		 			\eta^2_{B,h}(\tilde{u}_h, z_h^{\textit{rt}})\leq 	\int_{\Omega}{\big(D\psi_h(\cdot,\Pi_h \tilde{u}_h)-D\psi_h(\cdot,\Pi_h u_h^{\textit{cr}})\big)\cdot(\Pi_h \tilde{u}_h-\Pi_h u_h^{\textit{cr}})\,\textup{d}x},
		 			\end{aligned}\label{eq:data_error.1}
		 	\end{align}
	 		and
	 		\begin{align}
	 			\begin{aligned}
	 				\eta^2_{C,h}(\tilde{u}_h)\leq 	\int_{\Omega}{\big(D\psi_h(\cdot,\tilde{u}_h)-D\psi_h(\cdot,\Pi_h \tilde{u}_h)\big)\cdot(\tilde{u}_h-\Pi_h \tilde{u}_h)\,\textup{d}x}.
	 			\end{aligned}\label{eq:data_error.1.2}
	 		\end{align}
	 	\item[(iii)]  If $\phi^*\in C^1(\mathbb{R}^d)$, then for every $\tilde{z}_h\in \mathcal{R}T^0_N(\mathcal{T}_h)$, we have that
	 	\begin{align}
	 		\begin{aligned}
\eta^2_{D,h}( \tilde{z}_h)\leq \int_{\Omega}{\big(D\phi^*(\tilde{z}_h)-D\phi^*(\Pi_h \tilde{z}_h)\big)\cdot(\tilde{z}_h-\Pi_h \tilde{z}_h)\,\textup{d}x}.
\end{aligned}\label{eq:data_error.2}
	 	\end{align}
			\end{itemize}
		\end{corollary}
	
		\begin{proof}
			Let  $\tilde{u}_h\in \mathcal{S}^1_D(\mathcal{T}_h)$ be fixed, but arbitrary.
			
			\textit{ad (i)} If $\phi\in C^1(\mathbb{R}^d)$, then, owing to the convexity of $\phi:\mathbb{R}^d\to \mathbb{R}$, it holds
			\begin{align}
				\phi(\nabla_hu_h^{\textit{cr}})\ge \phi(\nabla \tilde{u}_h)+D\phi(\nabla \tilde{u}_h)\cdot(\nabla_hu_h^{\textit{cr}}-\nabla \tilde{u}_h)\quad\text{ in }\Omega.\label{cor:post-pocess1}
			\end{align}
			Then, the optimality relation $\Pi_hz_h^{\textit{rt}}=D\phi(\nabla_hu_h^{\textit{cr}})$ in $\mathcal{L}^0(\mathcal{T}_h)^d$~(cf.~\eqref{eq:optimality1}$_1$)~and~\eqref{cor:post-pocess1}  yield the assertion.
			
			\textit{ad (ii)} If $\psi_h(x,\cdot)\in C^1(\mathbb{R})$ for almost every  $x\in \Omega$, then, owing to the convexity of $\psi_h(x,\cdot):\mathbb{R}\to \mathbb{R}$ for almost every $x\in \Omega$, it holds
			\begin{align}
				\begin{aligned}
			\psi_h(\cdot,\Pi_hu_h^{\textit{cr}})&\ge \psi_h(\cdot,\Pi_h\tilde{u}_h)+D\psi_h(\cdot,\Pi_h\tilde{u}_h)\,(\Pi_hu_h^{\textit{cr}}- \Pi_h \tilde{u}_h)&&\quad\textup{ in }\Omega,\\
			\psi_h(\cdot,\Pi_h\tilde{u}_h)&\ge \psi_h(\cdot,\tilde{u}_h)+D\psi_h(\cdot,\tilde{u}_h)\,(\Pi_h \tilde{u}_h-\tilde{u}_h)&&\quad\textup{ in }\Omega.
		\end{aligned}\label{cor:post-pocess2}
			\end{align}
			Then, the optimality relation $\textup{div}(z_h^{\textit{rt}})=D\psi_h(\cdot,\Pi_hu_h^{\textit{cr}})$ in $\mathcal{L}^0(\mathcal{T}_h)$ (cf.~\eqref{eq:optimality1}$_2$), \eqref{cor:post-pocess2} and $\tilde{u}_h-\Pi_h \tilde{u}_h\perp 
			D\psi_h(\cdot,\Pi_h \tilde{u}_h)$ in $L^2(\Omega)$ yield the assertion.\vspace*{-2mm}\newpage
			
			\textit{ad (iii)} If  $\phi^*\in C^1(\mathbb{R}^d)$, then, due to the convexity of ${\phi^*:\mathbb{R}^d\to \mathbb{R}}$,~it~holds
			\begin{align}
				\phi^*(\Pi_h\tilde{z}_h)\ge \phi^*(\tilde{z}_h)+D\phi^*(\tilde{z}_h)\cdot(\Pi_h\tilde{z}_h-\tilde{z}_h)\quad\textup{ in }\Omega.
\label{cor:post-pocess3}
			\end{align}
			Then, \eqref{cor:post-pocess3} and $\tilde{z}_h-\Pi_h \tilde{z}_h\perp 
			D\phi^*(\Pi_h\tilde{z}_h)$ in $L^2(\Omega;\mathbb{R}^d)$ yield the assertion.\vspace*{-0.5mm} 
		\end{proof}
		
		\begin{remark}[Improved estimates for strongly convex functionals]
			If $\phi\in C^1(\mathbb{R}^d)$ is strongly convex, i.e., there exists a possibly vanishing bi-variate~functional~${\rho_{\phi}\!:\!\mathbb{R}^d\!\times\!\mathbb{R}^d\!\to \!\mathbb{R}_{\ge 0}}$ such that for every $a,b\in \mathbb{R}^d$, we have that $\phi(b)\ge \phi(a)+D\phi(a)\cdot(a-b )+\rho_{\phi}(a-b,a-b)$, 
			then applying Corollary \ref{cor:post-pocess}, \eqref{eq:general_pd_estimator0} can be improved since we can 
			subtract the error quantity
			\begin{align*}
				\vert \tilde{u}_h- u_h^{\textit{cr}}\vert_{\phi}^2\vcentcolon= \int_{\Omega}{\rho_{\phi}(\nabla \tilde{u}_h-\nabla_h u_h^{\textit{cr}},\nabla \tilde{u}_h-\nabla_h u_h^{\textit{cr}})\,\textup{d}x}
			\end{align*}
			on \hspace*{-0.1mm}the \hspace*{-0.1mm}right-hand \hspace*{-0.1mm}side \hspace*{-0.1mm}or, \hspace*{-0.1mm}equivalently, \hspace*{-0.1mm}add\hspace*{-0.1mm} it \hspace*{-0.1mm}to \hspace*{-0.1mm}the \hspace*{-0.1mm}left-hand \hspace*{-0.1mm}side, \hspace*{-0.1mm}i.e., 
			\hspace*{-0.1mm}in~\hspace*{-0.1mm}the~\hspace*{-0.1mm}\mbox{inequality}~\hspace*{-0.1mm}\eqref{eq:general_pd_estimator0}, the error quantity $\rho_I^2(\tilde{u}_h,u)$ can be replaced by the quantity ${\tilde{\rho}_I^2(\tilde{u}_h,u)\!\vcentcolon=\!\rho_I^2(\tilde{u}_h,u)\!+\!\vert \tilde{u}_h\!-\! u_h^{\textit{cr}}\vert_{\phi}^2}$.
			The~same~equally~applies~to the functional $\psi_h:\Omega\times \mathbb{R}\to \mathbb{R}\cup\{+\infty\}$.
		\end{remark}
		\begin{remark}[Iteration errors]
			If the Crouzeix--Raviart approximation $u_h^{\smash{cr,it}}\in\smash{\mathcal{S}^{1,\textit{cr}}_D(\mathcal{T}_h)}$ of the primal problem  is obtained
			by an iterative scheme, then there~exists~a~$\smash{r_h^{\smash{cr,it}}\in \mathcal{S}^{1,\textit{cr}}_D(\mathcal{T}_h)}$ such \hspace*{-0.1mm}that \hspace*{-0.1mm}$u_h^{\smash{cr,it}}\hspace*{-0.2em}\in\hspace*{-0.15em} \smash{\mathcal{S}^{1,\textit{cr}}_D(\mathcal{T}_h)}$ \hspace*{-0.1mm}is \hspace*{-0.1mm}minimal \hspace*{-0.1mm}for \hspace*{-0.1mm}the \hspace*{-0.1mm}functional \hspace*{-0.1mm}$I_h^{\smash{cr,it}}\hspace*{-0.2em}:\hspace*{-0.15em}\smash{\mathcal{S}^{1,\textit{cr}}_D(\mathcal{T}_h)}\hspace*{-0.15em}\to\hspace*{-0.15em} \mathbb{R}\cup\{\infty\}$,~\hspace*{-0.1mm}for~\hspace*{-0.1mm}\mbox{every} $\smash{v_h\in \mathcal{S}^{1,\textit{cr}}_D(\mathcal{T}_h)}$ defined by
			\begin{align*}
				I_h^{\smash{cr,it}}(v_h)\vcentcolon=\int_{\Omega}{\phi(\nabla_h v_h)\,\textup{d}x}+\int_{\Omega}{\psi_h(\cdot,\Pi_hv_h)\,\textup{d}x}-\int_{\Omega}{r_h^{cr,it}\,v_h\,\textup{d}x}.
			\end{align*}
			Then, the corresponding representation of $z_h^{\textit{rt}}\in \mathcal{R}T^0_N(\mathcal{T}_h)$ requires the contribution~$r_h^{\smash{it}} \vcentcolon= \Pi_hr_h^{\smash{cr,it}}\in \mathcal{L}^0(\mathcal{T}_h)$.
			The difference $r_h^{\smash{it}}-r_h^{\smash{cr,it}}\!\in\! \mathcal{L}^1(\mathcal{T}_h)$ can be controlled~as~the~approximation error in Remark \ref{rem:general_pd_estimator0} (iv).
			The function $r_h^{\smash{it}}\hspace*{-0.1em}\in\hspace*{-0.1em} \mathcal{L}^0(\mathcal{T}_h)$ is explicitly available~via~\mbox{post-processing} the discrete Euler--Lagrange equation
			\begin{align*}
				\int_{\Omega}{D\phi (\nabla_h u_h^{\smash{cr,it}})\cdot\nabla_hv_h\,\textup{d}x}
				+\int_{\Omega}{D\psi_h (\cdot,\Pi_h u_h^{\smash{cr,it}})\,\Pi_hv_h\,\textup{d}x}=\int_{\Omega}{r_h^{cr,it}\,v_h\,\textup{d}x}
			\end{align*}
			for all $v_h\!\in\! \mathcal{S}^{1,\textit{cr}}_D(\mathcal{T}_h)$, assuming that $\phi\!\in\! C^1(\mathbb{R}^d)$ and $\psi_h(x,\cdot)\!\in\! C^1(\mathbb{R})$ for~almost~every~${x\!\in\! \Omega}$.
		\end{remark}
	
		\begin{remark}[Data errors]
			If $\phi\!\in\! C^1(\mathbb{R}^d)$ such that
			$D\phi^*\!\in\! C^{0,\alpha}(\mathbb{R}^d;\mathbb{R}^d)$ for~some~${\alpha\!\in\! (0,1]}$ and $\psi_h(x,t)\vcentcolon=-f_h(x)t$ for almost every $x\in \Omega$ and all $t\in \mathbb{R}$,~where~${f\in L^{p'}(\Omega)}$,~${p\in(1,\infty)}$, and $f_h=\Pi_h f\in \mathcal{L}^0(\mathcal{T}_h)$, then, using the reconstruction formula \eqref{eq:reconstruction_formula}, we observe that
			\begin{align*}
				\begin{aligned}
				\int_{\Omega}{\big(D\phi^*(z_h^{\textit{rt}})-D\phi^*(\Pi_h z_h^{\textit{rt}})\big)\cdot(z_h^{\textit{rt}}-\Pi_h z_h^{\textit{rt}})\,\textup{d}x}
				\leq c_{\alpha,\phi^*}\int_{\Omega}{\big\vert f_hd^{-1}\big(\textup{id}_{\mathbb{R}^d}-\Pi_h\textup{id}_{\mathbb{R}^d}\big) \big\vert^{1+\alpha}\,\textup{d}x},
			\end{aligned}
			\end{align*}
			while, owing to $\textup{div}(z_h^{\textit{rt}})=-f_h$ in $\mathcal{L}^0(\mathcal{T}_h)$, it holds
			\begin{align*}
			\int_{\Omega}{\psi_h(\cdot,\Pi_h \tilde{u}_h)-\textup{div}(z_h^{\textit{rt}})\,(\Pi_h \tilde{u}_h-\Pi_h u_h^{\textit{cr}})+\psi_h(\cdot,\Pi_h u_h^{\textit{cr}})\,\textup{d}x}&=0,\\
			\int_{\Omega}{\psi_h(\cdot,\tilde{u}_h)-\psi_h(\cdot,\Pi_h \tilde{u}_h)\,\textup{d}x}&=0.
			\end{align*}
			As a result, resorting to Proposition \ref{prop:general_pd_estimator0}, Proposition \ref{prop:post-pocess0} and Corollary~\ref{cor:post-pocess},~we~conclude~that
			\begin{align*}
				\rho_I^2(\tilde{u}_h,u)&\leq \int_{\Omega}{\big(D\phi(\nabla \tilde{u}_h)-D\phi(\nabla_h u_h^{\textit{cr}})\big)\cdot(\nabla \tilde{u}_h-\nabla_h u_h^{\textit{cr}})\,\textup{d}x}\\&\quad+c_{\alpha,\phi^*}\int_{\Omega}{\big\vert f_hd^{-1}\big(\textup{id}_{\mathbb{R}^d}-\Pi_h\textup{id}_{\mathbb{R}^d}\big) \big\vert^{1+\alpha}\,\textup{d}x},
			\end{align*}
			where the first integral on the right-hand side has the interpretation of a residual, while the second integral contains data approximation errors.
		\end{remark}
		
		\section{Application: Non-linear Dirichlet problems}\label{sec:applications}\vspace*{-1mm}
		
		\qquad In this section, we apply the general theory built previously to well-known non-linear Dirichlet problems including  the $p$--Dirichlet problem and an optimal design problem.~For the $p$--Dirichlet problem, we establish a reliability and efficiency result (cf. Theorem~\ref{thm:equivalences}).
		\vspace*{-1.5mm}

		\subsection{Continuous non-linear Dirichlet problem}\vspace*{-1mm}
		
		\qquad Given a right-hand side $f\hspace*{-0.1em}\in \hspace*{-0.1em} L^{p'}\hspace*{-0.1em}(\Omega)$, $p\hspace*{-0.1em}\in\hspace*{-0.1em} (1,\infty)$, and a proper,~convex~and~lower~semi-continuous functional $\phi:\mathbb{R}^d\to \mathbb{R}\cup\{+\infty\}$, we examine the 
		non-linear Dirichlet problem defined by the minimization of $I:W^{1,p}_D(\Omega)\to \mathbb{R}\cup\{+\infty\}$,~for~every $v\in W^{1,p}_D(\Omega)$ defined~by\vspace*{-0.5mm}
		\begin{align}
			I(v)\vcentcolon=\int_{\Omega}{ \phi( \nabla v)\,\textup{d}x}-\int_{\Omega}{f\,v\,\textup{d}x },\label{eq:dirichlet_primal}
		\end{align}
		i.e., $\psi(t,x)\hspace*{-0.1em}\vcentcolon=\hspace*{-0.1em}-f(x)t$ for almost every $x\hspace*{-0.1em}\in\hspace*{-0.1em} \Omega$ and $t\hspace*{-0.1em}\in\hspace*{-0.1em} \mathbb{R}$. Proceeding as,~e.g.,~in~\mbox{\cite[p.~113~ff.]{ET99}}, one finds that
		the (Fenchel) dual problem determines~a~vector~field~$\smash{z\in W^{p'}_N(\textup{div};\Omega)}$~that~is maximal for $D:\smash{W^{p'}_N(\textup{div};\Omega)}\to \mathbb{R}\cup\{-\infty\}$, for every $y\in \smash{W^{p'}_N(\textup{div};\Omega)}$ defined by\vspace*{-0.5mm}
		\begin{align}
			D(y)\vcentcolon=-\int_{\Omega}{ \phi^*(y)\,\textup{d}x }-I_{\{-f\}}(\textup{div}(y)),\label{eq:dirichlet_dual}
		\end{align}
		where $I_{\{-f\}}\hspace*{-0.1em}:\hspace*{-0.1em}\smash{L^{p'}(\Omega)}\hspace*{-0.1em}\to\hspace*{-0.1em} \mathbb{R}\hspace*{-0.1em}\cup\hspace*{-0.1em}\{+\infty\}$ is defined by $I_{\{-f\}}(v)\hspace*{-0.1em}\vcentcolon=\hspace*{-0.1em}0$~if~$v\hspace*{-0.1em}=\hspace*{-0.1em}-f$ and  $I_{\{-f\}}(v)\hspace*{-0.1em}\vcentcolon=\hspace*{-0.1em}+\infty$ else. 
		If $\phi\in C^0(\mathbb{R}^d)$, then \cite[Proposition 5.1]{ET99} establishes the existence of a maximizer $\smash{z\in W^{p'}_N(\textup{div};\Omega)}$ of \eqref{eq:dirichlet_discrete_dual} and that strong duality, i.e., $D(z)=\inf_{v\in \smash{W^{1,p}_D(\Omega)}}{I(v)}$, applies. If, in addition,  a minimizer $u\!\in\! W^{1,p}_D(\Omega)$ of \eqref{eq:dirichlet_primal} exists, e.g., if there exist $c_0\!>\!0$ and $c_1\!\ge\! 0$ such that $\phi(t)\ge c_0\vert t\vert^p-c_1$ for all $t\in \mathbb{R}^d$, then \cite[Proposition 5.1]{ET99}~yields~the~optimality~relations\vspace*{-0.5mm}
		\begin{align}
				\smash{\textup{div}(z)=-f\quad\text{ in }L^{p'}(\Omega),\qquad z\cdot\nabla u=\phi^*(z)+\phi(\nabla u)\quad\textup{ a.e. in }\Omega.}\label{eq:dirichlet_optimality1}
		\end{align}
		If $\phi\in C^1(\mathbb{R}^d)$ and there exist $c_0,c_1\ge 0$ such that $\vert D\phi(t)\vert\leq c_0\vert t\vert^{p-1}+c_1$ for all $t\in \mathbb{R}^d$, then, by the Fenchel--Young identity (cf.~\eqref{eq:fenchel_young_id}), the optimality relations \eqref{eq:dirichlet_optimality1} are~equivalent~to\vspace*{-0.5mm}
		\begin{align}
			\smash{\textup{div}(z)=-f\quad\text{ in }\smash{L^{p'}(\Omega)},\qquad z=D\phi(\nabla u)\quad\text{ in }\smash{L^{p'}(\Omega;\mathbb{R}^d)}}.\label{eq:dirichlet_optimality2}
		\end{align}\vspace*{-8mm}
	
		\subsection{Discrete non-linear Dirichlet problem}\vspace*{-1mm}
		
			\qquad Given a right-hand side $f\in  L^{p'}(\Omega)$, $p\in (1,\infty)$, and a proper,~convex~and~lower~semi-continuous functional $\phi:\mathbb{R}^d\to \mathbb{R}\cup\{+\infty\}$, with $f_h\vcentcolon=\Pi_h f\hspace*{-0.1em}\in\hspace*{-0.1em} \mathcal{L}^0(\mathcal{T}_h)$, the discrete~\mbox{non-linear} Dirichlet \hspace*{-0.1mm}problem \hspace*{-0.1mm}determines \hspace*{-0.1mm}a \hspace*{-0.1mm}Crouzeix--Raviart \hspace*{-0.1mm}function \hspace*{-0.1mm}$u_h^{\textit{cr}}\!\in\! \mathcal{S}^{1,\textit{cr}}_D(\mathcal{T}_h)$ \hspace*{-0.1mm}that \hspace*{-0.1mm}is \hspace*{-0.1mm}\mbox{minimal}~for $I_h^{\textit{cr}}:\mathcal{S}^{1,\textit{cr}}_D(\mathcal{T}_h)\to \mathbb{R}\cup\{+\infty\}$, for every 
			$v_h\in \mathcal{S}^{1,\textit{cr}}_D(\mathcal{T}_h)$ defined by\vspace*{-0.5mm}
		\begin{align}
			I_h^{\textit{cr}}(v_h)\vcentcolon=\int_{\Omega}{ \phi( \nabla_h v_h)\,\textup{d}x}-\int_{\Omega}{f_h\,\Pi_hv_h\,\textup{d}x },\label{eq:dirichlet_discrete_primal}
		\end{align}
		Appealing to Subsection \ref{subsec:discrete_convex_min}, the corresponding dual problem to \eqref{eq:dirichlet_discrete_primal} determines a Raviart--Thomas vector field $z_h^{\textit{rt}}\in \mathcal{R}T^0_N(\mathcal{T}_h)$ that is maximal for $D_h^{\textit{rt}}:\mathcal{R}T^0_N(\mathcal{T}_h)\to \mathbb{R}\cup\{-\infty\}$, for every $y_h\in \mathcal{R}T^0_N(\mathcal{T}_h)$ defined by\vspace*{-0.5mm}
	\begin{align}
		D_h^{\textit{rt}}(y_h)\vcentcolon=-\int_{\Omega}{\phi^*( \Pi_hy_h)\,\textup{d}x}-I_{\{-f_h\}}(\textup{div}(y_h)).\label{eq:dirichlet_discrete_dual}
	\end{align}
	If a minimum $u_h^{\textit{cr}}\in \mathcal{S}^{1,\textit{cr}}_D(\mathcal{T}_h)$ of \eqref{eq:dirichlet_discrete_primal} and a maximum  $z_h^{\textit{rt}}\in \mathcal{R}T^0_N(\mathcal{T}_h)$~of~\eqref{eq:dirichlet_discrete_dual}~are~given~and discrete \hspace*{-0.1mm}strong duality, \hspace*{-0.1mm}i.e., \hspace*{-0.1mm}$I_h^{\textit{cr}}(u_h^{\textit{cr}})=D_h^{\textit{rt}}(z_h^{\textit{rt}})$,  \hspace*{-0.1mm}and  \hspace*{-0.1mm}\eqref{prop:optimality1} \hspace*{-0.1mm}apply, \hspace*{-0.1mm}then~\hspace*{-0.1mm}Proposition~\hspace*{-0.1mm}\ref{prop:optimality}~\hspace*{-0.1mm}yields\vspace*{-0.5mm}
	\begin{align*}
		\smash{\textup{div}(z_h^{\textit{rt}})=-f_h\quad\text{ in }\mathcal{L}^0(\mathcal{T}_h),\quad\;\Pi_hz_h^{\textit{rt}}\cdot\nabla_hu_h^{\textit{cr}}=\phi^*(\Pi_hz_h^{\textit{rt}})+\phi(\nabla_hu_h^{\textit{cr}})\quad\text{ in }\mathcal{L}^0(\mathcal{T}_h).}
	\end{align*}
	If a  minimizer $u_h^{\textit{cr}}\in \mathcal{S}^{1,\textit{cr}}_D(\mathcal{T}_h)$ of \eqref{eq:dirichlet_discrete_primal} exists, e.g., if   $\phi(t)\ge c_0\vert t\vert^p-c_1$ for all $t\in \mathbb{R}^d$~for~some  $c_0>0$ and $c_1\ge 0$, and if $\phi\in C^1(\mathbb{R}^d)$, then a maximizer  $z_h^{\textit{rt}}\in \mathcal{R}T^0_N(\mathcal{T}_h)$~of~\eqref{eq:dirichlet_discrete_dual}~is~given~by\vspace*{-0.5mm}
	\begin{align}
		\smash{z_h^{\textit{rt}}\vcentcolon=D\phi(\nabla_hu_h^{\textit{cr}})-f_hd^{-1}\big(\textup{id}_{\mathbb{R}^d}-\Pi_h\textup{id}_{\mathbb{R}^d}\big)\quad\textup{ in } \mathcal{R}T^0_N(\mathcal{T}_h)},\label{eq:dirichlet_discrete_reconstruction}
	\end{align}
	i.e., $\smash{\Pi_hz_h^{\textit{rt}}=D\phi(\nabla_hu_h^{\textit{cr}})}$ in $\smash{\mathcal{L}^0(\mathcal{T}_h)^d}\!$, 
    and discrete  strong duality applies.\vspace*{-1cm}\newpage 
		
		\begin{proposition}[Error estimate via comparison to a non-conforming approximation]\label{prop:post-processing_pDirichlet}
			Let $f=f_h\in \mathcal{L}^0(\mathcal{T}_h)$. Moreover, let $u_h^{\textit{cr}}\in \mathcal{S}^{1,\textit{cr}}_D(\mathcal{T}_h)$  and  $z_h^{\textit{rt}}\in \mathcal{R}T^0_N(\mathcal{T}_h)$ be~such~that~\eqref{eq:discrete_strong_duality},  \eqref{prop:optimality1} and \eqref{prop:optimality2} are satisfied.  Then, for every $\tilde{u}_h\in \mathcal{S}^1_D(\mathcal{T}_h)$, we have that
			\begin{align}
				\begin{aligned}	\rho^2_I(\tilde{u}_h,u)&\leq \eta^2_h(\tilde{u}_h,z^{\textit{rt}}_h)
					\\&=\int_{\Omega}{\phi(\nabla \tilde{u}_h)-\Pi_h z_h^{\textit{rt}}\cdot\big(\nabla \tilde{u}_h-\nabla_h u_h^{\textit{cr}}\big)-\phi(\nabla_h u_h^{\textit{cr}})\,\textup{d}x}
					\\&\quad+\int_{\Omega}{\phi^*(z_h^{\textit{rt}})-\phi^*(\Pi_hz_h^{\textit{rt}})\,\textup{d}x},
				\end{aligned}\label{eq:apost_plaplacian_cr1}
			\end{align}
			where, again, $\rho_I^2:\smash{W^{1,p}_D(\Omega)}\times \smash{W^{1,p}_D(\Omega)}\to \mathbb{R}_{\ge 0}$ denotes a measure for the co-coercivity~of $I: \smash{W^{1,p}_D(\Omega)}\to \mathbb{R}\cup\{+\infty\}$ at $u\in \smash{W^{1,p}_D(\Omega)}$, i.e.,  ${\rho_I^2(v,u)\leq I(v)-I(u)}$~for~all~${v\in \smash{W^{1,p}_D(\Omega)}}$.\vspace*{-0.5mm} 
		\end{proposition}
		
		\begin{proof}
				Follows from Proposition \ref{prop:general_pd_estimator0}, Proposition \ref{prop:post-pocess0} and $\eta_{B,h}^2(\tilde{u}_h,z_h^{\textit{rt}})=\eta_{C,h}^2(\tilde{u}_h)=0$.\vspace*{-0.5mm} 
		\end{proof}\vspace*{-0.5mm} 
	
		\subsection{$p$--Dirichlet problem}\vspace*{-0.5mm} 
		
		\qquad In the particular case $\phi\vcentcolon=\smash{\frac{1}{p}}\vert \cdot\vert^p\in C^1(\mathbb{R}^d)$, $p\in (1,\infty)$, the non-linear~\mbox{Dirichlet~problem} \eqref{eq:dirichlet_primal} \hspace*{-0.15mm}reduces \hspace*{-0.15mm}to \hspace*{-0.15mm}the \hspace*{-0.15mm}well-known \hspace*{-0.15mm}$p$\hspace*{-0.1mm}--\hspace*{-0.1mm}Dirichlet \hspace*{-0.15mm}problem.
		 \hspace*{-0.5mm}An \hspace*{-0.15mm}important \hspace*{-0.15mm}property~\hspace*{-0.15mm}of~\hspace*{-0.15mm}the~\hspace*{-0.15mm}\mbox{$p$\hspace*{-0.1mm}--\hspace*{-0.1mm}Dirichlet} \hspace*{-0.1mm}problem \hspace*{-0.1mm}is \hspace*{-0.1mm}that \hspace*{-0.1mm}its \hspace*{-0.1mm}defining \hspace*{-0.1mm}functional \hspace*{-0.1mm}\eqref{eq:dirichlet_primal}  \hspace*{-0.1mm}is \hspace*{-0.1mm}not \hspace*{-0.1mm}only co-coercive~\hspace*{-0.1mm}at~\hspace*{-0.1mm}a~\hspace*{-0.1mm}minimizer~\hspace*{-0.1mm}${u\hspace*{-0.15em}\in\hspace*{-0.15em}\smash{W^{1,p}_D(\Omega)}}$, but even strongly convex.~More~precisely, there exists a  metric ${\rho_I^2\hspace*{-0.15em}:\hspace*{-0.15em}\smash{W^{1,p}_D(\Omega)}\hspace*{-0.15em}\times\hspace*{-0.15em} \smash{W^{1,p}_D(\Omega)}\hspace*{-0.15em}\to\hspace*{-0.15em} \mathbb{R}_{\ge 0}}$ such that for every $v\in\smash{W^{1,p}_D(\Omega)}$,~it~holds
		\begin{align}
			\smash{c_\rho^{-1}\rho_I^2(v,u)\leq I(v)-I(u)\leq c_\rho\rho_I^2(v,u)},\label{eq:stong_convexity}
		\end{align}
		for some constant $c_\rho>0$.  But which is the right choice for ${\rho_I^2:\smash{W^{1,p}_D(\Omega)}\times \smash{W^{1,p}_D(\Omega)}\to \mathbb{R}_{\ge 0}}$? The canonical choice $\smash{\rho_I^2(v,w)\hspace*{-0.15em}\vcentcolon=\hspace*{-0.15em}\|\nabla v-\nabla w\|_{L^p(\Omega;\mathbb{R}^d)}^p}$ for all $v,w\hspace*{-0.15em}\in\hspace*{-0.15em} \smash{W^{1,p}_D(\Omega)}$ is not well-suited since \eqref{eq:stong_convexity}, in general, does not hold and one obtains convergence rates that are sub-optimal for 
		 $\smash{\mathcal{S}^1_D(\mathcal{T}_h)}$,~cf.~\cite{BL93}. Instead,
		a so-called $F$--metric, for every $v,w\in \smash{W^{1,p}_D(\Omega)}$~defined~by
		\begin{align*}
			\smash{\rho_F^2(v,w)\vcentcolon=\|F(\nabla v)-F(\nabla w)\|_{L^2(\Omega;\mathbb{R}^d)}^2},
		\end{align*}
		where  the vector-valued mapping $F:\mathbb{R}^d\to \mathbb{R}^d$ is defined~by~${F(a)\vcentcolon=\smash{\vert a\vert^{\smash{\frac{p-2}{2}}}a}}$~for~all~${a\in \mathbb{R}^d}$,
		has been introduced and widely employed in the literature,~cf.~\mbox{\cite{BL93,EL05,LY01A,LY01B,LY02,DK08,BDK12,BM20}}.
		The $F$--metric satisfies \eqref{eq:stong_convexity}, cf. \cite[Lemma 16]{DK08}, so that
		we set $\rho_I^2(v,w)\vcentcolon=\rho_F^2(v,w)$~for~all $v,w\in \smash{W^{1,p}_D(\Omega)}$. Note that from \eqref{eq:stong_convexity} and the definiteness of ${\rho_I^2\hspace*{-0.15em}:\hspace*{-0.15em}\smash{W^{1,p}_D(\Omega)}\hspace*{-0.15em}\times\hspace*{-0.15em} \smash{W^{1,p}_D(\Omega)}\hspace*{-0.15em}\to\hspace*{-0.15em} \mathbb{R}_{\ge 0}}$, it follows that the $p$--Dirichlet problem admits a unique minimizer.

		This particular choice, in turn, also enables us to 
		relate the primal-dual a posteriori error estimator to the  residual type a posteriori error estimator~in~\cite{DK08},~i.e.,~if~${u^{\textit{c}}_h\!\in\! \smash{\mathcal{S}^1_D(\mathcal{T}_h)}}$~denotes the unique~minimizer~of $\smash{I_h^{\textit{c}}\vcentcolon=\smash{I|_{\smash{\mathcal{S}^1_D(\mathcal{T}_h)}}}:\smash{\mathcal{S}^1_D(\mathcal{T}_h)}\to \mathbb{R}}$, the quantity\vspace*{-0.5mm} 
		\begin{align}
		\begin{aligned}
		\eta^2_{\textit{res},h}(u^{c}_h)\vcentcolon=\sum_{T\in \mathcal{T}_h}{\eta_{\textit{res},T}^2(u^{c}_h)},\end{aligned}\label{rem:residual.1.1}
		\end{align}
		where for every $T\in \mathcal{T}_h$ and $S\in \mathcal{S}_h\setminus \partial\Omega$ with $S\subseteq \partial T$
		\begin{align}
		\begin{aligned}
		\eta_{\textit{res},T}^2(u^{c}_h)&\vcentcolon=\eta_{E,T}^2(u^{c}_h)+\sum_{S\in \mathcal{S}_h\setminus \partial\Omega;S\subseteq \partial T}{\eta_{J,S}^2(u^{c}_h)},\\
	\eta_{E,T}^2(u^{c}_h)&\vcentcolon=\int_T{\big(\vert \nabla u^{c}_h\vert^{p-1}+h_T\vert f_h\vert\big)^{p'-2}h_T^2\vert f_h\vert^2\,\textup{d}x},\\
	\eta_{J,S}^2(u^{c}_h)&\vcentcolon=\int_S{h_S\vert \jump{F(\nabla u^{c}_h)}_S\vert^2\,\textup{d}s}.\end{aligned}\label{rem:residual.1.2}
	\end{align}
	In \cite[Lemma 8 \& Corollary 11]{DK08}, it has been shown that the error estimator~is~reliable~and efficient, i.e., there exist constants $ c_{\textit{rel}},c_{\textit{eff}}>0$~such~that\footnote{Here, we assume that $f=f_h\in \mathcal{L}^0(\mathcal{T}_h)$. For the general case $f\in L^{p'}(\Omega)$, $p\in (1,\infty)$,~oscillation terms, cf. Remark \ref{rmk:1} (ii), need to be added in \eqref{eq:rel-eff}.\vspace*{-13mm}}
	\begin{align}
		\smash{c_{\textit{eff}}\,\eta_{\textit{res},h}^2(u^{c}_h)\leq \rho_I^2(u,u^{c}_h)\leq c_{\textit{rel}}\,\eta_{\textit{res},h}^2(u^{c}_h)}.\label{eq:rel-eff}
	\end{align}
		
		Generalizing the procedure in  \cite{BreS08,Gud10A,Gud10B} and resorting~to~particular~\mbox{properties}~of~the node-averaging \hspace*{-0.1mm}operator \hspace*{-0.1mm}$\mathcal{J}^{\textit{av}}_h\hspace*{-0.15em}:\hspace*{-0.15em}\mathcal{S}^{1,\textit{cr}}_D(\mathcal{T}_h)\hspace*{-0.15em}\to\hspace*{-0.15em} \mathcal{S}^1_D(\mathcal{T}_h)$, \hspace*{-0.1mm}cf. \hspace*{-0.5mm}Appendix \hspace*{-0.1mm}\ref{subsec:auxiliary},~\hspace*{-0.1mm}we~\hspace*{-0.1mm}are~\hspace*{-0.1mm}able~\hspace*{-0.1mm}to~\hspace*{-0.1mm}\mbox{establish} the global equivalence of the primal-dual a posteriori~error~estimator,~cf.~\eqref{eq:apost_plaplacian_cr1}~with~$\smash{\tilde{u}_h\hspace*{-0.1em}=\hspace*{-0.1em}u_h^{\textit{c}}}$ $\in\smash{\hspace*{-0.15em}\mathcal{S}^1_D(\mathcal{T}_h)}$ \hspace*{-0.15mm}and \hspace*{-0.15mm}$\tilde{z}_h\hspace*{-0.15em}=\hspace*{-0.15em}z_h^{\textit{rt}}\hspace*{-0.15em}\in\hspace*{-0.15em} \smash{\mathcal{R}T^0_N(\mathcal{T}_h)}$ \hspace*{-0.15mm}and \hspace*{-0.15mm}the \hspace*{-0.15mm}residual \hspace*{-0.15mm}type \hspace*{-0.15mm}a \hspace*{-0.15mm}posteriori \hspace*{-0.15mm}error~\hspace*{-0.15mm}\mbox{estimator},~\hspace*{-0.15mm}cf.~\hspace*{-0.15mm}\eqref{rem:residual.1.1}.
	
		\begin{theorem}[Equivalence to residual type a posteriori error estimator]\label{thm:equivalences} Let 
			\hspace*{-0.15mm}$\phi\hspace*{-0.15em}\vcentcolon=\hspace*{-0.15em}\smash{\frac{1}{p}}\vert\hspace*{-0.1em}\cdot\hspace*{-0.1em}\vert^p\hspace*{-0.15em}\in\hspace*{-0.15em} C^1(\mathbb{R}^d)$, \hspace*{-0.15mm}$p\in (1,\infty)$, \hspace*{-0.15mm}and \hspace*{-0.15mm}$f=f_h\in \mathcal{L}^0(\mathcal{T}_h)$. \hspace*{-0.5mm}Then, \hspace*{-0.15mm}there~\hspace*{-0.15mm}exists~\hspace*{-0.15mm}a~\hspace*{-0.15mm}\mbox{constant}~\hspace*{-0.15mm}${c_{\textit{eq}}>0}$ such that 
			\begin{align}
					\smash{c_{\textit{eq}}^{-1}\, \eta_h^2(u_h^{\textit{c}},z_h^{\textit{rt}})\leq \eta_{\textit{res},h}^2(u^{\textit{c}}_h)\leq c_{\textit{eq}}\,\eta_h^2(u_h^{\textit{c}},z_h^{\textit{rt}}).}
		\label{eq:equivalence}
			\end{align}
		\end{theorem}
	
		\begin{remark}\label{rmk:1}
			\begin{description}[noitemsep,topsep=1pt,labelwidth=\widthof{\textit{(iii)}},leftmargin=!,font=\normalfont\itshape]
					\item[(i)] \hspace*{-2.5mm}Theorem \hspace*{-0.1mm}\ref{thm:equivalences} \hspace*{-0.1mm}can \hspace*{-0.1mm}be \hspace*{-0.1mm}extended \hspace*{-0.1mm}to \hspace*{-0.1mm}the \hspace*{-0.1mm}case \hspace*{-0.1mm}${\phi\!=\!\varphi\!\circ\! \vert \!\cdot\!\vert\!\in\! C^0(\mathbb{R}^d)\!\cap\! C^2(\mathbb{R}^d\!\setminus\!\{0\})}$, where $\varphi:\mathbb{R}_{\ge 0}\to \mathbb{R}_{\ge 0}$ denotes an $N$--function (cf. Appendix \ref{subsec:auxiliary}) satisfying the $\Delta_2$--condition, the $\nabla_2$--condition,   $\varphi\in C^2(0,\infty)$ and $\varphi'(t)\sim t\varphi''(t)$ uniformly in $t\ge 0$\footnote{Here, we employ the notation $f\sim g $ for two (Lebesgue--)measurable functions $f,g:\Omega\to \mathbb{R}$, if there exists a constant $c>0$ such that $c^{-1}f\leq g\leq cf$ almost everywhere in $\Omega$.\vspace*{-7mm}}, cf. \cite[Assumption 1]{DK08}.
					
					\item[(ii)] If  \hspace*{-0.1mm}$f\hspace*{-0.15em}\in\hspace*{-0.15em} L^{p'}(\Omega)$, \hspace*{-0.1mm}$p\hspace*{-0.15em}\in\hspace*{-0.15em} (1,\infty)$, \hspace*{-0.1mm}in \hspace*{-0.1mm}Theorem \hspace*{-0.1mm}\ref{thm:equivalences}, \hspace*{-0.1mm}then \hspace*{-0.1mm}the \hspace*{-0.1mm}equivalence \hspace*{-0.1mm}\eqref{eq:equivalence} \hspace*{-0.1mm}can \hspace*{-0.1mm}be \hspace*{-0.1mm}extended~\hspace*{-0.1mm}by adding the oscillation term $\textup{osc}(u_h^{\textit{c}})\!\vcentcolon=\!\sum_{T\in \mathcal{T}_h}{\textup{osc}(u_h^{\textit{c}},T)}$,~where~for~all~${T\!\in\! \mathcal{T}_h}$
					\begin{align*}
						\textup{osc}(u_h^{\textit{c}},T)\vcentcolon=\int_T{\big(\vert \nabla u_h^{\textit{c}}\vert^{p-1}+h_T\vert f-\Pi_h f\vert\big)^{p'-2}h_T^2\vert f-\Pi_h f\vert^2\,\textup{d}x}.
					\end{align*}
			\end{description}
		\end{remark}
			
		\begin{proof}
				Introducing the $\mathcal{S}^1_D$--\,$\mathcal{S}^{1,\textit{cr}}_D$\!\!--error $e_h\!\vcentcolon=\!u_h^{\textit{c}}-u_h^{\textit{cr}}\!\in\! \mathcal{S}^{1,\textit{cr}}_D(\mathcal{T}_h)$,  uniformly in ${h\!>\!0}$,~we~get
				\begin{align}\begin{aligned}
					\int_{\Omega}{\big( D\phi(\nabla u_h^{\textit{c}})-D\phi(\nabla_h u_h^{\textit{cr}}) \big)\cdot\nabla_h e_h \,\textup{d}x}&=
					\int_{\Omega}{D\phi(\nabla u_h^{\textit{c}})\cdot\nabla_h( e_h- \mathcal{J}_h^{\textit{av}} e_h) \,\textup{d}x}	\\&\quad+\int_{\Omega}{ f_h\,(\mathcal{J}_h^{\textit{av}} e_h-e_h) \,\textup{d}x}
					\\&\quad+\int_{\Omega}{\big( D\phi(\nabla u_h^{\textit{c}})-D\phi(\nabla u) \big)\cdot\nabla \mathcal{J}_h^{\textit{av}} e_h \,\textup{d}x}
			\\&=\vcentcolon I_h^1+I_h^2+I_h^3.
				\end{aligned}\hspace*{-0.5cm}	\label{rem:residual.3}
				\end{align}
				Using that $\jump{D\phi(\nabla u_h^{\textit{c}})n\cdot( e_h- \mathcal{J}_h^{\textit{av}} e_h)}_S=\jump{D\phi(\nabla u_h^{\textit{c}})n}_S\cdot\{e_h- \mathcal{J}_h^{\textit{av}} e_h\}_S +\{D\phi(\nabla u_h^{\textit{c}})n\}_S\cdot \jump{e_h\!-\!\mathcal{J}_h^{\textit{av}} e_h}_S$ \hspace*{-0.2mm}on \hspace*{-0.2mm}$S$, $\int_S{\jump{e_h\!-\!\mathcal{J}_h^{\textit{av}} e_h}_S\,\textup{d}s}\!=\!0$ \hspace*{-0.2mm}and \hspace*{-0.2mm}$\{D\phi(\nabla u_h^{\textit{c}})\}_S\!=\!\textup{const}$ \hspace*{-0.2mm}on \hspace*{-0.2mm}$S$~\hspace*{-0.2mm}for~\hspace*{-0.2mm}all~\hspace*{-0.2mm}${S\!\in\! \mathcal{S}_h\hspace*{-0.2em}\setminus\hspace*{-0.2em}\partial\Omega}$, an element-wise integration-by-parts,  a discrete trace inequality~\mbox{\cite[Lem.~12.8]{EG21}}~and~\textit{(\hyperlink{AV.4}{AV.4})}, we find that
				\begin{align}
					\begin{aligned}
					I_h^1&=\sum_{S\in \mathcal{S}_h\setminus \partial\Omega}{\int_S{\jump{D\phi(\nabla u_h^{\textit{c}})n}_S\cdot\{e_h- \mathcal{J}_h^{\textit{av}} e_h\}_S\,\textup{d}s}}
					\\&\leq\sum_{S\in \mathcal{S}_h\setminus \partial\Omega}{\vert \jump{D\phi(\nabla u_h^{\textit{c}})n}_S\vert \int_S{\vert e_h- \mathcal{J}_h^{\textit{av}} e_h\vert \,\textup{d}s}}
						\\&\leq c_{\textup{tr}}\sum_{S\in \mathcal{S}_h\setminus \partial\Omega}{\vert \jump{D\phi(\nabla u_h^{\textit{c}})n}_S\vert  \sum_{T\in \mathcal{T}_h;S\subseteq \partial T}{h_T^{-1}\int_T{\vert e_h- \mathcal{J}_h^{\textit{av}} e_h\vert \,\textup{d}s}}} \\&
					\leq \tilde c_{\textit{av}}c_{\textup{tr}} \sum_{S\in \mathcal{S}_h\setminus \partial\Omega}{ \sum_{T\in \mathcal{T}_h;S\subseteq \partial T}{\int_{\omega_T}{\vert \jump{D\phi(\nabla u_h^{\textit{c}})n}_S\vert\vert \nabla_ he_h\vert\,\textup{d}x}}}.
				\end{aligned}	\label{rem:residual.4}
				\end{align}
				Then, proceeding as for \cite[(3.8)--(3.10)]{DK08}, up to obvious adjustments, in particular,~\mbox{using}~for every $T\in \mathcal{T}_h$, in the patch $\omega_T$, the $\varepsilon$\hspace*{-0.1mm}--\hspace*{-0.1mm}Young \hspace*{-0.1mm}inequality \hspace*{-0.1mm}(cf. \hspace*{-0.1mm}\eqref{eq:eps-young}) \hspace*{-0.1mm}for \hspace*{-0.1mm}the \hspace*{-0.1mm}shifted \hspace*{-0.1mm}$N$--\hspace*{-0.1mm}function \hspace*{-0.1mm}$\smash{\varphi_{\vert\nabla u_h^{\textit{c}}(T)\vert}:\mathbb{R}_{\ge 0}\to \mathbb{R}_{\ge 0}}$, cf. \mbox{Appendix}~\ref{subsec:auxiliary} or \cite[Remark 5]{DK08}, defined by 
				\begin{align*}
					\varphi_{\vert\nabla u_h^{\textit{c}}(T)\vert}(t)\vcentcolon=\int_0^t{(\vert\nabla u_h^{\textit{c}}(T)\vert+s )^{p-2} s\,\textup{d}s}\quad\textup{ for all }t\ge 0,
				\end{align*}
				and $(\varphi_{\vert \nabla u_h^{\textit{c}}(T)\vert})^*(\vert \jump{D\phi(\nabla u_h^{\textit{c}})n}_S\vert)\sim\vert \jump{F(\nabla u_h^{\textit{c}})}_S\vert^2$ on $S$ for~all~${S\in \mathcal{S}_h\setminus\partial\Omega}$ with $S\subseteq \partial  T$ (cf.~\hspace*{-0.1mm}\cite[\hspace*{-0.1mm}Cor. \!6]{DK08}), \hspace*{-0.1mm}where \hspace*{-0.1mm}we \hspace*{-0.1mm}for \hspace*{-0.1mm}any \hspace*{-0.1mm}$T\!\in\! \mathcal{T}_h$ \hspace*{-0.1mm}write \hspace*{-0.1mm}$\nabla u_h^{\textit{c}}(T)$ \hspace*{-0.1mm}to \hspace*{-0.1mm}indicate~\hspace*{-0.1mm}that~\hspace*{-0.1mm}the~\hspace*{-0.1mm}shift~\hspace*{-0.1mm}on~\hspace*{-0.1mm}the~\hspace*{-0.1mm}whole patch $\omega_T$ depends on the value of $\nabla u_h^{\textit{c}}$ on the triangle $T$~and~where~${(\varphi_{\vert \nabla u_h^{\textit{c}}(T)\vert})^*\hspace*{-0.1em}:\hspace*{-0.1em}\mathbb{R}_{\ge 0}\hspace*{-0.1em}\to\hspace*{-0.1em} \mathbb{R}_{\ge 0}}$  denotes the Fenchel conjugate to $\varphi_{\vert\nabla u_h^{\textit{c}}(T)\vert}:\mathbb{R}_{\ge 0}\to\mathbb{R}_{\ge 0}$, for any $\varepsilon>0$, we conclude~that
				\begin{align}
					I_h^1&\leq   c_{\textit{av}}c_{\textup{tr}}\!\! \sum_{S\in \mathcal{S}_h\setminus \partial\Omega}{\sum_{T\in \mathcal{T}_h;S\subseteq \partial T}{\int_{\omega_T}{\!c_\varepsilon\,(\varphi_{\vert \nabla u_h^{\textit{c}}(T)\vert})^*(\vert \jump{D\phi(\nabla u_h^{\textit{c}})n}_S\vert)\!+\!\varepsilon\,\varphi_{\vert\nabla u_h^{\textit{c}}(T)\vert}(\vert \nabla_ he_h\vert)\,\textup{d}x}}}\notag
						\\[-1mm]
						&
						\leq   c_{\textit{av}}c_{\textup{tr}} c_\varepsilon\,\sum_{S\in \mathcal{S}_h\setminus \partial\Omega}{\eta_{J,S}^2(u^{c}_h)}+ \tilde c_{\textit{av}}c_{\textup{tr}} \varepsilon \sum_{T\in \mathcal{T}_h}{ \int_{\omega_T}{\varphi_{\vert\nabla u_h^{\textit{c}}(T)\vert}(\vert \nabla_ he_h\vert)\,\textup{d}x}}.
					\label{rem:residual.5}
				\end{align}
				Using \hspace*{-0.1mm}the \hspace*{-0.1mm}$\varepsilon$--\!Young \hspace*{-0.1mm}inequality \hspace*{-0.1mm}(cf. \hspace*{-0.5mm}\eqref{eq:eps-young}), \hspace*{-0.1mm}$(\vert \nabla u^{c}_h\vert^{p-1}+h_T\vert f_h\vert)^{p'-2}h_T^2\vert f_h\vert^2\!\sim\! (\varphi_{\vert \nabla u_h^{\textit{c}}\vert})^*(h_T\vert f_h\vert)$ in \hspace*{-0.1mm}$T$ \hspace*{-0.1mm}for \hspace*{-0.1mm}all \hspace*{-0.1mm}$T\in \mathcal{T}_h$  \hspace*{-0.1mm}(uniformly \hspace*{-0.1mm}in \hspace*{-0.1mm}$h>0$, \hspace*{-0.1mm}cf. \hspace*{-0.1mm}\cite[\hspace*{-0.1mm}(2.6)]{DK08}) \hspace*{-0.1mm}and \hspace*{-0.1mm}Corollary~\hspace*{-0.1mm}\ref{cor:n-function},~\hspace*{-0.1mm}for~\hspace*{-0.1mm}any~\hspace*{-0.1mm}${\varepsilon>0}$,~\hspace*{-0.1mm}we~\hspace*{-0.1mm}get
				\begin{align}
					\begin{aligned}
						I_h^2&\leq \sum_{T\in \mathcal{T}_h}{\int_T{c_\varepsilon\,(\varphi_{\vert \nabla u_h^{\textit{c}}\vert})^*(h_T\vert f_h\vert)+\varepsilon\,\varphi_{\vert \nabla u_h^{\textit{c}}\vert }\big(h_T^{-1}\vert e_h-\mathcal{J}_h^{\textit{av}}e_h\vert  \big)\,\textup{d}x}}\\[-1mm]&\leq 
						 c_\varepsilon\,\sum_{T\in \mathcal{T}_h}{\eta_{E,T}^2(u^{c}_h)}+ \tilde c_{\textit{av}}\varepsilon\sum_{T\in \mathcal{T}_h}{\int_{\omega_T}{\varphi_{\vert \nabla u_h^{\textit{c}}(T)\vert }(\vert \nabla_he_h\vert  )\,\textup{d}x}}.
					\end{aligned}	\label{rem:residual.6}
				\end{align}
				The  $\varepsilon$--\!Young inequality (cf. \eqref{eq:eps-young}), ${(\varphi_{\vert\nabla u_h^{\textit{c}}\vert })^*(\vert D\phi(\nabla u_h^{\textit{c}})\hspace*{-0.1em}-\hspace*{-0.1em}D\phi(\nabla u)\vert )\!\sim\! \vert F(\nabla u_h^{\textit{c}})\hspace*{-0.1em}-\hspace*{-0.1em} F(\nabla u)\vert^2}$ in \hspace*{-0.1mm}$T$ \hspace*{-0.1mm}for \hspace*{-0.1mm}all \hspace*{-0.1mm}$T\hspace*{-0.1em}\in\hspace*{-0.1em} \mathcal{T}_h$  \hspace*{-0.1mm}(uniformly \hspace*{-0.1mm}in \hspace*{-0.1mm}$h\hspace*{-0.1em}>\hspace*{-0.1em}0$, \hspace*{-0.1mm}cf. \hspace*{-0.1mm}\cite[\hspace*{-0.1mm}Cor. \!6]{DK08}), 
				\hspace*{-0.1mm}and \hspace*{-0.1mm}Corollary~\hspace*{-0.1mm}\ref{cor:n-function},~\hspace*{-0.1mm}for~\hspace*{-0.1mm}any~\hspace*{-0.1mm}${\varepsilon\hspace*{-0.1em}>\hspace*{-0.1em}0}$,~\hspace*{-0.1mm}yield
				\begin{align}
					\begin{aligned}
					I_h^3&\leq \sum_{T\in \mathcal{T}_h}{\int_T{c_\varepsilon\,(\varphi_{\vert\nabla u_h^{\textit{c}}\vert })^*(\vert D\phi(\nabla u_h^{\textit{c}})-D\phi(\nabla u)\vert )+\varepsilon\,\varphi_{\vert\nabla u_h^{\textit{c}}\vert }(\vert \nabla \mathcal{J}_h^{\textit{av}}e_h\vert )\,\textup{d}x}}\\[-1mm]&\leq 
					c_\varepsilon\,\rho_I^2(u_h^{\textit{c}},u)
					+ \tilde c_{\textit{av}}\varepsilon\sum_{T\in \mathcal{T}_h}{\int_{\omega_T}{\varphi_{\vert\nabla u_h^{\textit{c}}(T)\vert }(\vert \nabla_h e_h\vert )\,\textup{d}x}}.
				\end{aligned}	\label{rem:residual.7}
				\end{align}
				Proceeding as in \cite[p. 9 \& 10]{DK08}, we obtain a constant $c>0$ such that
				\begin{align}
					\hspace*{-1mm}\sum_{T\in \mathcal{T}_h}{\int_{\omega_T}{\varphi_{\vert \nabla u_h^{\textit{c}}(T)\vert }(\vert \nabla_he_h\vert  )\,\textup{d}x}}\leq c \sum_{T\in \mathcal{T}_h}{\int_T{\varphi_{\vert \nabla u_h^{\textit{c}}\vert }(\vert \nabla_he_h\vert  )\,\textup{d}x}}+c\!\sum_{S\in \mathcal{S}_h\setminus \partial\Omega}{\eta_{J,S}^2(u^{c}_h)}.	\hspace*{-1mm}\label{rem:residual.8}
				\end{align}
				Then, combining \eqref{eq:rel-eff} and \eqref{rem:residual.3}--\eqref{rem:residual.8}, for any $\varepsilon>0$, we conclude that
				\begin{align*}
					\begin{aligned}
				\int_{\Omega}{\big( D\phi(\nabla u_h^{\textit{c}})-D\phi(\nabla_h u_h^{\textit{cr}}) \big)\cdot\nabla_he_h \,\textup{d}x}\leq c_\varepsilon\,\eta_{\textit{res},h}^2(u^{\textit{c}}_h)+\varepsilon\sum_{T\in \mathcal{T}_h}{\int_T{\varphi_{\vert \nabla u_h^{\textit{c}}\vert }(\vert \nabla_he_h\vert  )\,\textup{d}x}}.
			\end{aligned}
				\end{align*}
				Resorting the reconstruction formula \eqref{eq:dirichlet_discrete_reconstruction} and \cite[Lemma 3]{DK08}, we obtain a constant $c>0$  such that 
				for~every~${T\in \mathcal{T}_h}$, we deduce that
				\begin{align*}
						\int_T{\big(D\phi^*(z_h^{\textit{rt}})-D\phi^*(\Pi_hz_h^{\textit{rt}})\big)\cdot(z_h^{\textit{rt}}-\Pi_hz_h^{\textit{rt}})\,\textup{d}x}&\leq c
						\int_T{(\varphi_{\vert \nabla_hu_h^{\textit{cr}}\vert})^*(h_T\vert f_h\vert  )\,\textup{d}x}.
				\end{align*}
				Then, a change of shift (cf. \cite[Corollary 28]{DK08}),  for every $\varepsilon>0$, provides a constant   $c_\varepsilon>0$ such that   for~every~${T\in \mathcal{T}_h}$, it holds
				\begin{align*}
					\begin{aligned}
						\int_T{(\varphi_{\vert \nabla_hu_h^{\textit{cr}}\vert})^*(h_T\vert f_h\vert  )\,\textup{d}x}
						\leq c_\varepsilon\,\eta_{E,T}^2(u^{\textit{c}}_h)+\varepsilon \int_T{\varphi_{\vert \nabla u_h^{\textit{c}}\vert }(\vert \nabla_he_h\vert  )\,\textup{d}x}.
					\end{aligned}
				\end{align*}
				Thanks to $\varphi_{\vert \nabla u_h^{\textit{c}}\vert }(\vert \nabla_he_h\vert  )\sim ( D\phi(\nabla u_h^{\textit{c}})-D\phi(\nabla_h u_h^{\textit{cr}}))\cdot(\nabla u_h^{\textit{c}}-\nabla_h u_h^{\textit{cr}})$ in $T$ for all $T\in \mathcal{T}_h$ (uniformly in $h>0$), for $\varepsilon>0$ sufficiently small, we obtain a constant $\tilde c_{\textit{eq}}>0$ such that
				\begin{align}
					\begin{aligned}
						&\int_{\Omega}{\big( D\phi(\nabla u_h^{\textit{c}})-D\phi(\nabla_h u_h^{\textit{cr}}) \big)\cdot(\nabla u_h^{\textit{c}}-\nabla_h u_h^{\textit{cr}}) \,\textup{d}x}\\&\quad+
							\int_{\Omega}{\big(D\phi^*(z_h^{\textit{rt}})-D\phi^*(\Pi_hz_h^{\textit{rt}})\big)\cdot(z_h^{\textit{rt}}-\Pi_hz_h^{\textit{rt}})\,\textup{d}x}
						\leq \tilde c_{\textit{eq}}\,\eta_{\textit{res},h}^2(u^{\textit{c}}_h).
					\end{aligned}\label{rem:residual.2.2}
				\end{align}	
				From \eqref{rem:residual.2.2}, Proposition \ref{prop:post-processing_pDirichlet}, Corollary \ref{cor:post-pocess} (i) \& (iii) and \eqref{eq:rel-eff}~we,~in~turn,~conclude~that $\rho_I^2(u_h^{\textit{c}},u)\leq \eta_h^2(u_h^{\textit{c}},z_h^{\textit{rt}})\leq  \tilde c_{\textit{eq}}\,\eta_{\textit{res},h}^2(u^{\textit{c}}_h)\leq \tilde c_{\textit{eq}}\,c_{\textit{eff}}\,\rho_I^2(u_h^{\textit{c}},u)$, which implies \eqref{eq:equivalence}.
			\end{proof}
		
			\begin{corollary}[Gobal reliability and efficiency]\label{cor:releff}
			Let $\phi\vcentcolon=\smash{\frac{1}{p}}\vert \cdot\vert^p\in C^1(\mathbb{R}^d)$, $p\in (1,\infty)$, and $f=f_h\in \mathcal{L}^0(\mathcal{T}_h)$. Then, there exist constants $ c_{\textit{rel}},c_{\textit{eff}}>0$ such that 
			\begin{align}
				\begin{aligned}
					\smash{c_{\textit{eff}}\, \rho_I^2(u_h^{\textit{c}},u)\leq \eta_h^2(u_h^{\textit{c}},z_h^{\textit{rt}})\leq c_{\textit{rel}}\,\rho_I^2(u_h^{\textit{c}},u).}
				\end{aligned}\label{eq:releff}
			\end{align}
		\end{corollary}
	
		\begin{remark}\label{rmk:2}
			\begin{description}[noitemsep,topsep=1pt,labelwidth=\widthof{\textit{(iii)}},leftmargin=!,font=\normalfont\itshape]
					\item[(i)] The extensions described in Remark \ref{rmk:1} equally apply to Corollary~\ref{cor:releff}.
					
					\item[(ii)] Since \hspace*{-0.1mm}we \hspace*{-0.1mm}have \hspace*{-0.1mm}used \hspace*{-0.1mm}global \hspace*{-0.1mm}arguments, \hspace*{-0.1mm}e.g., \hspace*{-0.1mm}discrete \hspace*{-0.1mm}and \hspace*{-0.1mm}continuous \hspace*{-0.1mm}Euler--Lagrange~\hspace*{-0.1mm}equa-tions, cf. \eqref{rem:residual.3}, and element-wise integration-by-parts, cf. \eqref{rem:residual.4}, it remains~unclear whether the primal-dual a posteriori error estimator $\eta_h^2(u_h^{\textit{c}},z_h^{\textit{rt}})$~and~the~\mbox{residual}~type a posteriori error estimator $\eta_{\textit{res},h}^2(u_h^{\textit{c}})$ are also locally equivalent,~i.e.,~if~there~exists a constant $c_{\textit{eq}}\hspace*{-0.17em}>\hspace*{-0.17em}0$,  such that $\smash{c_{\textit{eq}}^{-1}\, \eta_{h,T}^2(u_h^{\textit{c}},z_h^{\textit{rt}})\hspace*{-0.17em}\leq\hspace*{-0.17em} \eta_{\textit{res},T}^2(u^{\textit{c}}_h)\hspace*{-0.17em}\leq\hspace*{-0.17em} c_{\textit{eq}}\,\eta_{h,T}^2(u_h^{\textit{c}},z_h^{\textit{rt}})}$~for~all~${T\hspace*{-0.17em}\in\hspace*{-0.17em} \mathcal{T}_h}$.
						
					\item[(iii)] As, according to (ii), the local equivalence of  the primal-dual a posteriori~error estimator $\eta_h^2(u_h^{\textit{c}},z_h^{\textit{rt}})$ and the residual type a posteriori error~estimator~$\eta_{\textit{res},h}^2(u_h^{\textit{c}})$ is still open, we cannot refer to \cite{DK08} to infer the convergence of~the~adaptive~algorithm,~cf.~Algorithm~\ref{alg:afem}. In fact, in  \cite{DK08}, it was decisively used that the~residual~type~a~posteriori error estimator $\eta_{\textit{res},h}^2(u_h^{\textit{c}})$  is locally efficient, i.e., there~exists~a~constant~${c_{\textit{eff}}>0}$~such~that for every~$T\in\mathcal{T}_h$, it holds $\smash{c_{\textit{eff}}\,\eta_{\textit{res},T}^2(u^{\textit{c}}_h)\leq \|F(\nabla u^{\textit{c}}_h)-F(\nabla u)\|_{\smash{L^2(T;\mathbb{R}^d)}}^2}$.~This,~in~turn, suggests  to use residual type a posteriori error estimators for adaptive mesh refinement and primal-dual a posteriori~error estimators for error estimation.\vspace*{-0.5mm}
			\end{description}
		\end{remark}
	
	\subsection{A degenerate minimization problem: An optimal design problem}

		\qquad If $p=2$ and $\phi\vcentcolon=\psi\circ \vert\cdot \vert\in C^1(\mathbb{R}^d)$, where $\psi\in C^1(\mathbb{R}_{\ge 0})$  is prescribed~by~the~initial~value ${\psi(0)= 0}$
		and  the derivative  $\psi'\in C^0(\mathbb{R}_{\ge 0})$, for all $t\ge0$ defined by
	\begin{align}
		\psi'(t)\vcentcolon=\begin{cases}
			\mu_2 t&\quad \text{ for }t\in [0,t_1]\\
			\mu_2 t_1&\quad \text{ for }t\in [t_1,t_2]\\
			\mu_1 t&\quad \text{ for }t\in [t_2,+\infty)
		\end{cases},\label{eq:psi}
	\end{align}
	 $0\hspace*{-0.05em}<\hspace*{-0.05em}t_1\hspace*{-0.05em}<\hspace*{-0.05em}t_2$ and $0\hspace*{-0.05em}<\hspace*{-0.05em}\mu_1\hspace*{-0.05em}<\hspace*{-0.05em}\mu_2$ are given parameters such that $t_1\mu_2\hspace*{-0.05em}=\hspace*{-0.05em}t_2\mu_1$,~then~the~\mbox{non-linear} Dirichlet problem \eqref{eq:dirichlet_primal} reduces to the optimal design problem for maximal~torsion~\mbox{stiffness} of \hspace*{-0.1mm}an \hspace*{-0.1mm}infinite \hspace*{-0.1mm}bar \hspace*{-0.1mm}of \hspace*{-0.1mm}a \hspace*{-0.1mm}given \hspace*{-0.1mm}geometry \hspace*{-0.1mm}and \hspace*{-0.1mm}unknown \hspace*{-0.1mm}distribution \hspace*{-0.1mm}of \hspace*{-0.1mm}two~\hspace*{-0.1mm}materials~\hspace*{-0.1mm}of~\hspace*{-0.1mm}\mbox{prescribed} amounts, \hspace*{-0.2mm}a \hspace*{-0.2mm}classical \hspace*{-0.2mm}example \hspace*{-0.2mm}from \hspace*{-0.2mm}topology \hspace*{-0.2mm}optimization, \hspace*{-0.2mm}cf. \hspace*{-1mm}\cite{Che00}. \hspace*{-1mm}The \hspace*{-0.2mm}optimal~\hspace*{-0.2mm}design~\hspace*{-0.2mm}\mbox{problem} \hspace*{-0.1mm}is \hspace*{-0.1mm}a \hspace*{-0.1mm}degenerate \hspace*{-0.1mm}convex \hspace*{-0.1mm}minimization \hspace*{-0.1mm}problem, \hspace*{-0.1mm}i.e., \hspace*{-0.1mm}in \hspace*{-0.1mm}contrast~\hspace*{-0.1mm}to~\hspace*{-0.1mm}the~\hspace*{-0.1mm}\mbox{$p$--\hspace*{-0.1mm}Dirichlet}~\hspace*{-0.1mm}problem, the defining functional \hspace*{-0.1mm}\eqref{eq:dirichlet_primal} \hspace*{-0.1mm}is \hspace*{-0.1mm}not~\hspace*{-0.1mm}strongly~\hspace*{-0.1mm}convex, \hspace*{-0.1mm}but \hspace*{-0.1mm}only~\hspace*{-0.1mm}co-coercive~\hspace*{-0.1mm}since~\hspace*{-0.1mm}the~\hspace*{-0.1mm}\mbox{energy}~\hspace*{-0.1mm}density $\phi\vcentcolon=\psi\circ \vert\cdot \vert\in C^1(\mathbb{R}^d)$, cf. \cite[Proposition 4.2]{BC08}, for every $a,b\in \mathbb{R}^d$~satisfies
	\begin{align}
		\smash{(2\mu_2)^{-1}\vert D\phi(a)-D\phi(b)\vert^2\leq \phi(a)-\phi(b)-D\phi(a)\cdot(a-b).}\label{co-coercive1}
	\end{align}
	If  $u\!\in \!\smash{W^{1,2}_D(\Omega)}$ is minimal for \eqref{eq:dirichlet_primal}, from the finite-dimensional co-coercivity~property~\eqref{co-coercive1}, for every $v\in \smash{W^{1,2}_D(\Omega)}$, we have the following infinite-dimensional~co-coercivity~property
	\begin{align}
		\smash{(2\mu_2)^{-1}\|D\phi(\nabla v)-D\phi(\nabla u)\|_{L^2(\Omega;\mathbb{R}^d)}^2\leq I(v)-I(u).}\label{co-coercive2}
	\end{align}
	It is well-known that minimizers $u\in \smash{W^{1,2}_D(\Omega)}$ of \eqref{eq:dirichlet_primal} are (possibly) non-unique, while from \eqref{co-coercive2} we directly conclude that
	$D\phi(\nabla u)\in L^2(\Omega;\mathbb{R}^d)$ is unique. The co-coercivity~property motivates to define a measure ${\rho_I^2:\smash{W^{1,2}_D(\Omega)}\times \smash{W^{1,2}_D(\Omega)}\to \mathbb{R}_{\ge 0}}$ for the co-coercivity~of~\eqref{eq:dirichlet_primal} by 
	$\rho_I^2(v,w)\vcentcolon=
	\|D\phi(\nabla v)-D\phi(\nabla w)\|_{\smash{L^2(\Omega;\mathbb{R}^d)}}^2$ for all $v,w\in \smash{W^{1,2}_D(\Omega)}$. However,~since~\mbox{explicit} representations
	\hspace*{-0.1mm}of \hspace*{-0.1mm}minimizers \hspace*{-0.1mm}$u\hspace*{-0.1em}\in\hspace*{-0.1em}\smash{W^{1,2}_D(\Omega)}$ \hspace*{-0.1mm}of \hspace*{-0.1mm}\eqref{eq:dirichlet_primal} \hspace*{-0.1mm}for \hspace*{-0.1mm}simple \hspace*{-0.1mm}data, \hspace*{-0.1mm}e.g., \hspace*{-0.1mm}$f\hspace*{-0.1em}=\hspace*{-0.1em}1$,~\hspace*{-0.1mm}are~\hspace*{-0.1mm}rare,~\hspace*{-0.1mm}in~\hspace*{-0.1mm}our experiments, we consider
	$\rho_I^2\!:\!\smash{W^{1,2}_D(\Omega)}\times \smash{W^{1,2}_D(\Omega)}\!\to\! \mathbb{R}_{\ge 0}$,~for~every~${v,w\!\in\! \smash{W^{1,2}_D(\Omega)}}$~defined~by\vspace*{-0.5mm}
	\begin{align}
	\smash{\rho_I^2(v,w)\vcentcolon=I(v)-I(w)},\label{co-coercive3}
	\end{align}
	which has the particular advantage that the exact value $I(u)$ can be approximated resorting to Aitken's  $\delta^2$--process, cf. \cite{Ait26}.
	\newpage		

	\section{Numerical Experiments}\label{sec:numerical_experiments}

	\qquad In this section, we verify our theoretical findings via numerical experiments.~More~preci-sely, we 
	\hspace*{-0.1mm}present\hspace*{-0.1mm} numerical \hspace*{-0.1mm}results \hspace*{-0.1mm}for \hspace*{-0.1mm}the \hspace*{-0.1mm}approximation \hspace*{-0.1mm}of \hspace*{-0.1mm} the \hspace*{-0.1mm}$p$--Dirichlet~\hspace*{-0.1mm}problem~\hspace*{-0.1mm}and~\hspace*{-0.1mm}an~\hspace*{-0.1mm}op-timal design problem by deploying adaptive mesh~refinements~on~the~basis~of~the~\mbox{trapezoidal} primal-dual a posteriori error estimators $(\widehat{\eta}_{h,T}^2(\tilde{u}_h, z_h^{\textit{rt}}))_{T\in \mathcal{T}_h}$, cf. \eqref{eq:hat2}.
	
	Before \hspace*{-0.15mm}we \hspace*{-0.15mm}present \hspace*{-0.15mm}the \hspace*{-0.15mm}computational \hspace*{-0.15mm}experiments, \hspace*{-0.15mm}we \hspace*{-0.15mm}briefly \hspace*{-0.15mm}outline \hspace*{-0.15mm}the \hspace*{-0.15mm}general~\hspace*{-0.15mm}details~\hspace*{-0.15mm}of \hspace*{-0.15mm}our  \hspace*{-0.15mm}implementations. \hspace*{-0.15mm}In \hspace*{-0.15mm}general, \hspace*{-0.15mm}we \hspace*{-0.15mm}follow \hspace*{-0.15mm}the~\hspace*{-0.15mm}adaptive~\hspace*{-0.15mm}\mbox{algorithm},~\hspace*{-0.15mm}cf.~\hspace*{-0.15mm}\cite{Vee02,Ste07,DK08,CKNS08}:
	\begin{algorithm}[AFEM]\label{alg:afem}
		Let $\varepsilon_{\textup{STOP}}>0$, $\theta\in (0,1)$ and  $\mathcal{T}_0$ a conforming initial  triangu-lation of $\Omega$. Then, for $k\ge 0$:
	\begin{description}[noitemsep,topsep=1pt,labelwidth=\widthof{\textit{('Estimate')}},leftmargin=!,font=\normalfont\itshape]
		\item[('Solve')]\hypertarget{Solve}{}
		Compute both a conforming approximation  ${\tilde{u}_k\in \smash{\mathcal{S}^1_D(\mathcal{T}_k)}}$ and a minimizer $u_k^{\textit{cr}}\in \smash{\mathcal{S}^{1,\textit{cr}}_D(\mathcal{T}_k)}$ of \eqref{discrete_primal}. Post-process $u_k^{\textit{cr}}\in \smash{\mathcal{S}^{1,\textit{cr}}_D(\mathcal{T}_k)}$ to obtain~a~\mbox{maximizer} $z_k^{\textit{rt}}\in \smash{\mathcal{R}T^0_N(\mathcal{T}_k)}$ of \eqref{discrete_dual}.
		\item[('Estimate')]\hypertarget{Estimate}{} Compute the refinement indicators $(\widehat{\eta}^2_{k,T}(\tilde{u}_k,z_k^{\textit{rt}}))_{T\in \mathcal{T}_k}$. If ${\widehat{\eta}^2_k(\tilde{u}_k,z_k^{\textit{rt}})\hspace*{-0.1em}\leq \hspace*{-0.1em}\varepsilon_{\textup{STOP}}}$, then \textup{STOP}.
		\item[('Mark')]\hypertarget{Mark}{}  Choose a minimal (in terms of cardinality) subset $\mathcal{M}_k\subseteq\mathcal{T}_k$ such that
		\begin{align*}
			\sum_{T\in \mathcal{M}_k}{\widehat{\eta}_{k,T}^2(\tilde{u}_k, z_k^{\textit{rt}})}\ge \theta^2\sum_{T\in \mathcal{T}_k}{\widehat{\eta}_{k,T}^2(\tilde{u}_k, z_k^{\textit{rt}})}.
		\end{align*}
		\item[('Refine')]\hypertarget{Refine}{} Perform \hspace*{-0.1mm}a \hspace*{-0.1mm}(minimal) \hspace*{-0.1mm}conforming \hspace*{-0.1mm}refinement \hspace*{-0.1mm}of \hspace*{-0.1mm}$\mathcal{T}_k$ \hspace*{-0.1mm}to \hspace*{-0.1mm}obtain \hspace*{-0.1mm}$\mathcal{T}_{k+1}$~\hspace*{-0.1mm}such~\hspace*{-0.1mm}that~\hspace*{-0.1mm}each $T\in \mathcal{M}_k$  is refined in $\mathcal{T}_{k+1}$, i.e., each $T\in \mathcal{M}_k$ and each~of~its~sides~contains~a node of $\mathcal{T}_{k+1}$ in its interior. Increase~$k\to k+1$~and~continue~with~('Solve').
	\end{description}
	\end{algorithm}

	\begin{remark}
			\begin{description}[noitemsep,topsep=1pt,labelwidth=\widthof{\textit{(iii)}},leftmargin=!,font=\normalfont\itshape]
				\item[(i)] The computation of a conforming approximation  $\tilde{u}_k\in\smash{ \mathcal{S}^1_D(\mathcal{T}_k)}$ and a minimizer $u_k^{\textit{cr}}\in \smash{\mathcal{S}^{1,\textit{cr}}_D(\mathcal{T}_k)}$ of \eqref{discrete_primal} in (\hyperlink{Solve}{'Solve'}) is adjusted to the respective problem.
				\item[(ii)] The reconstruction of a maximizer $z_k^{\textit{rt}}\in \smash{\mathcal{R}T^0_N(\mathcal{T}_k)}$ of \eqref{discrete_dual} in (\hyperlink{Solve}{'Solve'}) is based on explicit representation formulas and does not entail further computational costs.
				\item[(iii)] If not otherwise specified, we employ the parameter $\theta=\smash{\frac{1}{2}}$ in (\hyperlink{Estimate}{'Estimate'}).
				\item[(iv)] To find the minimal (in terms of cardinality) set $\mathcal{M}_k\subseteq \mathcal{T}_k$ in (\hyperlink{Mark}{'Mark'}), we~deploy~the  \textit{D\"orfler} marking strategy, cf. \cite{Doe96}.
				\item[(v)] The \hspace*{-0.1mm}(minimal) \hspace*{-0.1mm}conforming \hspace*{-0.1mm}refinement \hspace*{-0.1mm}of \hspace*{-0.1mm}$\mathcal{T}_k$ \hspace*{-0.1mm}with \hspace*{-0.1mm}respect \hspace*{-0.1mm}to \hspace*{-0.1mm}$\mathcal{M}_k$~\hspace*{-0.1mm}in~\hspace*{-0.1mm}(\hyperlink{Refine}{'Refine'})~\hspace*{-0.1mm}is~\hspace*{-0.1mm}\mbox{obtained} deploying the \textit{red}--\textit{green}--\textit{blue}--refinement algorithm.
				
			\end{description}
	\end{remark}
	
	All \hspace*{-0.1mm}experiments \hspace*{-0.1mm}were \hspace*{-0.1mm}conducted \hspace*{-0.1mm}using \hspace*{-0.1mm}the \hspace*{-0.1mm}finite \hspace*{-0.1mm}element \hspace*{-0.1mm}software \hspace*{-0.1mm}package~\hspace*{-0.1mm}\mbox{\textsf{FEniCS}},~\hspace*{-0.1mm}cf.~\hspace*{-0.1mm}\cite{LW10}. All graphics are generated using the \textsf{Matplotlib} library, cf. \cite{Hun07}.
	
	\subsection{$p$--Dirichlet problem}\label{subsec:expDirichlet}
	
	\qquad \!We \hspace*{-0.2mm}examine \hspace*{-0.15mm}the \hspace*{-0.2mm}$p$--\hspace*{-0.15mm}Dirichlet \hspace*{-0.2mm}problem \hspace*{-0.2mm}with \hspace*{-0.2mm}prescribed \hspace*{-0.2mm}in-homogeneous \hspace*{-0.2mm}Dirichlet~\hspace*{-0.2mm}\mbox{boundary} data on an $L$--shaped domain. More precisely, we let $\Omega\hspace*{-0.15em}\vcentcolon=\hspace*{-0.15em} \left(-1,1\right)^2 \setminus ([0,1]\hspace*{-0.1em}\times\hspace*{-0.1em} [-1,0])$,~${\Gamma_D \hspace*{-0.15em}\vcentcolon= \hspace*{-0.15em}\partial\Omega}$, $\Gamma_N\hspace*{-0.1em}\vcentcolon=\hspace*{-0.1em}\emptyset$, $\phi\hspace*{-0.1em}\vcentcolon=\smash{\hspace*{-0.1em}\frac{1}{p}\vert \cdot\vert^p\hspace*{-0.1em}\in\hspace*{-0.1em} C^1(\mathbb{R}^2)}$, $p\hspace*{-0.1em}\in\hspace*{-0.1em}(1,\infty)$, and prescribe in-homogeneous Dirichlet~boundary data $u_D = u|_{\partial\Omega}\in L^p(\partial\Omega)$ through restriction of the unique exact solution $u\in W^{1,p}(\Omega)$, in polar coordinates, for every $(r,\theta)^\top\in (0,\infty)\times (0,2\pi)$ defined by 
	\begin{align*}
		\smash{u(r, \theta) \vcentcolon= r^\delta \sin(\delta \theta),}
	\end{align*}
	to the boundary $\partial\Omega$. The particular choice of $\delta>0$ will be specified later in dependence of the choice of $p\in (1,\infty)$. Then, the corresponding non-smooth right-hand side $f\in L^{p'}(\Omega)$, in polar coordinates, is for every $\smash{(r,\theta)^\top}\in (0,\infty)\times (0,2\pi)$ defined by
	\begin{align*}
		 \smash{f(r, \theta) \vcentcolon= -(2-p)\delta^{p-1}(1 - \delta)r^{(\delta-1)(p-1)-1} \sin(\delta \theta).}
	\end{align*}
	For $p\in (1,\infty)$, we let $\delta = \frac{6}{5}(1-\frac{1}{p})>0$. Then, we have that $u \in  W^{1,p}(\Omega)$,~but~${u \notin W^{2,p}(\Omega)}$. 
	
	\hspace*{-5.5mm}The \hspace*{-0.1mm}initial \hspace*{-0.1mm}triangulation \hspace*{-0.1mm}$\mathcal{T}_0$ \hspace*{-0.1mm}consists \hspace*{-0.1mm}of \hspace*{-0.1mm}96 \hspace*{-0.1mm}elements \hspace*{-0.1mm}and \hspace*{-0.1mm}65 \hspace*{-0.1mm}vertices.~\hspace*{-0.1mm}We~\hspace*{-0.1mm}use~\hspace*{-0.1mm}${f_k\!\vcentcolon=\!\Pi_kf\!\in\! \mathcal{L}^0(\mathcal{T}_k)}$ for all $k=0,\dots,19$.
	In what follows, for every $k=0,\dots,19$, we denote by
	 $u_k^{\textit{c}} \in  \smash{\mathcal{S}^1_D(\mathcal{T}_k)}$ the minimizer of $\smash{I_k^{\textit{c}}\hspace*{-0.12em}\vcentcolon=\hspace*{-0.12em}I|_{ \mathcal{S}^1_D(\mathcal{T}_k)}\hspace*{-0.12em}:\hspace*{-0.12em} \mathcal{S}^1_D(\mathcal{T}_k)\hspace*{-0.12em}\to\hspace*{-0.12em} \mathbb{R}}$ and~by~$u_k^{\textit{cr}}\hspace*{-0.12em} \in \hspace*{-0.12em} \smash{\mathcal{S}^{1,\textit{cr}}_D(\mathcal{T}_k)}$~the~\mbox{minimizer}~of~\eqref{eq:dirichlet_discrete_primal}.~Both minimizers are computed 
	using the   Newton~line~search algorithm of \textsf{PETSc}, cf. \cite{PETSc19}, with an absolute tolerance of about $\tau_{\textit{abs}}=1\textrm{e}{-}8$ and a relative tolerance~of~about~${\tau_{\textit{rel}}=1\textrm{e}{-}10}$. The linear system emerging~in each Newton step is solved using the generalized minimal residual method \mbox{(GMRES)}. 
	Globally convergent  semi-implicit discretizations~of~the~respective~$L^2$--gradient flows yield comparable results, but terminate~significantly~slower.~Then,~for~\mbox{every} ${k= 0,\dots,19}$,~via~post-processing $u_k^{\textit{cr}} \in  \smash{\mathcal{S}^{1,\textit{cr}}_D(\mathcal{T}_k)}$, we obtain a maximizer $z_k^{\textit{rt}}\in \smash{\mathcal{R}T^0(\mathcal{T}_k)}$~of \eqref{eq:dirichlet_discrete_dual}~\hspace*{-0.1mm}by~\hspace*{-0.1mm}resorting~\hspace*{-0.1mm}to~\hspace*{-0.1mm}the \hspace*{-0.1mm}reconstruction \hspace*{-0.1mm}formula \hspace*{-0.1mm}\eqref{eq:dirichlet_discrete_reconstruction}. 
	\hspace*{-0.1mm}In~\hspace*{-0.1mm}Figure~\hspace*{-0.1mm}\ref{fig:pDirichletPD},~\hspace*{-0.1mm}for~\hspace*{-0.1mm}every~\hspace*{-0.1mm}${k=0,\dots,19}$~\hspace*{-0.1mm}and $\tilde{u}_{h_k}\!\vcentcolon=\!u_k^{\textit{c}}\!\in\! \smash{\mathcal{S}^1_D(\mathcal{T}_k)}$, \hspace*{-0.2mm}the \hspace*{-0.2mm}square \hspace*{-0.2mm}root \hspace*{-0.2mm}of \hspace*{-0.2mm}the \hspace*{-0.2mm}trapezoidal \hspace*{-0.2mm}primal-dual \hspace*{-0.2mm}a \hspace*{-0.2mm}posteriori~\hspace*{-0.2mm}error~\hspace*{-0.2mm}\mbox{estimator}
	\begin{align}
		\begin{aligned}
		\widehat{\eta}^2_k(\tilde{u}_{h_k},z_k^{\textit{rt}})&=\int_{\Omega}{\phi(\nabla \tilde{u}_{h_k})-\Pi_{h_k} z_k^{\textit{rt}}\cdot(\nabla \tilde{u}_{h_k}-\nabla_{h_k} u_k^{\textit{cr}})+\phi(\nabla_{h_k} \tilde{u}_{h_k})\,\textup{d}x}
		\\&\quad+\int_{\Omega}{\widehat{I}_{h_k}\big[\phi^*(z_k^{\textit{rt}})\big]-\phi^*(\Pi_{h_k}z_k^{\textit{rt}})\,\textup{d}x},
	\end{aligned}\label{eq:pd-estimator}
	\end{align}
	 and square root of the error on the left-hand side of the estimate in Proposition \ref{prop:post-processing_pDirichlet},~i.e.,~~~~~~~~
	\begin{align}
		\rho^2_I(u,\tilde{u}_{h_k})=\|F (\nabla u) - F (\nabla \tilde{u}_{h_k})\|_{L^2(\Omega;\mathbb{R}^2)}^2,\label{eq:error-quantity}
	\end{align}
	are plotted versus the number of degrees of freedom $N_k \vcentcolon= \textup{card}(\mathcal{N}_{h_k})$~in~a~\mbox{$\log\log$--plot}. In it, one \hspace*{-0.15mm}clearly \hspace*{-0.15mm}observes \hspace*{-0.15mm}that \hspace*{-0.15mm}mesh \hspace*{-0.15mm}adaptivity \hspace*{-0.15mm}yields \hspace*{-0.15mm}the \hspace*{-0.15mm}quasi-optimal
	\hspace*{-0.15mm}convergence~\hspace*{-0.15mm}rate~\hspace*{-0.15mm}$\smash{h_k\! \sim\! N_k^{\smash{-\frac{1}{2}}}}$. In particular, for every $k\hspace*{-0.1em}=\hspace*{-0.1em}0,\dots,19$, the trapezoidal primal-dual~a~posteriori~error~\mbox{estimator} $\widehat{\eta}^2_k(u_k^{\textit{c}},z_k^{\textit{rt}})$ \hspace*{-0.1mm}defines \hspace*{-0.1mm}a \hspace*{-0.1mm}reliable \hspace*{-0.1mm}upper \hspace*{-0.1mm}bound \hspace*{-0.1mm}for \hspace*{-0.1mm}the \hspace*{-0.1mm}error \hspace*{-0.1mm}quantity \hspace*{-0.1mm}$\rho^2_I(u,u_k^{\textit{c}})$.~\hspace*{-0.1mm}Also~\hspace*{-0.1mm}note~\hspace*{-0.1mm}that~\hspace*{-0.1mm}data approximation~terms such as, e.g., in Remark \ref{rem:general_pd_estimator0}~(v.b)~are~disregarded~in~all~experiments.~~~~~
		\begin{figure}[H]\vspace*{-0.5cm}
				\hspace*{-0.25cm}\includegraphics[width=14cm]{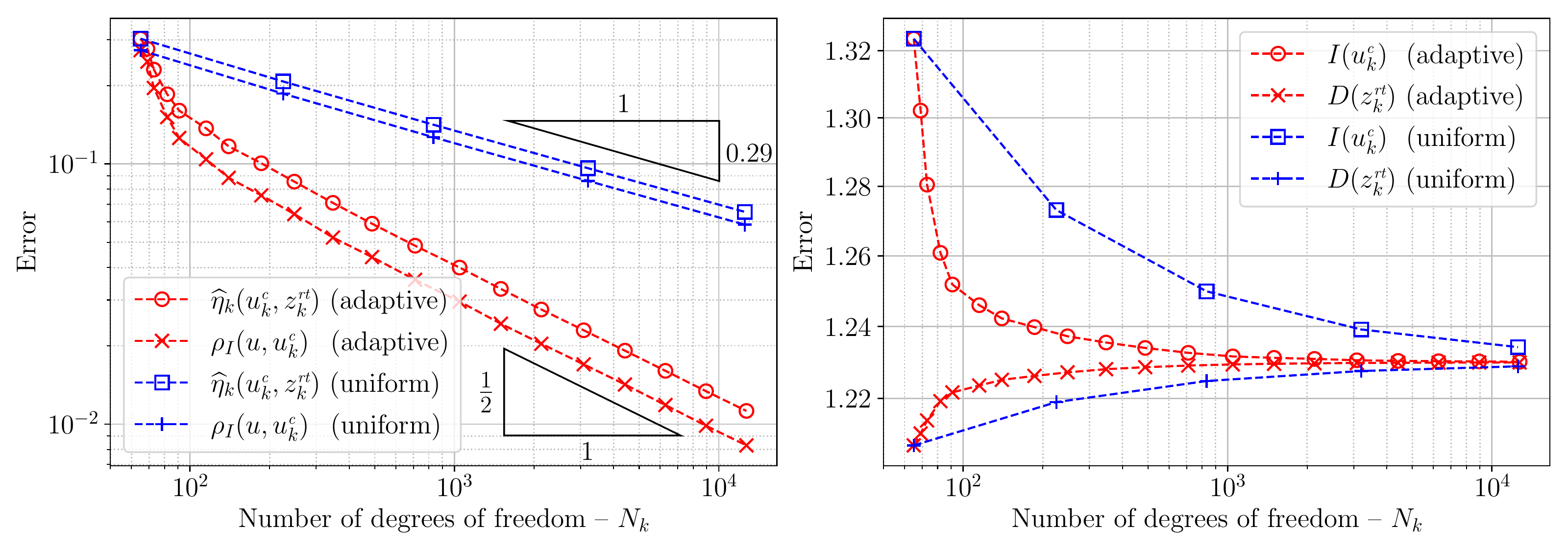}\vspace*{-0.1cm}
			\hspace*{-0.25cm}\includegraphics[width=14cm]{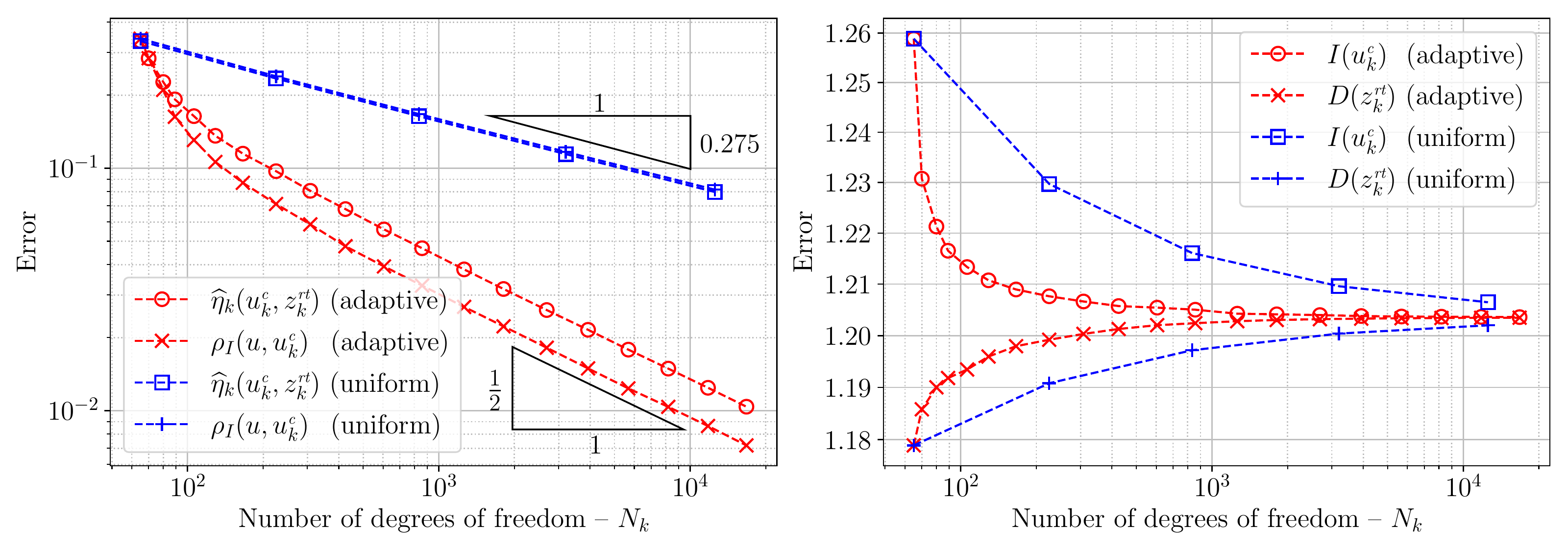}\vspace*{-0.25cm}
		\caption{The trapezoidal primal-dual a posteriori error estimators $\widehat{\eta}^2_k(u_k^{\textit{c}},z_k^{\textit{rt}})$,~cf.~\eqref{eq:pd-estimator}, and \hspace*{-0.25mm}the \hspace*{-0.25mm}error \hspace*{-0.25mm}quantities  $\rho^2_I(u,u_k^{\textit{c}})$, \hspace*{-0.25mm}cf. \!\eqref{eq:error-quantity}, \hspace*{-0.25mm}(left) \hspace*{-0.25mm}as \hspace*{-0.25mm}well \hspace*{-0.25mm}as \hspace*{-0.25mm}the \hspace*{-0.25mm}primal~\hspace*{-0.25mm}energies~\hspace*{-0.25mm}$I(u_k^{\textit{c}})$,~\hspace*{-0.25mm}cf.~\!\eqref{eq:dirichlet_primal}, and dual energies 
			  $D (z_k^{\textit{rt}})$, \eqref{eq:dirichlet_dual}, (right) for uniform and adaptive mesh refinement~for $k=0,\dots,4$ and $k=0,\dots,19$, resp. Top: $p$--Dirichlet problem with $p\hspace*{-0.1em}=\hspace*{-0.1em}1.6$. Bottom:
			$p$--Dirichlet~problem~with~${p\hspace*{-0.1em}=\hspace*{-0.1em}1.2}$.\vspace*{-1cm}}
		\label{fig:pDirichletPD}
	\end{figure}
	
	On the right-hand side of Figure \ref{fig:pDirichletPD}, we displayed the energy curves for
	$I(u_k^{\textit{c}})$~and~$D (z_k^{\textit{rt}})$, $k=0,\dots,19$, resp.~The~primal and dual energies converge to the optimal value $I(u) = D(z)$ and the primal-dual gap $I(u_k^{\textit{c}})-D (z_k^{\textit{rt}})$, $k=0,\dots,19$, converges~to~zero~as~${N_k \to \infty}$, and even at a linear rate, when local mesh refinement is used. 

	In \hspace*{-0.15mm}Figure \hspace*{-0.15mm}\ref{fig:pDirichletResidual}, \hspace*{-0.15mm}for \hspace*{-0.15mm}every \hspace*{-0.15mm}$k\!=\!0,\dots,19$, \hspace*{-0.15mm}we \hspace*{-0.15mm}compare \hspace*{-0.15mm}the \hspace*{-0.15mm}trapezoidal \hspace*{-0.15mm}primal-dual~a~\hspace*{-0.15mm}\mbox{posteriori}~er-ror \hspace*{-0.15mm}estimator  \hspace*{-0.15mm}$\widehat{\eta}^2_k(u_k^{\textit{c}},z_k^{\textit{rt}})$ \hspace*{-0.15mm}with \hspace*{-0.15mm}the \hspace*{-0.15mm}residual \hspace*{-0.15mm}a \hspace*{-0.15mm}posteriori \hspace*{-0.15mm}error \hspace*{-0.15mm}estimator~\hspace*{-0.15mm}${\eta_{\textit{res},k}^2(u_k^{\textit{c}})\!\vcentcolon=\!\eta_{\textit{res},h_k}^2(u_k^{\textit{c}})}$, for the $p$--Dirichlet problem with
	in-homogeneous Dirichlet boundary~data~on the $L$--shaped domain $\Omega$ for $p =  1.6$
	and $p = 1.2$. In it, 
	one~observes~that both estimators decay at the
	same quasi-optimal rate $\mathcal{O}(N^{\smash{-\frac{1}{2}}}_k)$. The experiments confirm that~$\widehat{\eta}^2_k(u_k^{\textit{c}},z_k^{\textit{rt}})$~and~$\smash{\eta_{\textit{res},k}^2(u_k^{\textit{c}})}$,~up to an overestimation of $\smash{\eta_{\textit{res},k}^2(u_k^{\textit{c}})}$, behave identically supporting the findings~of~\mbox{Theorem}~\ref{thm:equivalences}.~~~

		\begin{figure}[H]\vspace*{-0.25cm}
		\hspace*{-0.25cm}\includegraphics[width=14cm]{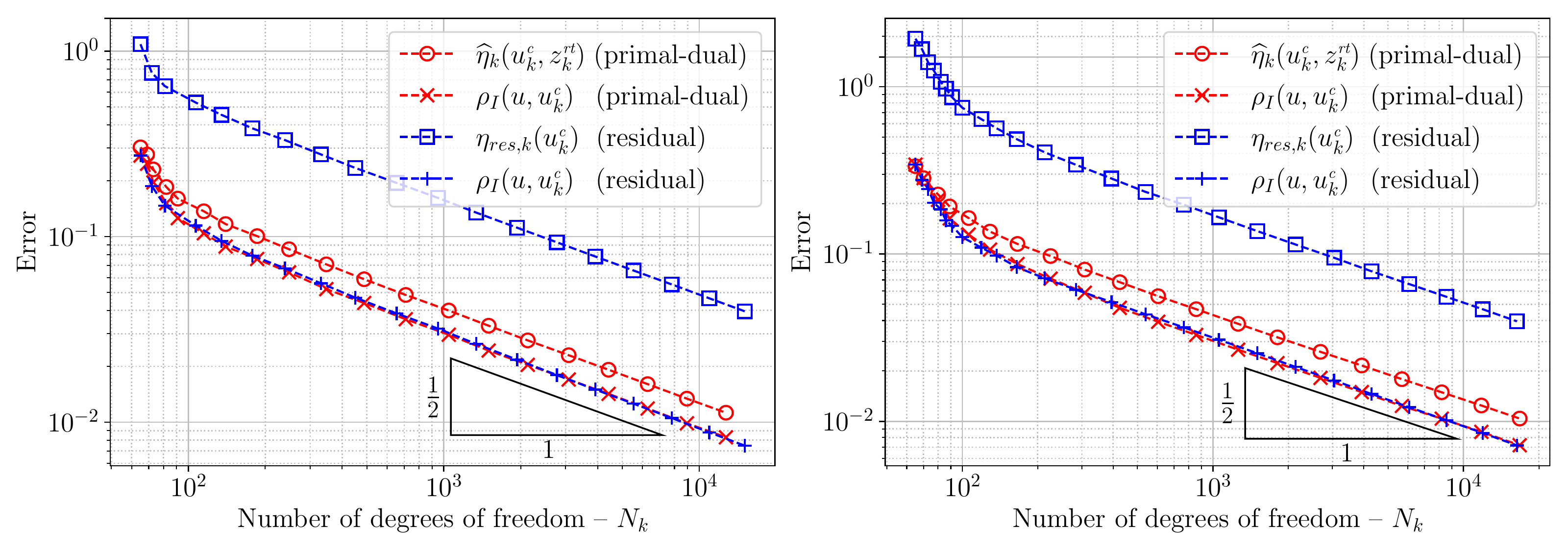}\vspace*{-0.25cm}
		\caption{\hspace*{-0.25mm}The \hspace*{-0.25mm}trapezoidal \hspace*{-0.25mm}primal-dual \hspace*{-0.25mm}a \hspace*{-0.25mm}posteriori \hspace*{-0.25mm}error \hspace*{-0.25mm}estimators \hspace*{-0.25mm}$\widehat{\eta}^2_h(u_k^{\textit{c}},z_k^{\textit{rt}})$,~\hspace*{-0.25mm}cf.~\hspace*{-0.25mm}\eqref{eq:pd-estimator},~\hspace*{-0.25mm}and  residual type a posteriori  error estimators $\eta^2_{\textit{res},k}(u_k^{\textit{c}})$, cf.~\eqref{rem:residual.1.1},~for~$k=0,\dots,19$. Left: $p$--Dirichlet problem~with~${p=1.6}$. Right:
			$p$--Dirichlet problem~with~${p=1.2}$.\vspace*{-0.25cm}}
		\label{fig:pDirichletResidual}
	\end{figure}

	In Figure \ref{fig:pDirichletAveraging}, for every $k=0,\dots,19$ and $\tilde{u}_{h_k}\vcentcolon=\smash{\mathcal{J}_{h_k}^{\textit{av}}}u_k^{\textit{cr}}\in \mathcal{S}^1_D(\mathcal{T}_k)$, the square~root~of~the primal-dual a posteriori error estimator $\widehat{\eta}^2_k(\smash{\mathcal{J}_{h_k}^{\textit{av}}}u_k^{\textit{cr}},z_k^{\textit{rt}})$,  cf. \eqref{eq:pd-estimator},  
	and~of~the~error~quantity $\rho^2_I(	\hspace*{-0.2mm}u,	\hspace*{-0.2mm}\smash{\mathcal{J}_{h_k}^{\textit{av}}}u_k^{\textit{cr}}	\hspace*{-0.2mm})$,  \hspace*{-0.3mm}cf. \hspace*{-1mm}\eqref{eq:error-quantity}, 
	\hspace*{-0.3mm}are \hspace*{-0.2mm}plotted \hspace*{-0.2mm}versus \hspace*{-0.2mm}the \hspace*{-0.2mm}number \hspace*{-0.2mm}of \hspace*{-0.2mm}degrees \hspace*{-0.2mm}of \hspace*{-0.2mm}freedom \hspace*{-0.2mm}$N_k$~\hspace*{-0.2mm}in~\hspace*{-0.2mm}a~\hspace*{-0.2mm}\mbox{$\log\hspace*{-0.5mm}\log$\hspace*{-0.2mm}--\hspace*{-0.2mm}plot}. In~it, one observes that mesh adaptivity yields the quasi-optimal
	convergence~rate~$\smash{h_k \!\sim\! N_k^{\smash{-\frac{1}{2}}}\!}$ and that for every $k=0,\dots,19$, the trapezoidal primal-dual a posteriori error~\mbox{estimator} $\widehat{\eta}^2_k(\smash{\mathcal{J}_{h_k}^{\textit{av}}}u_k^{\textit{cr}},z_k^{\textit{rt}})$ \hspace*{-0.1mm}defines \hspace*{-0.1mm}a \hspace*{-0.1mm}reliable \hspace*{-0.1mm}upper \hspace*{-0.1mm}bound \hspace*{-0.1mm}for \hspace*{-0.1mm}the \hspace*{-0.1mm}error \hspace*{-0.1mm}quantity \hspace*{-0.1mm}$\rho^2_I(u,\smash{\mathcal{J}_{h_k}^{\textit{av}}}u_k^{\textit{cr}})$.~\hspace*{-0.1mm}\mbox{Moreover}, on the right-hand side of Figure \ref{fig:pDirichletAveraging}, we displayed the energy curves for
	$I(\smash{\mathcal{J}_{h_k}^{\textit{av}}}u_k^{\textit{cr}})$~and~$D (z_k^{\textit{rt}})$, $k=0,\dots,19$,~resp., whose distance likewise  converges to zero as $N_k \to \infty$.
	The experiments justify to employ $\tilde{u}_{h_k}=\smash{\mathcal{J}_{h_k}^{\textit{av}}}u_k^{\textit{cr}}\in \mathcal{S}^1_D(\mathcal{T}_k)$ instead of ${\tilde{u}_{h_k}=\smash{u_k^{\textit{c}}}\in \mathcal{S}^1_D(\mathcal{T}_k)}$   in Algorithm  \ref{alg:afem}.  Then, only one non-linear  problem  per~iteration~has~to~be~solved~in~(\hyperlink{Solve}{'Solve'}).

	\begin{figure}[H]\vspace*{-0.25cm}
		\hspace*{-0.25cm}\includegraphics[width=14cm]{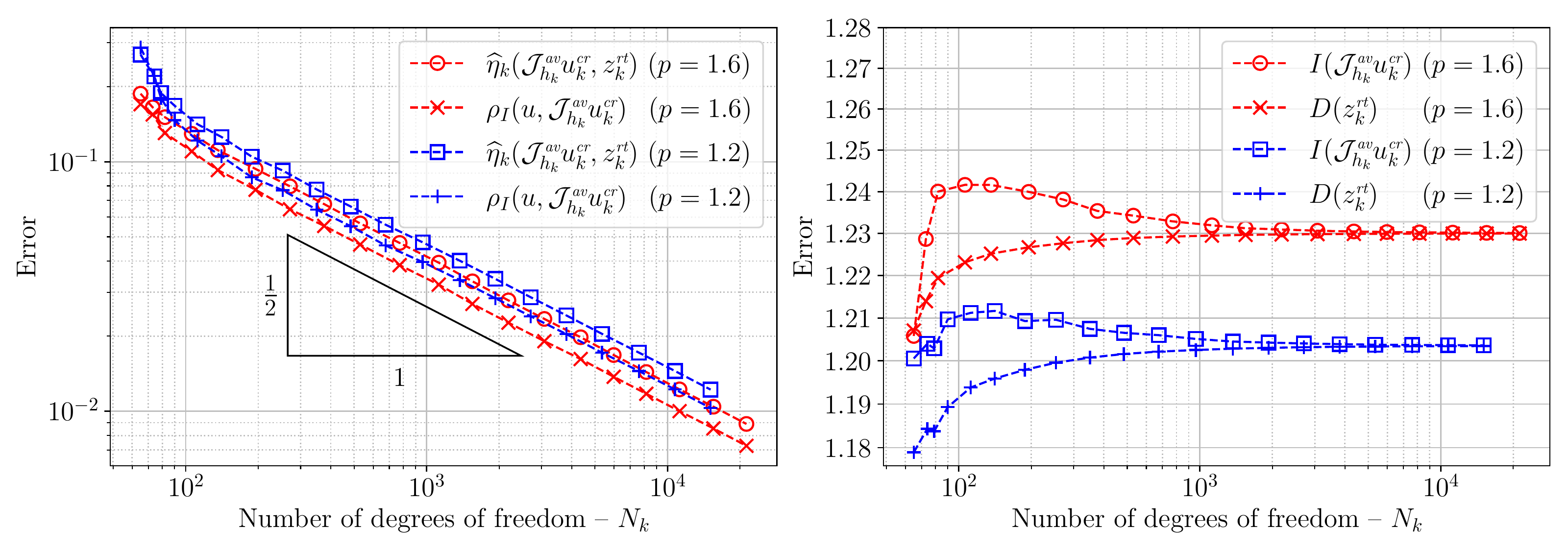}\vspace*{-0.25cm}
		\caption{\hspace*{-0.25mm}The \hspace*{-0.25mm}trapezoidal \hspace*{-0.25mm}primal-dual~\hspace*{-0.25mm}a~\hspace*{-0.25mm}posteriori~\hspace*{-0.25mm}error~\hspace*{-0.25mm}estimators~\hspace*{-0.25mm}$\widehat{\eta}^2_k(\smash{\mathcal{J}_{h_k}^{\textit{av}}}u_k^{\textit{cr}},z_k^{\textit{rt}})$,~\hspace*{-0.25mm}cf.~\hspace*{-0.25mm}\eqref{eq:pd-estimator}, and the error quantities  $\rho^2_I(u,\smash{\mathcal{J}_{h_k}^{\textit{av}}}u_k^{\textit{cr}})$, cf. \eqref{eq:error-quantity}, (left) and the primal~energies~$I(\smash{\mathcal{J}_{h_k}^{\textit{av}}}u_k^{\textit{cr}})$, cf. \eqref{eq:dirichlet_primal}, and the  dual energies $D (z_k^{\textit{rt}})$, cf. \eqref{eq:dirichlet_dual}, (right) for~adaptive~mesh~\mbox{refinement} for $k=0,\dots,19$, $p=1.6$ and $p=1.2$, resp.\vspace*{-0.75cm}}
		\label{fig:pDirichletAveraging}
	\end{figure}

	\subsection{Optimal design problem}\label{subsec:expOptimalDesign}
	
	\qquad \!We \hspace*{-0.2mm}examine \hspace*{-0.2mm}an \hspace*{-0.2mm}optimal \hspace*{-0.2mm}design \hspace*{-0.2mm}problem \hspace*{-0.2mm}with \hspace*{-0.2mm}prescribed \hspace*{-0.2mm}homogeneous~\hspace*{-0.2mm}Dirichlet~\hspace*{-0.2mm}\mbox{boundary} \hspace*{-0.2mm}data \hspace*{-0.2mm}on \hspace*{-0.2mm}an \hspace*{-0.2mm}$L$\hspace*{-0.2mm}--shaped \hspace*{-0.2mm}domain. \hspace*{-0.2mm}More \hspace*{-0.2mm}precisely, \hspace*{-0.2mm}we \hspace*{-0.2mm}let \hspace*{-0.2mm}${\Omega\!\vcentcolon=\!(-1,1)^2\!\setminus\!([0,1]\!\times\! [-1,0])}$,~\hspace*{-0.1mm}${\Gamma_D\!\vcentcolon=\!\partial\Omega}$, $\Gamma_N\vcentcolon=\emptyset$, $\phi\vcentcolon=\psi\circ\vert \cdot\vert C^1(\mathbb{R}^2)$, where $\psi\in C^1(\mathbb{R}_{\ge 0})$ is defined by~\eqref{eq:psi}~for~${\mu_1=1}$,~${\mu_2=2}$, $t_1\hspace*{-0.1em}=\hspace*{-0.1em}\sqrt{2\lambda \mu_1/\mu_2}\hspace*{-0.1em}=\hspace*{-0.1em}\sqrt{\lambda}$, $t_2\hspace*{-0.1em}=\hspace*{-0.1em}\sqrt{2\lambda \mu_2/\mu_1}\hspace*{-0.1em}=\hspace*{-0.1em}2\sqrt{\lambda}$, and $\lambda\hspace*{-0.1em} =\hspace*{-0.1em}0.0145$ as~in~\cite{BC08},~and~${f\hspace*{-0.1em}=\hspace*{-0.1em}1}$. We~use~the same initial triangulation $\mathcal{T}_0$ as in Subsection \ref{subsec:expDirichlet} and exploit~that~${f_k\vcentcolon=\Pi_{h_k}f=f\in \mathcal{L}^0(\mathcal{T}_k)}$ for all $k\hspace*{-0.12em}=\hspace*{-0.12em}0,\dots,19$. \hspace*{-0.2em}Apart from that, for every $k\hspace*{-0.12em}=\hspace*{-0.12em}0,\dots,19$,~we~again~denote~by~${u_k^{\textit{c}}\hspace*{-0.12em} \in\hspace*{-0.12em}  \mathcal{S}^1_D(\mathcal{T}_k)}$ the minimizer of $\smash{I_k^{\textit{c}}\hspace*{-0.12em}\vcentcolon=\hspace*{-0.12em}I|_{ \mathcal{S}^1_D(\mathcal{T}_k)}\hspace*{-0.12em}:\hspace*{-0.12em} \mathcal{S}^1_D(\mathcal{T}_k)\hspace*{-0.12em}\to\hspace*{-0.12em} \mathbb{R}}$ and by $u_k^{\textit{cr}}\hspace*{-0.12em} \in \hspace*{-0.12em} \mathcal{S}^{1,\textit{cr}}_D(\mathcal{T}_k)$~the~minimizer~of~\eqref{eq:dirichlet_discrete_primal}. Both minimizers are computed  resorting to a semi-implicit discretization~of~the respective $L^2$--gradient flows,  cf. \cite[Sec. 5]{Bar21}, with stopping criterion $\varepsilon_{\textit{stop}}=\smash{h}^{\smash{2}}/20$.
	Since these schemes are unconditionally strongly stable, cf. \cite[Prop. 5.2]{Bar21}, we employ~the~fixed~\mbox{step-size}~${\tau=1}$. In Figure \ref{fig:OptimalDesign}, 
  for every $k\hspace*{-0.05em}=\hspace*{-0.05em}0,\dots,19$, $\tilde{u}_{h_k}\hspace*{-0.05em}=\hspace*{-0.05em}u_k^{\textit{c}}\hspace*{-0.05em}\in\hspace*{-0.05em} \smash{\mathcal{S}^{1,\textit{cr}}_D(\mathcal{T}_k)}$ and~a~maximizer~of~${z_k^{\textit{rt}}\hspace*{-0.05em}\in\hspace*{-0.05em} \smash{\mathcal{R}T^0(\mathcal{T}_k)}}$ of \eqref{eq:dirichlet_discrete_dual} obtained  using the reconstruction formula \eqref{eq:dirichlet_discrete_reconstruction}, the square root of the trapezoidal primal-dual a posteriori error estimator~\eqref{eq:pd-estimator}~and~of~${\rho_I^2(u,\tilde{u}_{h_k})=I(\tilde{u}_{h_k})-I(u)}$,~cf.~\eqref{co-coercive3},
	where the exact value $I(u)\approx -0.0745503$ is \mbox{approximated}~\mbox{using}~Aitken's~\mbox{$\delta^2$--process},~cf.~\cite{Ait26}, are plotted versus the number of degrees of freedom $N_k$~in~a~\mbox{$\log\log$--plot}.~In~it, one  observes that mesh adaptivity yields the  quasi-optimal
   convergence~rate~$\smash{h_k \sim N_k^{\smash{-\frac{1}{2}}}}$.~In~\mbox{particular}, for every $k=0,\dots,19$, the trapezoidal primal-dual~a~posteriori~error~\mbox{estimator} $\widehat{\eta}^2_k(u_k^{\textit{c}},z_k^{\textit{rt}})$ defines a reliable upper bound for the error quantity $\rho^2_I(u,u_k^{\textit{c}})$.
	On the right-hand side of Figure~\ref{fig:OptimalDesign}, we displayed the energy curves for
	$I(u_k^{\textit{c}})$~and~$D (z_k^{\textit{rt}})$, $k=0,\dots,19$,~resp., \hspace*{-0.1mm}whose \hspace*{-0.1mm}distance  \hspace*{-0.1mm}converges \hspace*{-0.1mm}to \hspace*{-0.1mm}zero \hspace*{-0.1mm}as \hspace*{-0.1mm}$N_k \hspace*{-0.1em}\to\hspace*{-0.1em} \infty$.
	\hspace*{-0.1mm}In~\hspace*{-0.1mm}\mbox{agreement} \hspace*{-0.1mm}with \hspace*{-0.1mm}experimental~\hspace*{-0.1mm}results~\hspace*{-0.1mm}in~\hspace*{-0.1mm}\cite{CL15},~\hspace*{-0.1mm}our~\hspace*{-0.1mm}\mbox{error} estimator \hspace*{-0.1mm}avoids \hspace*{-0.1mm}a \hspace*{-0.1mm}systematic \hspace*{-0.1mm}reliability-efficiency \hspace*{-0.1mm}gap \hspace*{-0.1mm}that \hspace*{-0.1mm}arises~\hspace*{-0.1mm}in~\hspace*{-0.1mm}\mbox{residual-type}~\hspace*{-0.1mm}estimates.
		\begin{figure}[H]\vspace*{-0.25cm}
		\hspace*{-0.25cm}\includegraphics[width=14cm]{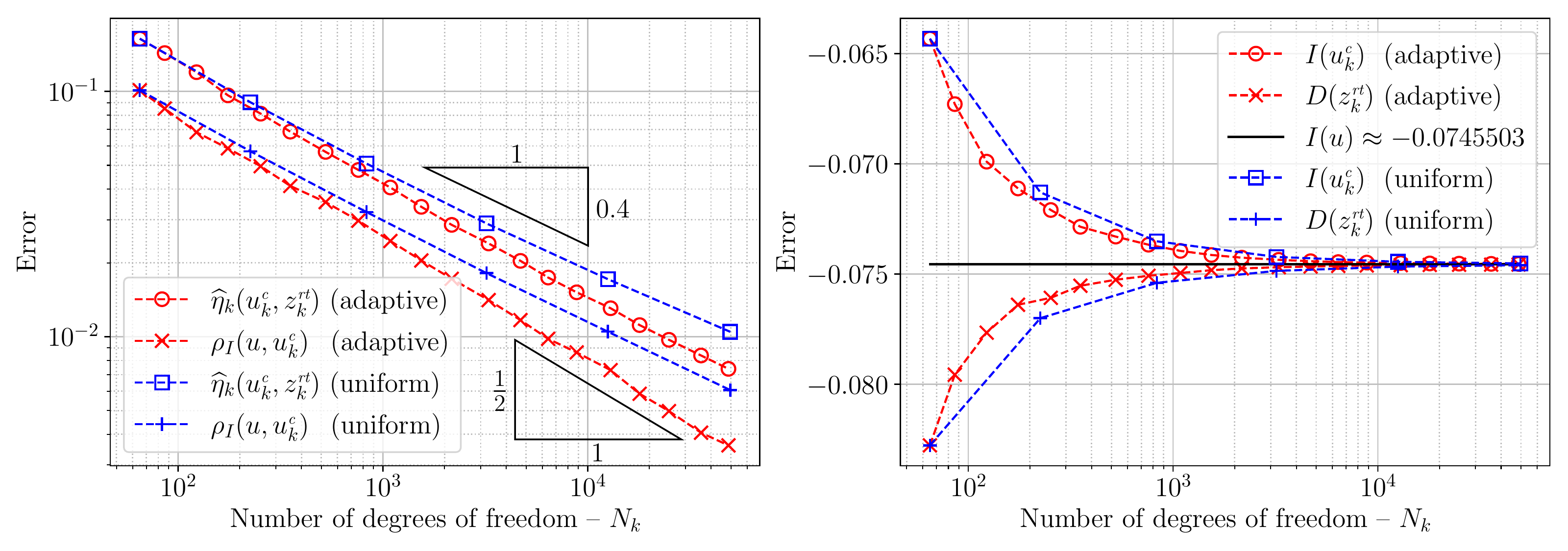}\vspace*{-0.25cm}
		\caption{The trapezoidal primal-dual a posteriori error estimators $\widehat{\eta}^2_k(u_k^{\textit{c}},z_k^{\textit{rt}})$, cf. \eqref{eq:pd-estimator}, and the error quantities  $\rho^2_I(u,u_k^{\textit{c}})$, cf. \eqref{eq:error-quantity}, (left) and the primal energy~$I(u_k^{\textit{c}})$,~cf. \eqref{eq:dirichlet_primal},  and  the dual energy $D (z_k^{\textit{rt}})$, cf. \eqref{eq:dirichlet_dual}, (right) for uniform and adaptive~mesh~refinement for $k=0,\dots,5$ and ${k=1,\dots,19}$, resp.\vspace*{-0.25cm}}
		\label{fig:OptimalDesign}
	\end{figure}
	
	\appendix
	\section{Appendix}\vspace*{-1mm}
	
	\subsection{Convex analysis}\label{subsec:convex_analysis}\vspace*{-1mm}
	
	\qquad For a (real) Banach space $X$ equipped with the norm $\|\cdot\|_X:X\to \mathbb{R}_{\ge 0}$,~we~denote its (continuous) dual space by $X^*$ equipped with the dual norm 
	${\|\cdot\|_{X^*}:X^*\to \mathbb{R}_{\ge 0}}$,~defined~by $\|x^*\|_{X^*}\vcentcolon=\sup_{\|x\|_X\leq 1}{\langle x^*,x\rangle_X}$ for every $x^*\in X^*$, where ${\langle \cdot,\cdot\rangle_X:X^*\times X\to \mathbb{R}}$, defined by $\langle x^*,x\rangle_X\vcentcolon=x^*(x)$ for every $x^*\in X^*$ and $x\in X$, denotes the duality \mbox{pairing}.
	A functional $F\hspace*{-0.17em}:\hspace*{-0.17em}X\hspace*{-0.17em}\to\hspace*{-0.17em} \mathbb{R}\cup\{+\infty\}$ \hspace*{-0.1mm}is \hspace*{-0.1mm}called \hspace*{-0.1mm}sub-differentiable \hspace*{-0.1mm}in \hspace*{-0.1mm}$x\hspace*{-0.17em}\in\hspace*{-0.17em} X$,~\hspace*{-0.1mm}if~\hspace*{-0.1mm}${F(x)\hspace*{-0.17em}<\hspace*{-0.17em}\infty}$~\hspace*{-0.1mm}and~\hspace*{-0.1mm}if~\hspace*{-0.1mm}\mbox{there}~\hspace*{-0.1mm}\mbox{exists}~\hspace*{-0.1mm}${x^*\hspace*{-0.17em}\in\hspace*{-0.17em} X^*}$, called sub-gradient, such that for~every~${y\in X}$,~it~holds\vspace*{-0.5mm}
	\begin{align}
		\langle x^*,y-x\rangle_X\leq F(y)-F(x).\label{eq:subgrad}
	\end{align} 
	The  sub-differential  $\partial F:X\to  2^{X^*}$ of a functional $F:X\to \mathbb{R}\cup\{+\infty\}$~for~every~${x\in X}$, is defined by $(\partial F)(x)\vcentcolon=\{x^*\in X^*\mid \eqref{eq:subgrad}\text{ holds for }x^*\}$ if $F(x)<\infty$ and $(\partial F)(x)\vcentcolon=\emptyset$~else.\vspace*{-8mm}
	
	\newpage
	\hspace*{-5.5mm}For a functional $F:X\to \mathbb{R}\cup\{\pm\infty\}$, we denote its (Fenchel) conjugate functional by $F^*:X^*\to \mathbb{R}\cup\{\pm\infty\}$, which for every $x^*\in X^*$ is defined by ${F^*(x^*)\vcentcolon=\sup_{x\in X}{\langle x^*,x\rangle_X-F(x)}}$. 
	If $F:X\to \mathbb{R}\cup\{+\infty\}$ is a proper, convex and lower semi-continuous~\mbox{functional}, then its conjugate $F^*:X^*\to\mathbb{R}\cup\{+\infty\}$ equally is  proper, convex and lower semi-continuous functional,  cf. \cite[p.  17]{ET99}. 
	Moreover, for every ${x^*\in X^*}$~and~${x\in X}$ such that 
	$ F^*(x^*)+F(x)$ is well-defined, i.e., the critical case $\infty-\infty$ does~not~occur, the Fenchel--Young inequality\vspace*{-0.5mm}
	\begin{align}
		\langle x^*,x\rangle_X\leq F^*(x^*)+F(x)\label{eq:fenchel_young_ineq}
	\end{align}
	applies. 
	In particular, 
	for every $x^*\in X^*$ and $x\in X$, it holds\vspace*{-0.5mm}
	\begin{align}
		x^*\in (\partial F)(x)\quad\Leftrightarrow \quad	\langle x^*,x\rangle_X= F^*(x^*)+F(x).\label{eq:fenchel_young_id}
	\end{align}
	
	\subsection{Estimates for node-averaging operator in terms of (shifted) $N$--functions}\label{subsec:auxiliary}\vspace*{-0.5mm}
	
	\qquad In this subsection, we want to prove several estimates for the node-averaging operator $\mathcal{J}_h^{\textit{av}}\!:\!\mathcal{L}^k(\mathcal{T}_h)\!\to\! \mathcal{S}^k_D(\mathcal{T}_h)$ in terms of (shifted) $N$--functions. A convex function $\varphi\!:\!\mathbb{R}_{\ge 0}\!\to\! \mathbb{R}_{\ge 0}$ is said to be an $N$--function if and only if $\varphi(0)=0$, $\varphi(t)>0$ for all $t>0$, $\lim_{t\to 0}{\varphi(t)/t}=0$, and $\lim_{t\to\infty}{\varphi(t)/t}=\infty$. As a consequence, there exists a right-derivative $\varphi':\mathbb{R}_{\ge 0}\to \mathbb{R}_{>0}$, which is non-decreasing and satisfies $\varphi'(0)\!=\!0$, $\varphi'(t)\!>\!0$ for all $t\!>\!0$,~and~${\lim_{t\to \infty}{\varphi'(t)}\!=\!\infty}$.
	In addition, an $N$--function $\varphi:\mathbb{R}_{\ge 0}\to \mathbb{R}_{\ge 0}$ satisfies the $\Delta_2$--condition  (in short, $\varphi\in \Delta_2$), if and only if there exists a constant $c>0$ such that $\varphi(2t)\leq c\,\varphi(t)$ for all $t\ge 0$. We denote the smallest such constant by $\Delta_2(\varphi)\!>\!0$. 
	We say that an $N$--function~${\varphi\!:\!\mathbb{R}_{\ge 0}\!\to\! \mathbb{R}_{\ge 0}}$~satisfies~the $\nabla_2$--condition (in short, $\varphi\in \nabla_2$), if its Fenchel conjugate $\varphi^*:\mathbb{R}_{\ge 0}\to \mathbb{R}_{\ge 0}$~is~an~$N$--function satisfying the $\Delta_2$--condition. If $\varphi:\mathbb{R}_{\ge 0}\to \mathbb{R}_{\ge 0}$ satisfies the $\Delta_2$-- and~the~\mbox{$\nabla_2$--condition} (in short, $\varphi\in \Delta_2\cap \nabla_2$), then for  $a\ge 0$, we define $\varphi_a':\mathbb{R}_{\ge 0}\to \mathbb{R}_{\ge 0}$~by $\smash{\varphi_a'(t)\vcentcolon=\varphi'(a+t)\frac{t}{a+t}}$~for all $t\ge 0$. Furthermore, for  $a\ge 0$,  we define $\varphi_a:\mathbb{R}_{\ge 0}\to \mathbb{R}_{\ge 0}$,  called~shifted~\mbox{$N$--functions},~by $\varphi_a(t)\vcentcolon=\smash{\int_0^t{\varphi_a'(s)\,\textup{d}s}}$  for all $t\ge 0$. It holds ${c_{\varphi}\vcentcolon=\sup_{a\ge 0}{\Delta_2(\varphi_a)}<\infty}$, cf.~\cite[Lemma~22]{DK08}.	In particular,  for every $\varepsilon\!>\!0$, there exists a constant 
	$c_\varepsilon\!>\!0$, not depending on $a\!\ge\! 0$,~such~that for every $t,s\ge 0$ and $a\ge 0$, it holds\vspace*{-1mm}
	\begin{align}
		 s\, t\leq  c_\varepsilon\,(\varphi_a)^*(s)+\varepsilon\,\varphi_a(t).\label{eq:eps-young}
	\end{align}
	
	\begin{proposition}\label{prop:n-function}
		Let $\varphi:\mathbb{R}_{\ge 0}\to \mathbb{R}_{\ge 0}$ be an $N$--function such that $\varphi\in\Delta_2\cap \nabla_2$. Then,  for every $v_h\in \mathcal{L}^k(\mathcal{T}_h)^l$, $k,l\in \mathbb{N}$, $m\in \{1,\dots,k+1\}$, $a\ge 0$ and $T\in \mathcal{T}_h$, we have that\vspace*{-0.5mm}
		\begin{align*}
			\fint_T{\varphi_a\big(h_T^m\vert \nabla_h^m(v_h-\mathcal{J}_h^{\textit{av}}v_h)\vert\big)\,\textup{d}x}\leq  c_{\textit{av}}\sum_{S\in \mathcal{S}_h(T)}{\fint_S{\varphi_a(\vert\jump{v_h}_S\vert)\,\textup{d}s}},
		\end{align*}\vspace*{-2.5mm}
	
		\hspace*{-5.5mm}where $c_{\textit{av}}>0$ only depends on $k,l\in \mathbb{N}$, $\smash{c_{\varphi}}>0$ and a constant $c_{\mathcal{T}}>0$ that depends on 
		geometry~of the triangulations $\mathcal{T}_h$, $h>0$, but not on their maximal, minimal~or~mean~\mbox{mesh-sizes}.\vspace*{-0.5mm}
	\end{proposition}

	\begin{proof}\let\qed\relax
		Appealing to \cite[Lemma 22.12]{EG21}, there exists a constant $\overline c_{\textit{av}}>0$, such that~for~every $v_h\in \mathcal{L}^k(\mathcal{T}_h)^l$ and
     ${T\in \mathcal{T}_h}$, we have that\vspace*{-0.5mm}
		\begin{align}
			h_T^m\| \nabla_h^m(v_h-\mathcal{J}_h^{\textit{av}}v_h)\|_{L^\infty(T;\mathbb{R}^{l\times d^m})}\leq \overline c_{\textit{av}}\sum_{\smash{S\in \mathcal{S}_h(T)}}{\|\jump{v_h}_S\|_{L^\infty(S;\mathbb{R}^l)}}.\label{prop:n-function1}
		\end{align}\vspace*{-2.5mm}
	
		\hspace*{-5.5mm}Hence, since $\|\jump{v_h}_S\|_{L^\infty(S;\mathbb{R}^l)}\leq c_{\mathcal{T}}\fint_S{\vert \jump{v_h}_S\vert\,\textup{d}s}$ (cf. \cite[Lemma 12.1]{EG21}), also using~the~$\Delta_2$--condition and
		convexity of ${\varphi_a:\mathbb{R}_{\ge 0}\to \mathbb{R}_{\ge 0}}$,~${a\ge 0}$, in particular, Jensen's inequality, and that $\sup_{h>0}{\sup_{T\in \mathcal{T}_h}{\textup{card}(\mathcal{S}_h(T))}}\leq c_{\mathcal{T}}$ in \eqref{prop:n-function1},	for every $T\in \mathcal{T}_h$, we find that\vspace*{-0.5mm}
		\begin{align*}
		\fint_T{\varphi_a\big(h_T^m\vert \nabla_h^m(v_h-\mathcal{J}_h^{\textit{av}}v_h)\vert\big)\,\textup{d}x}
			&\leq 
		\Delta_2(\varphi_a)^{\lceil \overline c_{\textit{av}}c_{\mathcal{T}}^2\rceil }
			\varphi_a\bigg(\frac{1}{\textup{card}(\mathcal{S}_h(T))}\sum_{S\in\mathcal{S}_h(T)}{\fint_S{\vert \jump{v_h}_S\vert\,\textup{d}s}}\bigg)
		\\[-0.5mm]&\leq  c_{\varphi}^{\lceil \overline c_{\textit{av}}c_{\mathcal{T}}^2\rceil }\sum_{S\in\mathcal{S}_h(T)}{\fint_S{\varphi_a(\vert \jump{v_h}_S\vert)\,\textup{d}s}}.\tag*{$\square$}
		\end{align*}
		\end{proof}
	
		\begin{corollary}\label{cor:n-function}
			Let $\varphi:\mathbb{R}_{\ge 0}\to \mathbb{R}_{\ge 0}$ be an $N$--function  such that $\varphi\in\Delta_2\cap \nabla_2$. 
			Then,  for every $v_h\in  \mathcal{S}^{1,\textit{cr}}(\mathcal{T}_h)$, $m\in\{0,1,2\}$, $a\ge 0$ and $T\in \mathcal{T}_h$, we have that\vspace*{-0.5mm}
		\begin{align*}
			\fint_T{\varphi_a(h_T^m\vert \nabla_h^m(v_h-\mathcal{J}_h^{\textit{av}}v_h)\vert)\,\textup{d}x}&\leq  c_{\textit{av}}\sum_{S\in\mathcal{S}_h(T)}{\fint_S{\varphi_a(h_S\vert\jump{\nabla_h v_h}_S\vert)\,\textup{d}s}}\\[-1mm]&\leq 
			 \tilde{c}_{\textit{av}}\fint_{\omega_T}{\varphi_a(h_T\vert\nabla_h v_h\vert)\,\textup{d}x}.
		\end{align*}
		where $\tilde{c}_{\textit{av}}>0$ only depends on $c_{\varphi}>0$ and a constant $c_{\mathcal{T}}>0$ that depends on 
		geometry of the triangulations $\mathcal{T}_h$, $h>0$, but not on their maximal, minimal or mean mesh-sizes.\vspace*{-0.5mm}
		\end{corollary}
		
		\begin{proof}
			Follows from Proposition \ref{prop:n-function}, if we exploit that $\jump{v_h}_S=\jump{\nabla_h v_h}_S\cdot(\textup{id}_{\mathbb{R}^d}-x_S)$~on~$S$ for all $S\in \mathcal{S}_h$ and $v_h\in \mathcal{S}^{1,\textit{cr}}(\mathcal{T}_h)$ and the discrete trace inequality~\cite[Lemma~12.8]{EG21}.\vspace*{-0.5mm}
		\end{proof}

	{\setlength{\bibsep}{0pt plus 0.0ex}\small

	
	\providecommand{\bysame}{\leavevmode\hbox to3em{\hrulefill}\thinspace}
	\providecommand{\noopsort}[1]{}
	\providecommand{\mr}[1]{\href{http://www.ams.org/mathscinet-getitem?mr=#1}{MR~#1}}
	\providecommand{\zbl}[1]{\href{http://www.zentralblatt-math.org/zmath/en/search/?q=an:#1}{Zbl~#1}}
	\providecommand{\jfm}[1]{\href{http://www.emis.de/cgi-bin/JFM-item?#1}{JFM~#1}}
	\providecommand{\arxiv}[1]{\href{http://www.arxiv.org/abs/#1}{arXiv~#1}}
	\providecommand{\doi}[1]{\url{http://dx.doi.org/#1}}
	\providecommand{\MR}{\relax\ifhmode\unskip\space\fi MR }
	\providecommand{\MRhref}[2]{%
		\href{http://www.ams.org/mathscinet-getitem?mr=#1}{#2}
	}
	\providecommand{\href}[2]{#2}

}
	
\end{document}